\documentclass[11pt]{article}
\usepackage[margin=1in,left=1in]{geometry}
\usepackage{amsmath}
\usepackage{mathpazo}
\usepackage{epigraph}
\usepackage{amsfonts}
\usepackage{amssymb}
\usepackage{amsthm}
\usepackage{eucal}
\usepackage{wasysym}
\usepackage{booktabs}
\usepackage{enumitem}
\usepackage{mathtools}
\usepackage{caption}
\usepackage{graphicx}
\usepackage{color}
\usepackage{url}
\usepackage[colorlinks]{hyperref}
\usepackage{stmaryrd}
\usepackage{tikz-cd}

\usepackage[utf8]{inputenc} 
\usepackage[T1]{fontenc}

\numberwithin{section}{part}

\definecolor{aqua}{rgb}{0, 1.0, 1.0}
\definecolor{fuschia}{rgb}{1.0, 0, 1.0}
\definecolor{gray}{rgb}{0.502, 0.502, 0.502}
\definecolor{lime}{rgb}{0, 1.0, 0}
\definecolor{maroon}{rgb}{0.502, 0, 0}
\definecolor{navy}{rgb}{0, 0, 0.502}
\definecolor{olive}{rgb}{0.502, 0.502, 0}
\definecolor{purple}{rgb}{0.502, 0, 0.502}
\definecolor{silver}{rgb}{0.753, 0.753, 0.753}
\definecolor{teal}{rgb}{0, 0.502, 0.502}
\definecolor{cite}{RGB}{143, 113, 218}
\definecolor{url}{RGB}{218, 113, 136}

\theoremstyle{plain}
\newtheorem{theorem}{Theorem}[section]
\newtheorem{lemma}[theorem]{Lemma}

\newtheorem{corollary}[theorem]{Corollary}

\theoremstyle{definition}
\newtheorem{definition}[theorem]{Definition}
\newtheorem{example}[theorem]{Example}

\newtheorem{remark}[theorem]{Remark}

\newtheorem{construction}[theorem]{Construction}

\newcommand{\dslash}{/\!\!/}

\renewcommand{\owedge}{\varowedge}
\newcommand{\esmash}{\!\vartriangle\!}

\providecommand{\keywords}[1]{{\small{\textit{Keywords:}} #1}}

\numberwithin{equation}{section}

\hypersetup{
  citecolor =  cite,
  urlcolor = url,
  linkcolor = red
  }

\begin{document}

%

\title{Combinatorial parametrised spectra}
\author{V.~Braunack-Mayer\footnote{Universit\"{a}t Hamburg, Bundesstra{\ss}e 55, 20146 Hamburg, Germany.\newline\indent\indent
\emph{Email address:} \href{mailto://v.braunackmayer@gmail.com}{\tt v.braunackmayer@gmail.com}}}
\date{}
\maketitle

\begin{abstract}
We obtain combinatorial model categories of parametrised spectra, together with systems of base change Quillen adjunctions associated to maps of parameter spaces.
We work with simplicial objects and use Hovey's sequential and symmetric stabilisation machines.
By means of a Grothendieck construction for model categories, we produce combinatorial model categories controlling the totality of parametrised stable homotopy theory.
The global model category of parametrised symmetric spectra is equipped with a symmetric monoidal model structure (the external smash product) inducing pairings in twisted cohomology groups.

As an application of our results we prove a tangent prolongation of Simpson's theorem, characterising tangent $\infty$-categories of presentable $\infty$-categories as accessible localisations of $\infty$-categories of presheaves of parametrised spectra.
Applying these results to the homotopy theory of smooth $\infty$-stacks produces well-behaved (symmetric monoidal) model categories of smooth parametrised spectra.
These models, which subsume previous work of Bunke and Nikolaus, provide a concrete foundation for studying twisted differential cohomology.
\end{abstract}

{\small
\tableofcontents
}

\section*{Introduction}
\addcontentsline{toc}{section}{Introduction}
Parametrised homotopy theory is the study of continuously varying families of spaces or spectra parametrised as fibres over a base space.
Parametrised objects in homotopy theory  encode a subtle mixture of both the homotopy types of the base and fibre and provide a robust general framework for studying homotopical group actions.
The modern perspective on parametrised homotopy is based on an analogy with algebraic geometry, where parametrised homotopy types are likened to derived categories of sheaves over a scheme.
In both contexts, associated to morphisms of the base (space or scheme) there are triples of adjunctions facilitating transport of information between different parametrised contexts.
These base change adjunctions are the most important feature  of the parametrised theory, playing a key role in the study of duality phenomena and providing a link between the parametrised and non-parametrised settings.

Parametrised homotopy theory is implicit throughout much of algebraic topology.
Working with bundles, fibrations or actions in homotopy theory is tantamount to studying parametrised spaces or spectra.
Several classical results in algebraic topology, such as the Eilenberg--Moore and Atiyah--Hirzebruch spectral sequences, are fundamentally statements about parametrised homotopy types.
Working with parametrised homotopy types also sheds new light on classical results in algebraic topology; providing a natural and intuitive setting for studying orientations, duality, and twisted umkehr maps.
Parametrised spectra have also recently enjoyed increasing importance in mathematical physics, being the classifying objects for twisted homology and cohomology theories.

The first systematic approach to parametrised spectra and their base change adjunctions using the tools of homotopical algebra was set down in \cite{may_parametrized_2006}.
Working with point-set topological spaces, May and Sigurdsson devote a great deal of effort in overcoming numerous subtleties and technical difficulties that arise in stable homotopy theory in the parametrised setting.
Their theory is rich and subtle, blending classical homotopy theory with model categorical methods.
Working with model categories of parametrised orthogonal spectra, May and Sigurdsson define fibrewise smash products and base change adjunctions, and show that these data organise into a Wirthm\"{u}ller context, specialising Grothendieck's six-functor formalism.

More recently, Ando, Blumberg, and Gepner have introduced a general theory of parametrised objects in the setting of $\infty$-categories \cite{ando_parametrized_2018}.
Their approach works in the generality of objects of a presentable $\infty$-category $\mathcal{C}$ parametrised by objects of an $\infty$-topos $\mathcal{X}$, and specialises to ordinary  parametrised stable homotopy theory when $\mathcal{C}=\mathrm{Sp}(\mathcal{S})$ and $\mathcal{X}=\mathcal{S}$ are the $\infty$-categories of spectra and spaces respectively.
The basic idea is to regard parametrised objects as local coefficient systems; $\mathcal{C}$-valued families parametrised by $S\in \mathcal{X}$ are $\mathcal{C}$-valued sheaves on $\mathcal{X}_{/S}$, that is limit-preserving functors
$
(\mathcal{X}_{/S})^\mathrm{op}\to \mathcal{C}
$.
For a map of parameter objects $f\colon S\to T$, the corresponding base change adjunctions are obtained as pullbacks and Kan extensions.
This general approach to parametrised homotopy theory is certainly very elegant, with the $\infty$-categorical machinery upon which it is based elegantly handling all homotopy coherence issues.

Taken together, the works \cite{may_parametrized_2006, ando_parametrized_2018} provide a useful but incomplete suite of tools for studying of parametrised homotopy theory.
In this article, we expand the toolbox by producing new combinatorial model categories presenting the homotopy theory of parametrised spectra.
By working with simplicial objects we are able to strike a balance between the abstract elegance of the $\infty$-categorical theory and the more technically involved but more down-to-earth topological approach.
The result is a pared-down version of parametrised stable homotopy theory, which has both its advantages and disadvantages; for instance we do not need to work as hard as May and Sigurdsson to establish the theory, yet important features such as fibrewise smash products are unclear in our models. 
On the other hand, the minimality afforded by the combinatorial models plays a key role in applications (for example in the rectification of parametrised rational stable homotopy types---see \cite{braunack-mayer_strict_2019, braunack-mayer_strict_2019-1}), and our methods shed new light on tangent $\infty$-categories and their external smash product structures.

Working with simplicial objects, we use Hovey's sequential and symmetric stabilisation machines (see \cite{hovey_spectra_2001}) to produce model categories of spectra parametrised over any simplicial set $X$.
These model categories are left proper, combinatorial and simplicial, so are amenable to Bousfield localisation techniques. 
For any map of simplicial sets $f\colon X\to X'$, we exhibit base change Quillen adjunctions relating model categories of spectra parametrised over $X$ and $X'$ (Theorems \ref{thm:SeqSpecStructureTheorem} and \ref{thm:SymSpecStructureTheorem}).

The heart of this article is the construction of global model categories of parametrised spectra.
In contrast to \emph{local} model categories of spectra parametrised over a fixed base space $X$, \emph{global} model categories model the totality of parametrised stable homotopy theory by allowing the parameter space to vary.
The terms local and global are used in loose analogy with petit and gros topoi (and should not be confused with homotopical or categorical localisations, nor with global equivariant homotopy theory!).
Using a Grothendieck construction for model categories (as worked out in \cite{harpaz_grothendieck_2015}) we construct global model categories of parametrised spectra using both sequential and symmetric spectrum objects (Theorem \ref{thm:GlobParamSpec}).
The result is a pair of model fibrations
\[
\mathrm{Sp}^\mathbb{N}_\mathrm{sSet}\longrightarrow \mathrm{sSet}
\;\;\mbox{ and }\;\;
\mathrm{Sp}^\Sigma_\mathrm{sSet}\longrightarrow \mathrm{sSet}
\]
projecting to the Kan model structure on simplicial sets; in each case the fibre model category over the simplicial set $X$ presents  $X$-parametrised stable homotopy theory. 
The symmetric and sequential global model categories are Quillen equivalent and we show that they are both combinatorial, left proper and simplicial (Theorems \ref{thm:SeqSpecGlobalComb} and \ref{thm:ExternalSmash}).
Furthermore, the symmetric global model category is equipped with a closed symmetric monoidal model structure presenting the external smash product of parametrised spectra (Theorem \ref{thm:ExternalSmash} and Lemma \ref{lem:ESmash=FSmash}).
Analogous to the external derived tensor product of sheaves, the external smash product 
\[
\vartriangle\colon \mathrm{Sp}(\mathcal{S})_{/X}\times \mathrm{Sp}(\mathcal{S})_{/X'}
\longrightarrow  \mathrm{Sp}(\mathcal{S})_{/X\times X'}
\]
sends parametrised spectra over $X$ and $X'$ to their fibrewise smash product parametrised over the product $X\times X'$.
The external smash product induces pairings on twisted cohomology groups (see \cite[Section 6]{ando_parametrized_2018} and Remark \ref{rem:TwistDiffPairings}).

As an application of our combinatorial models, we prove an extension of Simpson's theorem on presentable $\infty$-categories to tangent $\infty$-bundles.
The tangent $\infty$-bundle of a presentable $\infty$-category $\mathcal{C}$ is a categorical fibration $p\colon T\mathcal{C}\to \mathcal{C}$ whose fibre over the object $c\in \mathcal{C}$ is the stable $\infty$-category $\mathrm{Sp}(\mathcal{C}_{/c})$ \cite[Section 7.3]{lurie_higher_2017}. 
Objects of the fibre $T_c\mathcal{C} =\mathrm{Sp}(\mathcal{C}_{/c})$ are naturally interpreted as $c$-parametrised families of linear objects in the sense of Goodwillie calculus (see Remark \ref{rem:Goo}), hence $T\mathcal{C}$ encodes the parametrised stable homotopy theory intrinsic to $\mathcal{C}$.
The fundamental example is the categorical fibration $p\colon T\mathcal{S}\to \mathcal{S}$, with $T\mathcal{S}$ the $\infty$-category of parametrised spectra and $p$ projecting to parameter spaces.

We prove that the global model categories $\mathrm{Sp}^\mathbb{N}_\mathrm{sSet}$ and $\mathrm{Sp}^\Sigma_\mathrm{sSet}$ are both presentations of the $\infty$-category $T\mathcal{S}$ (Lemma \ref{lem:PresentTS}).
Using this fact, we obtain a general structural characterisation of tangent $\infty$-bundles as accessible localisations of $\infty$-categories of $T\mathcal{S}$-valued presheaves (Theorem \ref{thm:Tangent}).
This result related to the perspective taken in \cite{ando_parametrized_2018}, though our methods are model- rather than $\infty$-categorical.
The proof of Theorem \ref{thm:Tangent} involves a new construction of the tangent model category to a combinatorial simplicial model category, giving an alternative approach to \cite{harpaz_tangent_2018}.

In the final part of the article we further specialise our results in order to obtain combinatorial models for smooth parametrised spectra.
Smooth parametrised spectra are classifying objects for twisted differential cohomology theories and as such are poised to play an increasingly important role in mathematical physics.
For instance, it is widely believed that twisted differential $K$-theory controls the field content and charge quantisation of $D$-branes in string theory.
A key result of this article is Theorem \ref{thm:SmoothTangentooTopos}, which solves the open problem of providing a concrete foundation for twisted differential cohomology.
We explain how twisted differential cohomology theories are naturally represented by sheaves of parametrised spectra on the site of Cartesian spaces satisfying a stabilised \v{C}ech descent condition and show how this perspective subsumes the earlier work of Bunke and Nikolaus \cite{bunke_twisted_2014} (see Example \ref{ex:SmoothSpecLineBun}).
Moreover, we prove that smooth parametrised spectra organise into a symmetric monoidal $\infty$-category by exhibiting a symmetric monoidal model structure on the model category of \v{C}ech-local presheaves of parametrised symmetric spectra on the site of Cartesian spaces (Theorem \ref{thm:SmoothESmash}).
The external smash product of smooth parametrised spectra that we obtain induces pairings on twisted differential cohomology groups (Remark \ref{rem:TwistDiffPairings}), similarly to the case of ordinary parametrised spectra.

\paragraph*{Organisation.}
We begin in Section \ref{S:Ret} by studying retractive spaces, the main characters in unstable parametrised homotopy theory.
A retractive space over $X$ is a space over and under $X$, equipped with a projection and a section.
By taking homotopy fibres in a suitably coherent manner, retractive spaces are equivalent to local coefficient systems of pointed spaces.
Beginning with more or less well-known material, we quickly develop simplicial model categories of retractive spaces and their base change adjunctions.
Using the Grothendieck construction for model categories, we glue the various fibre categories together to obtain a global model category of retractive spaces.
This global model category is left proper and combinatorial, and becomes a symmetric monoidal model category with respect to the external smash product.
We give a brief exposition of unstable Koszul (pre-)duality between comodules and pointed modules over the loop space based on \cite{hess_waldhausen_2016}.

Section \ref{S:ParamSpec} is the core of the article. 
In this section we pass from unstable to stable parametrised homotopy by means of Hovey's sequential and symmetric stabilisation machines.
We derive and analyse base change adjunctions between combinatorial model categories of parametrised spectra, and prove that our model categories of parametrised spectra are pseudonaturally Quillen equivalent to each other and present equivalent models to  \cite{ando_parametrized_2018,may_parametrized_2006}. 
By means of the Grothendieck construction, we construct two Quillen equivalent models for the global homotopy theory of parametrised spectra.
Both of these model categories are left proper, combinatorial, and simplicial. 
Finally, we show that parametrised symmetric spectra organise into a symmetric monoidal model category with respect to the external smash product.

The last part of the article, Section \ref{S:CombTangent}, is concerned with applications to $\infty$-categories and twisted differential cohomology.
We prove a tangent prolongation of Simpson's theorem on presentable $\infty$-categories,  characterising tangent $\infty$-bundles as accessible localisations of $\infty$-categories of presheaves of parametrised spectra.
As a special case of this result, we derive a presentation for the homotopy theory of smooth parametrised spectra in terms of \v{C}ech-local presheaves on smooth Cartesian spaces valued in parametrised spectra.
Building upon this result, we also exhibit a symmetric monoidal external smash product for smooth parametrised spectra at the level of model categories. We conclude the article with some general remarks on descent spectral sequences and pairings in twisted differential cohomology.

\paragraph*{Remark.} The work \cite{hebestreit_multiplicative_2019} appeared while the present article was being completed.
Working with symmetric spectra parametrized over $\mathcal{I}$-spaces, Hebestreit, Sagave and Schlichtkrull produce (in our terminology) both local and global model categories in order to capture multiplicative properties of twisted cohomology at the point set level.
In the case that the parameter $\mathcal{I}$-space is a commutative monoid object, the corresponding category of parametrised symmetric spectra is shown to be symmetric monoidal, as is the global model category.
Working with something like $\mathcal{I}$-spaces is necessary to capture general multiplicative phenomena; since the unit component of a strict topological monoid is necessarily a generalised Eilenberg--Mac Lane space, not all parametrised $\mathbb{E}_\infty$-ring spectra can be modelled by strict commutative monoid objects.
Restricting the models of \cite{hebestreit_multiplicative_2019} to constant $\mathcal{I}$-spaces specialises to the model categories obtained in Sections \ref{SS:SymStabLoc} and \ref{SS:SymStabGlob}.
Note that our global symmetric model structure is in fact equivalent, as a symmetric monoidal model category, to the global model category defined using $\mathcal{I}$-spaces as both give model categorical presentations of the symmetric monoidal $\infty$-category $T\mathcal{S}^\otimes$ of parametrised spectra.

\paragraph*{Notation and conventions.}
We have attempted to make this article reasonably self-contained without, we hope, being overly verbose.
Some clarifying remarks with regards to notation and conventions are in order:
\begin{itemize}
  \item At various stages throughout the article we use different symbols (e.g.~$\otimes$, $\otimes_X$, $\owedge_X$) to distinguish between various tensoring bifunctors. 
  Which bifunctor is meant should always be either explicitly stated or else clear from context.
  
  \item We write $\wedge_X$ for the fibrewise smash product of $X$-parametrised objects (either retractive spaces or symmetric spectra over $X$).
  The external smash product is always denoted $\vartriangle$ to distinguish it from the ordinary smash product $\wedge$ (the intention is to remind the reader that there is something non-trivial happening on base spaces).
  
  \item Suspension and looping functors for parametrised symmetric spectra are written in boldface (e.g.~$\mathbf{\Sigma}_X$) to distinguish from the corresponding functors for parametrised sequential spectra (e.g.~$\Sigma_X$).
  
  \item When unadorned with a subscript, the boldface symbol $\mathbf{\Sigma}$ refers to the symmetric groupoid with objects $n\geq 0$ and morphisms
  \[
  \mathbf{\Sigma} (n,m) =
  \begin{cases}
  \Sigma_n & \mbox{ if $n=m$}\\
  \varnothing &\mbox{ otherwise}\,.
  \end{cases}
  \]
$\mathbf{\Sigma}$ is symmetric monoidal via the assignment $(n,m)\mapsto n+m$ on objects, extended to morphisms by the canonical homomorphisms $\Sigma_n\times \Sigma_m\hookrightarrow \Sigma_{n+m}$ of permutation groups.
  
  \item For the theory of locally presentable and accessible categories, we refer to  \cite{adamek_locally_1994} and frequently also \cite[Appendix A]{lurie_higher_2009}.
  The general theory of $\infty$-categories, their presentations and stabilisations is drawn from \cite{lurie_higher_2009,lurie_higher_2017}.
  In an attempt to delineate between the 1- and $\infty$-categorical contexts, we refer to \emph{locally presentable} categories as opposed to \emph{presentable} $\infty$-categories. 
  Our primary reference for the theory of (simplicial, monoidal) model categories is \cite{hovey_model_1999}, though we frequently draw on \cite[Appendix A]{lurie_higher_2009} and \cite{hirschhorn_model_2003} as well.
  
  \item We make use of the Grothendieck construction for model categories as explained in  \cite{harpaz_grothendieck_2015}, using the terminology of that article. 
  In particular, we write $\mathbf{Adj}$ for the $(2,1)$-category whose objects are categories, 1-morphisms are left adjoint functors and 2-morphisms are pseudonatural isomorphisms.
  $\mathbf{Model}$ is the sub-$(2,1)$-category on model categories, with 1-morphisms the left Quillen functors and the same 2-morphisms.
 
  \item In order to circumvent the usual set-theoretic difficulties, we always implicitly work within a chosen hierarchy of Grothendieck universes.
  We do not belabour this point, however, and take a generally naive attitude toward size issues throughout. 
  
  \item Manifolds are smooth and paracompact.
\end{itemize}

\paragraph*{Acknowledgements.} 
This article is the first of a series \cite{braunack-mayer_strict_2019,braunack-mayer_strict_2019-1} based on my 2018 PhD thesis written under the supervision of Alberto Cattaneo.
I would like to thank 
Urs Schreiber, Kathryn Hess, and Christoph Schweigert for their useful comments and encouragement at various stages during the writing of this article.
I acknowledge the support of the Deutsche Forschungsgemeinschaft RTG 1670 \lq\lq Mathematics inspired by String Theory and Quantum
Field Theory''.

\section{Retractive spaces}
\label{S:Ret}
The first part of this article is concerned with developing combinatorial models for unstable parametrised homotopy theory.
Working over a fixed base space $X$, we study the homotopy theory of retractive spaces\footnote{Called \lq\lq ex-spaces'' in \cite{may_parametrized_2006}.} over $X$ in terms of diagrams of simplicial sets.
The resulting simplicial model categories are related by various base change adjoint functors whose homotopical properties are studied.
These base change functors allow us to glue the various categories of retractive spaces together to obtain a new global model category of retractive spaces describing the totality of parametrised unstable homotopy theory.
This global model category has several useful technical properties (left properness, combinatoriality) and is a closed symmetric monoidal model category with respect to the external smash product.
We conclude with a discussion of Koszul (pre-)duality following \cite{hess_waldhausen_2016}.

\subsection{Local theory}
For any simplicial set $X$, we consider the category of retractive spaces over $X$:
\[
\mathrm{sSet}_{\dslash X} := \big(\mathrm{sSet}_{/X}\big)^{\mathrm{id}_X /} \cong \big(\mathrm{sSet}^{X/}\big)_{/\mathrm{id}_X}\,.
\]
A \emph{retractive space over $X$} is thus a diagram $X\xrightarrow{i} Y\xrightarrow{r} X$ exhibiting $X$ as a retract; we frequently refer to such objects simply as $Y$, omitting the section $i$ and projection $r$ from the notation.
A morphism of retractive spaces $\psi\colon Y\to Y'$ over $X$ is the data of a commuting diagram
\[
\begin{tikzcd}
X
\ar[d, "i"']
\ar[r, "i'"]
&
Y'
\ar[d, "r'"]
\\
Y
\ar[ur, "\psi"]
\ar[r, "r"]
&
X\,.
\end{tikzcd}
\] 
Forgetting sections determines a functor $u\colon \mathrm{sSet}_{\dslash X}\to \mathrm{sSet}_{/X}$; the forgetful functor $u$ has a left adjoint which freely adjoins a section:
\[
\big(
\!\begin{tikzcd}
Y\ar[r, "p"]& X
\end{tikzcd}
\!
\big)_{+X}:=
\big(
\!
\begin{tikzcd}
X\ar[r]
&
X\displaystyle\coprod Y
\ar[r, "\mathrm{id}_X+ p"]
&
X
\end{tikzcd}
\!
\big)\,.
\]
We write $Y_{+X}$ as shorthand for $(Y\to X)_{+X}$, the choice of projection $p$ being implicitly understood from the context.

The category of simplicial sets is locally presentable, and since locally presentable categories are stable under forming under- and over-categories we have the following
\begin{lemma}
\label{lem:RetLocPres}
For any simplicial set $X$, the category $\mathrm{sSet}_{\dslash X}$ is locally presentable.
\end{lemma}
\begin{remark}
\label{rem:ComputeRetSpaceCoLim}
Local presentability implies that $\mathrm{sSet}_{\dslash X}$ has all small limits and colimits.
Given a small diagram $D\colon \mathcal{I}\to \mathrm{sSet}_{\dslash X}$, write $\mathcal{I}^\triangleleft :=\{\ast\}\star \mathcal{I}$ for the categorical join.
From $D$ we construct an auxiliary $\mathcal{I}^\triangleleft$-diagram of simplicial sets
\[
D^\triangleleft (i)=
\begin{cases}
\;\;X & \mbox{$i= \ast$ (the cone point)}\\
D(i) &\mbox{$i\in \mathcal{I}$}\,,
\end{cases}
\]
where the morphisms $\ast \to i$ are sent to the section maps $i_{D(i)}$ and all other morphisms are as for $D$.
The simplicial set $\mathrm{colim}_{\mathcal{I}^\triangleleft} D^\triangleleft$ is canonically a retractive space over $X$ and 
\[
X \longrightarrow \underset{\mathcal{I}^\triangleleft}{\mathrm{colim}}\, D^\triangleleft\longrightarrow X
\]
exhibits the colimit of the diagram $D$ in $\mathrm{sSet}_{\dslash X}$.
The dual statement applies for computing limits of diagrams in retractive spaces.
\end{remark}

To each map of simplicial sets $f\colon X\to X'$ there is an associated triple of base change adjunctions
\[
(f_!\dashv f^\ast \dashv f_\ast)
\colon
\begin{tikzcd}
\mathrm{sSet}_{\dslash X}
\ar[rr, shift left=2.2ex]
\ar[rr, leftarrow, "\bot"]
\ar[rr, shift left=-2.2ex, "\bot"]
&&
\mathrm{sSet}_{\dslash X'}
\end{tikzcd}
\]
relating the respective categories of retractive spaces.
The pullback functor $f^\ast$ sends a retractive space $Z$ over $X'$ to $f^\ast Z$, which is regarded as a retractive space over $X$ in the obvious manner.
The pushforward functor $f_!$ sends a retractive space $Y$ over $X$ to the object of $\mathrm{sSet}_{\dslash X'}$ determined by the bottom row of the iterated pushout diagram
\[
\begin{tikzcd}
X \ar[r]
\ar[d, "f"']
&
Y
\ar[r]
\ar[d]
&
X
\ar[d, "f"]
\\
X' 
\ar[r]
&
f_! Y
\ar[r]
&
X'\,.
\end{tikzcd}
\]
It is easy to verify that the functors $f_!$ and $f^\ast$ are adjoints.
The right adjoint $f_\ast$ is more subtle.
To see that it exists, observe that stability of colimits in $\mathrm{sSet}$ under pullback and Remark \ref{rem:ComputeRetSpaceCoLim} together imply that $f^\ast$ preserves colimits.
Since $\mathrm{sSet}_{\dslash X}$ and $\mathrm{sSet}_{\dslash X'}$ are locally presentable, $f^\ast$ has a right adjoint $f_\ast$ by the adjoint functor theorem.
\begin{remark}
We write the terminal morphism as $X\colon X\to \ast$, so that for any simplicial set $X$ we have an induced triple $(X_!\dashv X^\ast \dashv X_\ast)$ of adjoint functors between $\mathrm{sSet}_{\dslash X}$ and $\mathrm{sSet}_\ast$.
\end{remark}
\begin{remark}
The functor $f_\ast$ can also be described more or less explicitly as follows.
Given a retractive space $Y$ over $X$, the underlying simplicial set can be expressed as a coend over the category of elements of $X$:
\[
Y\cong \int^{(\sigma \colon \Delta^n\to X)\in \mathrm{el}(X)}
\mathrm{sSet}_{/X}(\sigma, Y)\otimes \Delta^n\,.
\]
This expression encodes how $Y\to X$ may be obtained by gluing together sections over the simplices of $X$.
Given an $n$-simplex $\sigma$ of $X'$, the pullback $f^\ast\sigma\colon f^\ast \Delta^n \to X$ generally fails to be an $n$-simplex of $X$.
Nevertheless, we have
\[
f_\ast Y \cong \int^{(\sigma\colon \Delta^n \to X')\in \mathrm{el}(X')}
\mathrm{sSet}_{/X}(f^\ast\sigma, Y) \otimes \Delta^n
\]
expressing $f_\ast Y$ in terms of sections of $Y\to X$ over the pullbacks $f^\ast\sigma$.
Each of the hom-sets $\mathrm{sSet}_{/X}(f^\ast \sigma, Y)$ in this formula are pointed by the zero map $f^\ast  \sigma\to X \to Y$.
These choices of basepoint are all compatible, so that computing coends provides a section $X'\to f_\ast Y$.
\end{remark}

For any simplicial set $X$, the category of retractive spaces $\mathrm{sSet}_{\dslash X}$ is closed symmetric monoidal with respect to the fibrewise smash product $\wedge_X$.
The \emph{fibrewise smash product of retractive spaces} $Y$ and $Z$ is the colimit $Y\wedge_X Z$ of the diagram of simplicial sets
\begin{equation}
\label{eqn:SmashColim}
\begin{tikzcd}
X
\ar[r]
\ar[d]
&
Z
\ar[dd]
\ar[dr]
\\
Y
\ar[dr]
\ar[rr, crossing over]
&&
Y\times_X Z
\\
& X\,
\end{tikzcd}
\end{equation}
equipped with the canonical structure maps as a retractive space over $X$.
The monoidal unit is $X_{+X}$, while associator, unitor, and symmetry natural isomorphisms obeying the requisite coherence conditions are readily defined.
\begin{remark}
\label{rem:InternalHom}
The existence of internal homs in $\mathrm{sSet}_{\dslash X}$ can be deduced by applying the adjoint functor theorem, noting that pullback-stability of colimits in $\mathrm{sSet}$ implies that the fibrewise smash product bifunctor preserves colimits in each variable.
Explicitly, the internal hom functor is given by the assignment
\[
F_X(Y,-)\colon Z\longmapsto \int^{(\sigma\colon\Delta^n \to X)\in \mathrm{el}(X)} \mathrm{sSet}_{\dslash X}(\sigma_{+X}\wedge_X Y, Z)\otimes \Delta^n\,.
\]
The hom-sets $\mathrm{sSet}_{\dslash X}(\sigma_{+X} \wedge_X Y, Z)$ are canonically and compatibly pointed, so that $F_X(Y,Z)$ is indeed a retractive space over $X$.

For any retractive space $Y$ over $X$ (with section $i$ and projection $r$), the underlying simplicial set of the internal hom $F_X(Y,Z)\in \mathrm{sSet}_{\dslash X}$ can also be computed as the pullback of the cospan
\begin{equation}
\label{eqn:InternalHomCospan}
\begin{tikzcd}
X
\ar[r, "j"]
&
Z
\ar[r, leftarrow]
&
r_\ast r^\ast Z\,.
\end{tikzcd}
\end{equation}
Here, the map $j\colon X\to Z$ is the section of $Z$ and $r_\ast r^\ast Z \to Z$ is the result of whiskering with the $(i^\ast\dashv i_\ast)$-unit: $r_\ast r^\ast Z\to r_\ast i_\ast i^\ast r^\ast Z\cong Z$.
To see this, observe that $r_\ast r^\ast Z$ is the internal hom for the Cartesian closed structure on $\mathrm{sSet}_{/X}$.
Indeed, for any $K\to X$ there are natural isomorphisms
\[
\mathrm{sSet}_{/X}(K, r_\ast r^\ast Z)\cong \mathrm{sSet}_{/Y}(r^\ast K, r^\ast Z) \cong \mathrm{sSet}_{/X}(K\times_X Y, Z)\,.
\]
That the pullback of \eqref{eqn:InternalHomCospan} computes $F_X(Y,Z)$ now follows from the dual statement that $Y\wedge_X Z$ is computed as the pushout of the span $X\leftarrow Y\coprod_X Z\to Y\times_X Z$.
\end{remark}

\begin{lemma}
\label{lem:PBisStonglyClosed}
For any map of simplicial sets $f\colon X\to X'$, the pullback functor $f^\ast \colon \mathrm{sSet}_{\dslash X'}\to \mathrm{sSet}_{\dslash X}$ is strongly closed symmetric monoidal.
\end{lemma}
\begin{proof}
There is an evident isomorphism $f^\ast (X'_{+X'}) \cong X_{+X}$ and, since fibrewise smash products are defined in terms of limits and colimits which are stable under pullback in $\mathrm{sSet}$, there are natural isomorphisms $f^\ast(A\wedge_{X'} B) \cong f^\ast A \wedge_X f^\ast B$.
Using these isomorphisms, one shows that $f^\ast$ is strongly symmetric monoidal and hence also closed.

To see that $f^\ast$ is strongly closed, we use the characterisation of internal homs given in Remark \ref{rem:InternalHom}.
Fix a retractive space $Y\in \mathrm{sSet}_{\dslash X'}$ with projection $r\colon Y\to X'$.
Writing $\overline{r}\colon f^\ast Y\to X$ for the pullback projection, we have a natural isomorphism $f^\ast(r_\ast r^\ast K) \cong \overline{r}_\ast\overline{r}^\ast  (f^\ast K)$ for any $K\to X'$.
When $X' =\ast$ this is a consequence of the fact that $X\times(-)\colon \mathrm{sSet}\to \mathrm{sSet}_{/X}$ preserves exponential objects, as is easily checked directly. More generally, there is an isomorphism of categories between $\mathrm{sSet}_{/X}$ and $(\mathrm{sSet}_{/X'})_{/f\colon X\to X'}$ identifying $f^\ast$ with the product functor
\[
f\times(-) \colon \mathrm{sSet}_{/X'}\longrightarrow \big(\mathrm{sSet}_{/X'}\big)_{/(f\colon X\to X')}\,,
\]
which also preserves exponential objects.

For any $Z\in\mathrm{sSet}_{\dslash X'}$, the internal hom $F_{X'}(Y,Z)$ is the pullback of the cospan $X'\to Z\leftarrow r_\ast r^\ast Y$.
Applying $f^\ast$ sends $F_{X'}(Y,Z)$ to the pullback of the cospan $X\to f^\ast Z\leftarrow \overline{r}_\ast \overline{r}^\ast (f^\ast Z)$, hence to $F_{X}(f^\ast Y, f^\ast Z)$.
\end{proof}
\begin{corollary}
The projection formula
\[
f_! \big(A \wedge_X f^\ast B\big) \cong f_! A\wedge_{X'} B
\]
holds for any map of simplicial sets $f\colon X\to X'$.
\end{corollary}
\begin{proof}
This is a formal consequence of the adjunction $(f_!\dashv f^\ast)$ and the fact that $f^\ast$ is strongly closed.
\end{proof}
\begin{corollary}
\label{cor:sSetastEnrichmentFacts}
For any simplicial set $X$, the category $\mathrm{sSet}_{\dslash X}$ is enriched, tensored and cotensored over $(\mathrm{sSet}_\ast, \wedge)$.
Moreover, for any map of simplicial sets $f\colon X\to X'$, the induced base change functor $f_!$ preserves $\mathrm{sSet}_\ast$-tensors and $f^\ast$ preserves both $\mathrm{sSet}_\ast$-tensors and cotensors.
\end{corollary}
\begin{proof}
For retractive spaces $Y$ and $Z$ over $X$, the $\mathrm{sSet}_\ast$-enriched hom-space is $X_\ast F_X(Y,Z)$.
The $\mathrm{sSet}_\ast$-composition morphism $X_\ast F_X (Z,W)\wedge X_\ast F_X(Y,Z)\to  X_\ast F_X(Y,W)$
is the $(X^\ast\dashv X_\ast)$-adjunct of
\[
X^\ast X_\ast F_X (Z,W)\wedge_X X^\ast X_\ast F_X(Y,Z)\longrightarrow
F_X(Z,W)\wedge_X F_X(Y,Z)
\longrightarrow F_X(Y,W)\,.
\]
Here we have used strong monoidality of $X^\ast$ and the first morphism is the fibrewise smash product of $(X^\ast\dashv X_\ast)$-counits.
The second morphism is itself the adjunct of the map
\[
F_X(Z,W)\wedge_X F_X(Y,Z)\wedge_X Y\longrightarrow F_X(Z,W)\wedge_X Z\longrightarrow W
\]
obtained by iterated evaluation.

The $\mathrm{sSet}_\ast$-tensoring is given by the bifunctor
$(K,Y)\mapsto K\owedge_X Y := (X^\ast K)\wedge_X Y$,
as is readily checked using the adjunction $(X^\ast\dashv X_\ast)$.
For any map of simplicial sets $f\colon X\to X'$, we have
\[
f^\ast \big((X')^\ast K\wedge_{X'} Y\big) \cong 
f^\ast ((X')^\ast K))\wedge_{X} f^\ast Y \cong X^\ast K\wedge_X f^\ast Y 
\]
so that the pullback functor $f^\ast$ preserves $\mathrm{sSet}_\ast$-tensors.
Since $X^\ast K = (X'\circ f)^\ast K \cong f^\ast (X')^\ast K $, the pushforward functor $f_!$ preserves $\mathrm{sSet}_\ast$-tensors by the projection formula.

Finally, the $\mathrm{sSet}_\ast$-cotensoring is given by the bifunctor
$(K, Y)\mapsto F_X(X^\ast K, Y)$.
The pullback functors are strongly closed, so preserve $\mathrm{sSet}_\ast$-cotensors.
\end{proof}

The categories of retractive spaces inherit natural model structures from $\mathrm{sSet}$.
A morphism $\psi$ in $\mathrm{sSet}_{\dslash X}$ is a weak equivalence, fibration, or cofibration precisely if its underlying map of simplicial sets is such.
This model structure inherits a number of useful properties from the Kan model structure:
\begin{lemma}
\label{lem:RetSpaceLocalModelStructure}
For any simplicial set $X$, the model category $\mathrm{sSet}_{\dslash X}$ is combinatorial and left proper.
Moreover
\begin{itemize}
  \item the collection
  \[
  \mathcal{I}^\mathrm{Kan}_X :=\left\{
  \big(\sigma|_{\partial \Delta^n} \to \sigma\big)_{+X} \;\Big|\; (\sigma \colon \Delta^n \to X)\in \mathrm{el}(X)
  \right\}\,
  \] 
  is a set of generating cofibrations; and
  
  \item the collection
  \[
  \mathcal{J}^\mathrm{Kan}_X :=\left\{
  \big(\sigma|_{\Lambda^n_k} \to \sigma\big)_{+X} \;\Big|\; 0\leq k\leq n,\;(\sigma \colon \Delta^n \to X)\in \mathrm{el}(X)
  \right\}
  \]
  is a set of generating acyclic cofibrations.
\end{itemize}
\end{lemma}
\begin{theorem}
\label{thm:RetSpStructureTheorem}
For any simplicial set $X$ and map $f\colon X\to X'$, the category $\mathrm{sSet}_{\dslash X}$ is a $\mathrm{sSet}_\ast$-model category and $(f_!\dashv f^\ast)$ is a $\mathrm{sSet}_\ast$-enriched Quillen adjunction.
Moreover
\begin{enumerate}[label=\emph{(\roman*)}]
  \item if $f$ is a weak equivalence then $(f_!\dashv f^\ast)$ is a $\mathrm{sSet}_\ast$-Quillen equivalence;
  \item if $f$ is a fibration or a projection to a factor of a product then $(f^\ast\dashv f_\ast)$ is a $\mathrm{sSet}_\ast$-Quillen adjunction; and 
  \item if $f$ is an acyclic fibration then $(f^\ast\dashv f_\ast)$ is a $\mathrm{sSet}_\ast$-Quillen equivalence.
\end{enumerate}
\end{theorem}
\begin{proof}
To show that $\mathrm{sSet}_{\dslash X}$ is a $\mathrm{sSet}_\ast$-model category, we must verify that the tensor bifunctor satisfies the pushout-product axiom.
That is, we must show that for cofibrations $i\colon K\to L$ and $i_X \colon Y\to Z$ in $\mathrm{sSet}_\ast$ and $\mathrm{sSet}_{\dslash X}$ respectively, the pushout-product
\[
i\,\square\, i_X \colon \big(L \owedge_X Y\big)\coprod_{K\owedge_X Y} \big(K\owedge_X Z\big)\longrightarrow L\owedge_X Z
\]
is a cofibration and is moreover an acyclic cofibration if one of $i$ or $i_X$ is acyclic.
In the case of spaces (resp.~retractive spaces) with basepoint (resp.~section) freely adjoined, we have an isomorphism of retractive spaces 
\[
\big(L_+\big)\owedge_X \big(Y_{+X}\big)\cong (L\times Y)_{+ X}\,.
\]
Hence for $(i_+\colon \partial \Delta^n_+ \to \Delta^n_+)\in \mathcal{I}^\mathrm{Kan}_\ast$ and $(i_X\colon \partial \Delta^m_{+X} \to \Delta^m_{+X})\in \mathcal{I}^\mathrm{Kan}_X$ (in the latter case, we are suppressing the map $\sigma\colon \Delta^m \to X$ in our notation), the pushout-product $i\,\square\, i_X$ is isomorphic to a map
\begin{equation}
\label{eqn:PushoutProdforsSetTensoring}
\left((\Delta^n\times \partial \Delta^m)\coprod_{\partial \Delta^n \times\partial\Delta^m} (\partial \Delta^n \times \Delta^m)\right){\!\!\vphantom{\big)}}_{+X}\longrightarrow \left(\Delta^n\times \Delta^m \right){\!\vphantom{\big)}}_{+X}\,,
\end{equation}
where maps to $X$ are determined in all cases by projecting to the second factor in the product and using the given map to $X$.
Using the shuffle decomposition of a product of simplices, the map \eqref{eqn:PushoutProdforsSetTensoring} is seen to be an iterated pushout of maps in $\mathcal{I}^\mathrm{Kan}_X$.
Thus we have shown that every pushout-product of generating cofibrations is a cofibration.
Since the tensoring bifunctor preserves pushouts, transfinite composition and retracts separately in each variable, cofibrant generation implies that the pushout-product of any two cofibrations is again a cofibration.
The remaining clause of the pushout-product axiom is proven analogously; replacing simplicial boundaries by horns in \eqref{eqn:PushoutProdforsSetTensoring}, it is easy to see that the sets
\[
\mathcal{I}^\mathrm{Kan}_\ast \,\square\, \mathcal{J}^\mathrm{Kan}_X
\;\;
\mbox{ and }
\;\;
\mathcal{J}^\mathrm{Kan}_\ast \,\square\, \mathcal{I}^\mathrm{Kan}_X
\]
consist of acyclic cofibrations.

We now fix a map $f\colon X\to X'$.
The pushforward functor $f_!$ sends $f_!(\mathcal{I}^\mathrm{Kan}_X)\subset \mathcal{I}^\mathrm{Kan}_{X'}$ and $f_!(\mathcal{J}^\mathrm{Kan}_X)\subset \mathcal{J}^\mathrm{Kan}_{X'}$, so $f_!$ preserves cofibrations and acyclic cofibrations.
Since $f_!$ and $f^\ast$ preserve $\mathrm{sSet}_\ast$-tensors by Corollary \ref{cor:sSetastEnrichmentFacts},  $(f_!\dashv f^\ast)$ is a $\mathrm{sSet}_\ast$-Quillen adjunction.
Note that pullbacks preserve monomorphisms, so $f^\ast$ also preserves cofibrations.

Now suppose that $f$ is a weak equivalence and fix a map $\psi\colon f_! Y \to Z$ in $\mathrm{sSet}_{\dslash X'}$ with fibrant codomain.
The map $\psi$ and its $(f_!\dashv f^\ast)$-adjunct $\psi^\vee$ fit into the commuting diagram
\[
\begin{tikzcd}
X
\ar[rr, rightarrowtail]
\ar[dd, "f"']
&&
Y
\ar[dl]
\ar[dr, "\psi^\vee"']
\ar[rr]
&&
X
\ar[dd, "f"]
\\
&
f_!Y
\ar[dr, "\psi"]
&&
f^\ast Z
\ar[ur, twoheadrightarrow]
\ar[dl]
\\
X'
\ar[ur, rightarrowtail]
\ar[rr, rightarrowtail]
&&
Z
\ar[rr, twoheadrightarrow]
&&
X'\,.
\end{tikzcd}
\]
Since $\mathrm{sSet}$ is a proper model category, the pushout morphism $Y\to f_! Y$ and pullback morphism $f^\ast Z \to Z$ are both weak equivalences. 
The $2$-out-of-$3$ property now implies that $\psi$ is a weak equivalence if and only if $\psi^\vee$ is.
Thus $(f_!\dashv f^\ast)$ is a $\mathrm{sSet}_\ast$-Quillen equivalence.

If $f$ is a fibration or a projection to a factor of a product then $f^\ast$ also preserves weak equivalences. We had previously remarked that $f^\ast$ preserves cofibrations and so $(f^\ast\dashv f_\ast)$ is a Quillen adjunction in this case.
It is a $\mathrm{sSet}_\ast$-Quillen adjunction since $f^\ast$ preserves $\mathrm{sSet}_\ast$-tensors.
We note that in this case there is an isomorphism of derived functors $\mathbf{R}f^\ast \cong \mathbf{L}f^\ast$ since $f^\ast$ preserves cofibrations, fibrations, and weak equivalences.
If $f$ is moreover an acyclic fibration, then we saw above that $(\mathbf{L}f_! \dashv \mathbf{R}f^\ast)$ is an equivalence of categories. Using the isomorphism of derived functors  $\mathbf{R}f^\ast \cong \mathbf{L}f^\ast$, the fact that an adjoint equivalence $(L\dashv R)$ can be modified to become an adjoint equivalence $(R\dashv L)$, and essential uniqueness of adjoints, we conclude that $\mathbf{L}f_!\cong \mathbf{R}f_\ast$ so that $(\mathbf{L}f^\ast\dashv \mathbf{R}f_\ast)$, too, is an equivalence of categories.
\end{proof}

\begin{remark}
\label{rem:SmashFail}
The monoidal unit $S^0 =\ast_+ = \ast \coprod \ast$ in $\mathrm{sSet}_\ast$ is cofibrant, so we have also (re)proven the fact that $(\mathrm{sSet}_\ast, \wedge)$ is a symmetric monoidal model category.
Yet, unless $X=\ast$ the fibrewise smash product $\wedge_X$ does \emph{not} make $\mathrm{sSet}_{\dslash X}$ a monoidal model category.
The issue at the heart of this deficiency is that the fibre product $A\times_X B$ does not generally preserve weak equivalences.
We can, however, partially ameliorate this issue due to the following result.
\end{remark}

\begin{lemma}
\label{lem:FibrantRetSpaceSmash}
If $Y\in \mathrm{sSet}_{\dslash X}$ is fibrant then the endofunctor $Y\wedge_X(-)$ preserves cofibrations and weak equivalences. 
In particular, it is a left Quillen endofunctor.
\end{lemma}
\begin{proof}
Arguing on generating cofibrations as in the proof of Theorem \ref{thm:RetSpStructureTheorem}, we show that the $\wedge_X$-pushout-product  of any pair of cofibrations is again a cofibration. 
All objects of $\mathrm{sSet}_{\dslash X}$ are cofibrant, so that $Y\wedge_X (-)$ preserves cofibration without any hypotheses on $Y$.

If $Y$ is moreover fibrant, so that the projection $r\colon Y\to X$ is a fibration, then $Y\times_X(-)$ preserves weak equivalences.
Now suppose that $\zeta\colon Z_0\to Z_1$ is a weak equivalence in $\mathrm{sSet}_{\dslash X}$.
Then the cube lemma \cite[Lemma 5.2.6]{hovey_model_1999} implies that
$
Y\coprod_X Z_0 \to Y\coprod_X Z_1
$
is a weak equivalence of simplicial sets (since $X\to Y$ is necessarily a cofibration).
Applying the cube lemma once more to the pushout diagrams
\[
\begin{tikzcd}
Y\displaystyle\coprod_X Z_i
\ar[r, rightarrowtail]
\ar[d]
&
Y\displaystyle\prod_X Z_i
\ar[d]
\\
X
\ar[r]
&
Y\wedge_X Z_i
\end{tikzcd}
\]
for $i=0,1$ implies that $Y\wedge_X \zeta\colon Y\wedge_X Z_0\to Y\wedge_X Z_1$ is a weak equivalence.
Note that we used that the $X$-parametrised shear maps $Y\coprod_X Z_i \to Y\prod_X Z_i$ are monomorphisms, hence cofibrations, which holds without any additional hypotheses on the $Z_i$.
\end{proof}

\subsection{Global theory}
Using the work we have done in the previous section, we can regard categories of retractive spaces as determining a pseudofunctor $\mathrm{sSet}_{\dslash -}\colon \mathrm{sSet}\to \mathbf{Model}$, where
\[
\big(f\colon X\to Y\big)\longmapsto
\left(
\begin{tikzcd}
\mathrm{sSet}_{\dslash X}
  \ar[rr, shift left=1.1ex, "f_!"]
  \ar[rr, leftarrow, shift left=-1.1ex , "\bot", "f^\ast"']
  &&
  \mathrm{sSet}_{\dslash Y}
\end{tikzcd}
\right).
\]
We would now like to produce a model structure on the Grothendieck construction $\int_{X\in \mathrm{sSet}}\mathrm{sSet}_{\dslash X}$ that combines the model structures on $\mathrm{sSet}$ and the $\mathrm{sSet}_{\dslash X}$ in a sensible fashion.
To do this we use the Grothendieck construction for model categories from \cite{harpaz_grothendieck_2015}, for which we need the following
\begin{lemma}
\label{lem:RetSpacPropRelative}
The pseudofunctor $\mathrm{sSet}_{\dslash -}$ is proper and relative.
\end{lemma}
\begin{proof}
A pseudofunctor $F\colon M\to \mathbf{Model}$ defined on a model category $M$ is \emph{relative} if it sends weak equivalences to Quillen equivalences. 
The pseudofunctor $F$ is \emph{left} (resp.~\emph{right}) \emph{proper} if for every acyclic cofibration (resp.~acyclic fibration) $f$ the pushforward $F(f)_!$ (resp.~pullback $F(f)^\ast$) preserves weak equivalences. We say that $F$ is \emph{proper} if it is both left and right proper.

In the case presently under consideration, Theorem \ref{thm:RetSpStructureTheorem} immediately implies that $\mathrm{sSet}_{\dslash -}$ is relative and right proper. To see that it is left proper, note that as all objects of $\mathrm{sSet}_{\dslash X}$ are cofibrant, the left Quillen functor $f_!$ preserves weak equivalences for \emph{any} map of simplicial sets $f\colon X\to X'$ (by Ken Brown's lemma).
\end{proof}

\begin{definition}
The \emph{(global) category of retractive spaces} is 
\[
\mathrm{sSet}_{\dslash \mathrm{sSet}} :=\int_{X\in \mathrm{sSet}}\mathrm{sSet}_{\dslash X}\,,
\]
 the Grothendieck construction of the pseudofunctor $\mathrm{sSet}_{\dslash -}$.
The objects of $\mathrm{sSet}_{\dslash \mathrm{sSet}}$ are pairs $(X,Y)$ with $X\in \mathrm{sSet}$ and $Y\in \mathrm{sSet}_{\dslash X}$; the morphisms $(f,\psi)\colon (X, Y)\to (X',Z)$ are pairs of morphisms $f\colon X\to X'$ and $\psi\colon f_!(Y)\to Z$ in $\mathrm{sSet}$ and $\mathrm{sSet}_{\dslash X'}$ respectively.
\end{definition}
\begin{remark}
\label{rem:GlobalMorphisms}
The following data are interchangeable:
\begin{enumerate}[label=(\roman*)]
  \item morphisms $(f,\psi)\colon (X,Y)\to (X', Z)$ in $\mathrm{sSet}_{\dslash \mathrm{sSet}}$\,;
  \item morphisms $\psi\colon f_! (Y)\to Z$ in $\mathrm{sSet}_{\dslash X'}$\,;
  \item morphisms $\psi^\vee \colon Y\to f^\ast(Z)$ in $\mathrm{sSet}_{\dslash X}$\,; and
  \item commuting diagrams of simplicial sets
  \[
  \begin{tikzcd}
  X
  \ar[r]
  \ar[rr, bend left=25, "\mathrm{id}_X"]
  \ar[d, "f"']
  &
  Y
  \ar[r]
  \ar[d, "\Psi"']
  &
  X
  \ar[d, "f"]
  \\
  X'
  \ar[rr, bend left=-25, "\mathrm{id}_{X'}"']
  \ar[r]
  &
  Z
  \ar[r]
  &
  X'\,.
  \end{tikzcd}
  \]
\end{enumerate}
That (i) and (ii) are equivalent data is clear. The equivalence between (ii) and (iii) comes from passing to $(f_!\dashv f^\ast)$-adjuncts.
Finally, (iii) and (iv) are equivalent by the universal property of the pullback.
\end{remark}

\begin{theorem}
$\mathrm{sSet}_{\dslash \mathrm{sSet}}$ has a \emph{global model structure}, with respect to which $(f,\psi)\colon (X,Y)\to (X', Z)$ is
\begin{itemize}
  \item a weak equivalence if $f$ and $\psi$ are weak equivalences in $\mathrm{sSet}$ and $\mathrm{sSet}_{\dslash X'}$ respectively;
  \item a cofibration if $f$ and $\psi$ are cofibrations in $\mathrm{sSet}$ and $\mathrm{sSet}_{\dslash X'}$ respectively; and
  \item  a fibration if $f$ and $\psi^\vee$ are fibrations in $\mathrm{sSet}$ and $\mathrm{sSet}_{\dslash X}$ respectively.
\end{itemize}
\end{theorem}
\begin{proof}
By Lemma \ref{lem:RetSpacPropRelative}, the conditions of \cite[Theorem 3.0.12]{harpaz_grothendieck_2015} are satisfied so that the integral model structure exists and has weak equivalences, cofibrations and fibrations as stated.
\end{proof}
\begin{corollary}
\label{cor:RetSpBase}
The projection functor $p\colon \mathrm{sSet}_{\dslash \mathrm{sSet}}\to \mathrm{sSet}$ is left and right Quillen.
\end{corollary}
\begin{proof}
The projection functor $p$ sends $(X,Y)\mapsto X$. This assignment clearly preserves cofibrations, fibrations and weak equivalences.
The functor $0_{-}\colon X\to (X\to X\to X)$ which sends each simplicial set $X$ to the zero object of $\mathrm{sSet}_{\dslash X}$ is a two-sided adjoint to $p$.
\end{proof}
\begin{corollary}
\label{cor:FibreInclusionRetSp}
For any simplicial set $X$, the fibre inclusion $\imath_X\colon \mathrm{sSet}_{\dslash X}\to \mathrm{sSet}_{\dslash \mathrm{sSet}}$ is left and right Quillen.
\end{corollary}
\begin{proof}
The fibre inclusion preserves limits and colimits. As $\mathrm{sSet}_{\dslash \mathrm{sSet}}$ is locally presentable (Lemma \ref{lem:GlobRetSpaceLocPres}), $\imath_X$ has a left and right adjoint by the adjoint functor theorem.
It is clear that $\imath_X$ preserves weak equivalences, fibrations and cofibrations so is both left and right Quillen.
\end{proof}
\begin{corollary}
There is a Quillen adjunction
\[
\begin{tikzcd}
\mathrm{Fun}(\Delta^1, \mathrm{sSet})
\ar[rr, shift left=1.1ex, "(\mathtt{0})_{+(\mathtt{1})}
"]
\ar[rr, leftarrow, shift left=-1.1ex, "\bot", "U"']
&&
\mathrm{sSet}_{\dslash \mathrm{sSet}}\,,
\end{tikzcd}
\]
where $\mathrm{Fun}(\Delta^1, \mathrm{sSet})$ is equipped with either the projective or injective model structure on functors.
\end{corollary}
\begin{proof}
The right adjoint $U$ sends a pair $(X,Y)$ to the arrow $Y \to X$ obtained by forgetting the section of the retractive space $Y$ over morphism in $\mathrm{sSet}_{\dslash \mathrm{sSet}}$, then $U(f,\psi)$ is the commuting diagram
\[
\begin{tikzcd}
Y
\ar[r]
\ar[d, "\Psi"']
&
X
\ar[d, "f"]
\\
Z
\ar[r]
&
X'\,.
\end{tikzcd}
\]
in which the horizontal morphisms are interpreted as objects in the arrow category (see Remark \ref{rem:GlobalMorphisms}).
The forgetful functor $U$ has a left adjoint defined on objects by
$
(\mathtt{0})_{+(\mathtt{1})}\colon (Y\to X) \mapsto Y_{+X}\,,
$ 
and which operates in the obvious way on morphisms. 
It is clear that the functor $(\mathtt{0})_{+(\mathtt{1})}$ sends levelwise cofibrations and weak equivalences to cofibrations and weak equivalences in $\mathrm{sSet}_{\dslash \mathrm{sSet}}$ respectively, so is left Quillen.
\end{proof}

We now give an alternative presentation of the global category of retractive spaces. 
This makes certain properties, such as local presentability and the external smash product, easier to check.
\begin{construction}
\label{cons:GlobRetSpaceAlg}
Let $i_{\{0,2\}}\colon \Delta^1 \to \Delta^2$ be the $1$-simplex on the vertices $0$ and $2$, denoted $\Delta^{\{0,2\}}$.
In the strict pullback diagram
\[
\begin{tikzcd}
  \mathrm{sSet}'_{\dslash\mathrm{sSet}}
  \ar[r]
  \ar[d, "p'"']
  & 
  \mathrm{Fun}(\partial\Delta^2, \mathrm{sSet})
  \ar[d, "i^\ast_{\{0,2\}}"]\\
  \mathrm{sSet}
  \ar[r, "X\mapsto \mathrm{id}_X"]
  &
  \mathrm{Fun}(\Delta^{\{0,2\}},\mathrm{sSet})
\end{tikzcd}
\]
$\mathrm{sSet}'_{\dslash\mathrm{sSet}}$ is then full subcategory of functors $\xi\colon \partial\Delta^2 \to \mathrm{sSet}$ such that $\xi|_{\Delta^{\{0,2\}}}$ is the identity of some simplicial set $X$. 

There is a canonical isomorphism of categories $\mathrm{sSet}_{\dslash \mathrm{sSet}} \cong \mathrm{sSet}'_{\dslash \mathrm{sSet}}$; a diagram $X\to Y\to X$ is identified with the pair $(X,Y)$ and this identification is extended to morphism via Remark \ref{rem:GlobalMorphisms}.
This isomorphism respects the projection functors to $\mathrm{sSet}$.

Without further comment, we shall write $\mathrm{sSet}_{\dslash \mathrm{sSet}}$ interchangeably for either of the categories that we have heretofore denoted $\mathrm{sSet}_{\dslash \mathrm{sSet}}$ and $\mathrm{sSet}'_{\dslash \mathrm{sSet}}$.
Objects of $\mathrm{sSet}_{\dslash \mathrm{sSet}}$ will be written either as pairs $(X,Y)$ or diagrams $X\to Y\to X$ depending on the context.
\end{construction}

\begin{lemma}
\label{lem:GlobRetSpaceLocPres}
$\mathrm{sSet}_{\dslash \mathrm{sSet}}$ is locally presentable.
\end{lemma}
\begin{proof}
$\mathrm{sSet}_{\dslash \mathrm{sSet}}$ is closed under small colimits, so it is sufficient to show that it is accessible.
Let us write
$\Delta_\mathrm{sSet}\hookrightarrow \mathrm{Fun}(\Delta^1,\mathrm{sSet})$ for the accessibly embedded full subcategory on the identity morphisms,  then by the above construction $\mathrm{sSet}_{\dslash\mathrm{sSet}}$ is equivalently the inverse image of $\Delta_\mathrm{sSet}$ by the colimit-preserving pullback functor $i_{\{0,2\}}^\ast$ so is itself accessible \cite[Corollary A.2.6.5]{lurie_higher_2009}.
\end{proof}

\begin{lemma}
\label{lem:GlobRetSpaceComb}
The global model structure on $\mathrm{sSet}_{\dslash \mathrm{sSet}}$ is left proper and combinatorial.
\end{lemma}
\begin{proof}
To prove left properness, consider a pushout diagram in $\mathrm{sSet}_{\dslash \mathrm{sSet}}$
\[
\begin{tikzcd}
  (X,A)
  \ar[r, rightarrowtail, "{(f,\varphi)}"]
  \ar[d, "{(g,\psi)}"', "\sim"]
  &
  (Y, B) 
  \ar[d, "{(g', \psi')}"]
  \\
  (X', A')
  \ar[r, rightarrowtail, "{(f', \varphi')}"] 
  & 
  (Y', B')
\end{tikzcd}
\]
with cofibrations and weak equivalences as marked.
We must show that $(g',\psi')$ is a weak equivalence.
Indeed, the map $g'$ is a weak equivalence by left properness of the Kan model structure on simplicial sets.
The retractive space $B'\in \mathrm{sSet}_{\dslash Y'}$ is the pushout of the span
\[
\begin{tikzcd}
  f'_! (A')
  \ar[r, leftarrow, "{f'_!(\psi)}"]
  &
  f'_! g_! (A) \cong g'_! f_! (A)
  \ar[r, rightarrowtail, "{g'_!(\varphi)}"]
  &
  g'_! (B)\,.
\end{tikzcd}
\]
By hypothesis, the morphism $\psi'\colon g_! (A)\to A'$ is a weak equivalence in $\mathrm{sSet}_{\dslash X'}$.
As observed in the proof of Lemma \ref{lem:RetSpacPropRelative}, the pushforward functors preserve weak equivalences so that $f'_! (\psi)$ is a weak equivalence.
As $\mathrm{sSet}_{\dslash X'}$ is left proper, the pushout morphism $\psi'\colon g_!'(B) \to B'$ is a weak equivalence, as required.

The category $\mathrm{sSet}_{\dslash \mathrm{sSet}}$ is locally presentable, so is combinatorial provided the global model structure is cofibrantly generated.
The small object argument always holds in locally presentable categories, so cofibrant generation is guaranteed if we can find sets $\mathcal{I}_{\mathrm{sSet}}^{\mathrm{Kan}}$ and $\mathcal{J}_{ \mathrm{sSet}}^{\mathrm{Kan}}$ characterising acyclic fibrations and fibrations respectively via the right lifting property.

To this end, let $\mathcal{I}_{ \mathrm{sSet}}^{\mathrm{Kan}}$ be the union of the sets
\begin{itemize}
  \item $0_-(\mathcal{I}_\mathrm{Kan})$ of morphisms obtained by applying the left Quillen functor $0_{-}\colon \mathrm{sSet}\to \mathrm{sSet}_{\dslash \mathrm{sSet}}$ to the set $\mathcal{I}_\mathrm{Kan}=\{i_n\colon \partial\Delta^n\hookrightarrow \Delta^n\}$; and
  
  \item $\mathcal{I}^\mathrm{Kan}_+$ of morphisms 
  \[
  \big(\mathrm{id}|_{\partial\Delta^n}\to \mathrm{id}\big)_{+\Delta^n}
  \equiv
  \left(
\!
\begin{tikzcd}
    \Delta^n \ar[r]\ar[d]  & \Delta^n\displaystyle\coprod\Delta^n 
    \ar[d]
    \\
    \partial\Delta^n\displaystyle\coprod\Delta^n 
    \ar[ur] 
    \ar[r]
    &
    \Delta^n
   \end{tikzcd}\!
   \right),
  \quad n\geq 0.
  \]
\end{itemize}
Since all morphisms in $\mathcal{I}_{\mathrm{sSet}}^{\mathrm{Kan}}$ are cofibrations we have $\mathcal{F}\cap \mathcal{W}\subset \mathrm{rlp}(\mathcal{I}_{\mathrm{sSet}}^{\mathrm{Kan}})$.
To show the reverse inclusion, suppose that $(f,\psi)\colon (X, Y)\to (X', Z)$ is a morphism with the right lifting property with respect to $\mathcal{I}_{\mathrm{sSet}}^{\mathrm{Kan}}$.
By Corollary \ref{cor:RetSpBase}, the map $f\colon X\to X'$ of base spaces has the right lifting property with respect to $\mathcal{I}^\mathrm{Kan}$, hence is an acyclic fibration.
Since $f$ is an acyclic fibration, $\psi\colon f_! Y\to Z$ is a weak equivalence precisely if its adjunct $\psi^\vee\colon Y\to f^\ast Z$ is. 
It is therefore sufficient to show that $\psi^\vee$ is an acyclic fibration.
For any $n$-simplex $\sigma \colon \Delta^n \to X$, we have
\[
\sigma_! \left(\big(\mathrm{id}|_{\partial\Delta^n}\to \mathrm{id}\big)_{+\Delta^n}\right) = \big(\sigma|_{\partial\Delta^n}\to \sigma\big)_{+X}
\]
under pushforward.
It follows that liftings of the diagrams
\[
\begin{tikzcd}
  (\sigma|_{\partial\Delta^n})_{+X}
  \ar[r, "\alpha"]
  \ar[d] & Y\ar[d, "\psi^\vee"] \ar[d]
  \\
  \sigma_{+X} \ar[r,"\beta"]& f^\ast Z
\end{tikzcd}
\quad
\text{in $\mathrm{sSet}_{\dslash X}$,}
\;\;\; \text{and}\quad
\begin{tikzcd}
  (\mathrm{id}|_{\partial\Delta^n})_{+\Delta^n} 
  \ar[r,"{(\sigma, \alpha)}"]
  \ar[d] 
  & 
  (X,Y) 
  \ar[d, "{(\mathrm{id}_X,\psi^\vee)}"]
  \\
  \mathrm{id}_{+\Delta^n}
  \ar[r, "{(\sigma,\beta)}"]
  & (X,f^\ast Z)
  \end{tikzcd}
 \quad
 \text{in $\mathrm{sSet}_{\dslash \mathrm{sSet}}$}
\]
are equivalent. The latter diagram admits lifts by hypothesis, so that $\psi^\vee$ has the right lifting property with respect to the set $\mathcal{I}^\mathrm{Kan}_X$ and hence is an acyclic fibration.

A similar argument shows that fibrations in $\mathrm{sSet}_{\dslash \mathrm{sSet}}$ are precisely those morphisms having the right lifting property with respect to the set $\mathcal{J}^\mathrm{Kan}_{\mathrm{sSet}} := 0_-(\mathcal{J}^\mathrm{Kan})\cup \mathcal{J}^\mathrm{Kan}_+$ (obtained by replacing boundary inclusions $\partial\Delta^n\hookrightarrow \Delta^n$ by horn inclusions $\Lambda^n_k\hookrightarrow \Delta^n$ throughout).
\end{proof}

Working in the global category of retractive spaces allows us to define the \emph{external} smash product.
Throughout the rest of this section, we define the external smash product and discuss some of its basic properties. 
In particular we show that the external smash product interacts nicely with the global model structure, unlike the fibrewise smash product over a fixed base space.

For retractive spaces $(X, Y), (X', Z)\in \mathrm{sSet}_{\dslash \mathrm{sSet}}$, the external smash product $(X, Y)\esmash (X', Z)$ is defined as the colimit of the diagram of simplicial sets
\[
\begin{tikzcd}
  X\times X' 
  \ar[r] 
  \ar[d]
  & X\times Z 
  \ar[dd]
  \ar[dr]
  \\
  Y\times X' 
  \ar[rr, crossing over]
  \ar[dr] && Y\times  Z\\
  & X\times X'\,.
\end{tikzcd}
\]
By construction, $(X, Y)\esmash (X', Z)$ is canonically equipped with the structure of a retractive space over $X\times X'$. 
We leave it as an exercise to the reader to verify that $\vartriangle$ defines a symmetric monoidal structure with monoidal unit $(\ast, S^0)$, and that it preserves colimits in each variable.
The latter fact combined with local presentability implies that $(\mathrm{sSet}_{\dslash \mathrm{sSet}},\vartriangle)$ is a closed symmetric monoidal category; we write $F_{\vartriangle}$ for the internal hom.
\begin{lemma}
\label{lem:RetSpProjClosedMonoidal}
The projection functor $p\colon \mathrm{sSet}_{\mathrm{sSet}}\to \mathrm{sSet}$ is strongly closed monoidal.
\end{lemma}
\begin{proof}
It is clear from the definitions that $p$ sends $\vartriangle$ to $\times$, so is strongly symmetric monoidal.
Now for $(X, Y), (X',Y')\in \mathrm{sSet}_{\dslash \mathrm{sSet}}$ and $Z\in \mathrm{sSet}$, the base projection is a left and right adjoint so that
\[
\frac{Z\longrightarrow p F_\vartriangle\big\{ (X,Y),(X',Y')\big\}}{(Z,0_Z)\esmash(X,Y)\to (X',Y')}\,,
\]
where $0_Z$ is the zero object of $\mathrm{sSet}_{\dslash Z}$.
But $(Z,0_Z)\esmash(X,Y) \cong (Z\times X , 0_{Z\times X})$ and hence morphisms $Z\to F_\vartriangle\big\{(X,Y),(X',Y')\big\}$ are equivalent to morphisms $(Z\times X , 0_{Z\times X})\to (X',Y')$ and hence, by adjunction, to morphisms $Z\to [X,X']$.
We conclude that $p$ sends the internal hom $F_\triangle(-,-)$ to the exponential object $[p-,p-]$.
\end{proof}

\begin{theorem}
\label{thm:RetSpExtSmash}
$(\mathrm{sSet}_{\dslash \mathrm{sSet}},\vartriangle)$ is a symmetric monoidal model category.
\end{theorem}
\begin{proof}
The monoidal unit $(\ast, S^0)$ is cofibrant, so we need only verify the pushout-product axiom for $\vartriangle$.
For this we use the generating sets $\mathcal{I}^\mathrm{Kan}_{\mathrm{sSet}}$ and $\mathcal{J}^\mathrm{Kan}_{\mathrm{sSet}}$ of Lemma \ref{lem:GlobRetSpaceComb}.

Suppose that $i, i' \in \mathcal{I}^\mathrm{Kan}_{\mathrm{sSet}}$. Due to the symmetry of $\vartriangle$, there are essentially two cases to consider:
\begin{enumerate}[label=(\arabic*)]
  \item If at least one of $i, i'$ is in $c(\mathcal{I}^\mathrm{Kan})$, then since $(X, 0_X)\esmash (X', Y) \cong (X\times X', 0_{X\times X'})$ for all $(X',Y)$, we have
  \[
  i\,\square\, i' = 0_{-} \big(p(i)\,\square\, p(i')\big)\,
  \]
  where on the right-hand side $\square$ signifies the pushout-product morphism for the Cartesian product on $\mathrm{sSet}$.
  But the Cartesian product is a monoidal model structure on $\mathrm{sSet}$, so $i\,\square\,i'$ is a cofibration by Corollary \ref{cor:RetSpBase}.
  
  \item If $i, i'\in \mathcal{I}^\mathrm{Kan}_+$, then let us fix $i = (\mathrm{id}|_{\partial\Delta^n}\to \mathrm{id})_{+\Delta^n}$ and  $i' = (\mathrm{id}|_{\partial\Delta^m}\to \mathrm{id})_{+\Delta^m}$ for some $m,n \geq 0$.
  A direct calculation shows that $i\,\square\,i'$ is the image under $(-)_{+ (\Delta^n\times \Delta^m)}$ of the morphism $\xi$ in the  diagram
  \[
  \begin{tikzcd}
    \partial\Delta^n \times \partial\Delta^m \ar[r]\ar[d] & \partial\Delta^n \times \Delta^m\ar[dr, bend left=25]
    \ar[d] &
    \\
    \Delta^n\times \partial\Delta^m \ar[r]
    &
    P
    \ar[r, "\xi"]
    &
    \Delta^n\times \Delta^m\,,
  \end{tikzcd}
  \]
  where the left-hand square is a pushout.
  The map $\xi$ is a cofibration in $\mathrm{sSet}$, hence is also a cofibration when regarded as a morphism in $\mathrm{sSet}_{/(\Delta^n\times \Delta^m)}$.
  Since the functor $(-)_{+(\Delta^n\times \Delta^m)}$ is left Quillen, we conclude that $i\,\square\,i'$ is a cofibration.
\end{enumerate}
We have shown that the set $\mathcal{I}^\mathrm{Kan}_\mathrm{sSet}\,\square\,\mathcal{I}^\mathrm{Kan}_\mathrm{sSet}$ consists of cofibrations.
Since $(-)\square(-)$ preserves retracts, pushouts and transfinite compositions separately in each variable, we conclude by cofibrant generation that the class of cofibrations is closed under forming $\vartriangle$-pushout-products.

To conclude, note that by symmetry it is sufficient to show that $i\,\square\, j$ is an acyclic cofibration for any cofibration $i$ and acyclic cofibration $j$.
Suppose firstly that $i\in \mathcal{I}^\mathrm{Kan}_{\mathrm{sSet}}$ and $j\in \mathcal{J}^\mathrm{Kan}_\mathrm{Kan}$, so that there are two cases to consider:
\begin{enumerate}[label=(\arabic*)]
  \item At least one $i$, $j$ is in the image of the functor $0_{-}$, in which case 
  \[
  i\,\square\, j = 0_{-}\big(p(i)\,\square\,p(j)\big)
  \]
  is an acyclic cofibration, similarly to out previous argument above.
 
 \item Otherwise, $i= (\mathrm{id}_{|\partial \Delta^n}\to \mathrm{id})_{+\Delta^n}$ and $j= (\mathrm{id}|_{\Lambda^m_k}\to \mathrm{id})_{+\Delta^m}$ for some $n,m\geq 0$ and $0\leq k\leq m$.
 Then $i\,\square\, j$ is the image under the left Quillen functor $(-)_{+(\Delta^n\times \Delta^m)}$ of the morphism $\zeta$ in the diagram
 \[
 \begin{tikzcd}
    \partial\Delta^n \times \Lambda^m_k \ar[r]\ar[d] & \partial\Delta^n \times \Delta^m\ar[dr, bend left=25]
    \ar[d] &
    \\
    \Delta^n\times \Lambda^m_k \ar[r]
    &
    P
    \ar[r, "\zeta"]
    &
    \Delta^n\times \Delta^m\,,
 \end{tikzcd}
 \] 
 where the left-hand square is a pushout.
 Since $\zeta$ is an acyclic cofibration of simplicial sets, we deduce that $i\,\square\, j$ is an acyclic cofibration as well.
\end{enumerate}
Thus the set $\mathcal{I}^\mathrm{Kan}_\mathrm{sSet}\,\square\, \mathcal{J}^\mathrm{Kan}_\mathrm{sSet}$ consists of acyclic cofibrations. 
By cofibrant generation, the pushout-product axiom for $\vartriangle$ now follows.
\end{proof}

\begin{corollary}[Fibrewise enrichment]
\label{cor:GlobalRetSpPSSet}
The global model structure on $\mathrm{sSet}_{\dslash \mathrm{sSet}}$ is $\mathrm{sSet}_\ast$-enriched.
\end{corollary}
\begin{proof}
The canonical inclusion $\imath_\ast\colon\mathrm{sSet}_\ast \to \mathrm{sSet}_{\dslash \mathrm{sSet}}$ is a strongly monoidal left Quillen functor, so that $(K, (X,Y))\mapsto (\ast, K)\esmash(X,Y)$ defines a $\mathrm{sSet}_\ast$-tensoring compatible with the model structures.

The $\mathrm{sSet}_\ast$-enriched hom-spaces of $\mathrm{sSet}_{\dslash\mathrm{sSet}}$ are thus given by 
\[
\wp_\ast F_\vartriangle\!\left\{(X,Y), (X',Z)\right\}\cong [X, X']_\ast F_\vartriangle\!\left\{(X,Y), (X',Z)\right\}\,,
\]
the space of sections of the internal hom $F_\vartriangle\!\left\{(X,Y), (X',Z)\right\}$ over the exponential object $[X,X']$.
Here we have written $\wp_\ast$ for the right adjoint to $\imath_\ast$.
\end{proof}

We conclude this section by recording some relations between the fibrewise and external smash products.
The proofs of the next two results are relatively straightforward, so have been omitted.
\begin{lemma}
Let $X$ be a simplicial set with diagonal map $\Delta_X\colon X\to X\times X$.
For any $Y,Z\in \mathrm{sSet}_{\dslash X}$ there is a natural isomorphism
\[
\Delta^\ast_X \big((X,Y)\esmash(X,Z))\big) \longrightarrow Y\wedge_X Z
\]
of retractive spaces covering the identity on $X$.
\end{lemma}
\begin{lemma}
\label{lem:ExtSmashtoFib}
For any $(X, Y), (X',Z)\in \mathrm{sSet}_{\dslash \mathrm{sSet}}$, there is a natural isomorphism
\[
(X,Y)\esmash(X',Z)\longrightarrow p_1^\ast Y\wedge_{X\times X'} p_2^\ast Z
\]
of retractive spaces covering the identity on $X\times X'$, where $p_1, p_2$ are the projection maps.
\end{lemma}

The last result gives us another way of realising $\mathrm{sSet}_{\dslash \mathrm{sSet}}$ as a $\mathrm{sSet}$-model category.
This second $\mathrm{sSet}$-enrichment will play a key role in our proof of Theorem \ref{thm:Tangent}.
\begin{corollary}[Base-wise enrichment]
\label{cor:GlobRetSpBaseEnrich}
The global model structure on $\mathrm{sSet}_{\dslash \mathrm{sSet}}$ is $\mathrm{sSet}$-enriched.
\end{corollary}
\begin{proof}
For any simplicial set $K$, there is an isomorphism $K_{+K}= (K\to K\coprod K\to K)\cong K^\ast (S^0)$ of retractive spaces over $K$.
For simplicial sets $K$ and $L$,  Lemma \ref{lem:ExtSmashtoFib} implies natural isomorphisms
\[
(K, K^\ast S^0 )\esmash (L, L^\ast S^0) \cong p^\ast_1 K^\ast S^0 \wedge_{K\times L } p^\ast_2 L^\ast S^0 \cong (K\times L)^\ast S^0\wedge_{K\times L}  (K\times L)^\ast S^0
\cong (K\times L)^\ast S^0 \,,
\]
so that the assignment $
K\mapsto (K, K^\ast S^0)$
defines a strongly monoidal functor $\varepsilon\colon\mathrm{sSet}\to \mathrm{sSet}_{\dslash \mathrm{sSet}}$.
Since $\varepsilon$ clearly preserves cofibrations and weak equivalences, the assignment
\begin{align*}
\mathrm{sSet}\times \mathrm{sSet}_{\dslash \mathrm{sSet}}&\longrightarrow \mathrm{sSet}_{\dslash \mathrm{sSet}}
\\
\big(K, (X,Y))
&
\longmapsto 
\varepsilon(K) \esmash (X,Y)
\end{align*}
is a Quillen bifunctor.
\end{proof}

\begin{remark}
We have seen that $\mathrm{sSet}_{\dslash\mathrm{sSet}}$ carries two simplicial enrichments of very different flavours: the enrichment of Corollary \ref{cor:GlobalRetSpPSSet} captures simplicial homotopy information in the \lq\lq fibre direction'', whereas that of Corollary \ref{cor:GlobRetSpBaseEnrich} captures information in the \lq\lq base direction''.
\end{remark}

\subsection{Koszul (pre-)duality}
In this section, we discuss a model-theoretic formulation of Koszul duality for retractive spaces.
We now restrict our attention to the full subcategory $\mathrm{sSet}_{\geq 1}\hookrightarrow\mathrm{sSet}$ of reduced simplicial sets. 
Recall that a simplicial set $X$ is \emph{reduced} if it has precisely one $0$-simplex.
For reduced spaces, Kan's functor $\mathbb{G}$ sends the reduced simplicial set $X$ to a levelwise free simplicial group $\mathbb{G}X$ with the homotopy type of $\Omega X$.
There is an adjunction
\[
\begin{tikzcd}
\mathrm{sSet}_{\geq 1}
\ar[rr, shift left=1.1ex, "\mathbb{G}"]
\ar[rr, leftarrow, shift left=-1.1ex, "\bot", "\overline{W}"']
&&
\mathrm{sGrp}\,,
\end{tikzcd}
\]
with the right adjoint $G\mapsto\overline{W}G$ sending a simplicial group to Kan's model for the classifying space.
For each simplicial group $G$, there is a principal $G$-bundle $WG\to \overline{W}G$ with weakly acyclic total space; pulling back along the $(\mathbb{G}\dashv\overline{W})$-unit yields a principal $\mathbb{G}X$-bundle $p(X)\colon \mathbb{P}X\to X$ modelling the path fibration.

\begin{remark}
The category of reduced simplicial sets has a model structure with respect to which the cofibrations and weak equivalences are created by the inclusion $\mathrm{sSet}_{\geq 1}\hookrightarrow \mathrm{sSet}$.
The category of simplicial groups also has a model structure, in this case induced by transferring the Kan model structure along the free-forgetful adjunction
\[
\begin{tikzcd}
\mathrm{sSet}
\ar[rr, shift left=1.1ex, "\mathrm{free}"]
\ar[rr, leftarrow, shift left=-1.1ex, "\bot", ""']
&&
\mathrm{sGrp}\,.
\end{tikzcd}
\]
With respect to these model structures, the adjunction $(\mathbb{G}\dashv \overline{W})$ is a Quillen equivalence with the particularly nice property that both the unit and counit are natural weak equivalences.
\end{remark}

For a simplicial group $G$, write $G\mathrm{-sSet}$ for the category of simplicial sets endowed with a $G$-action.
We also write $G_+\mathrm{-Mod}_u$ for the category of pointed $G$-spaces; $G$-spaces with a chosen $G$-fixed basepoint or, equivalently, $G_+$-modules in $\mathrm{sSet}_\ast$ (the subscript \lq\lq u'' is for \lq\lq unstable'').
The free-forgetful adjunctions
\[
\begin{tikzcd}
\mathrm{sSet}
\ar[rr, shift left=1.1ex, "G\times(-)"]
\ar[rr, leftarrow, shift left=-1.1ex, "\bot", ""']
&&
G\mathrm{-sSet}
\end{tikzcd}
  \qquad
  \mbox{ and }
  \qquad
\begin{tikzcd}
\mathrm{sSet}_\ast
\ar[rr, shift left=1.1ex, "G_+\wedge(-)"]
\ar[rr, leftarrow, shift left=-1.1ex, "\bot", ""']
&&
G_+\mathrm{-Mod}_u
\end{tikzcd}
\]
equip both $G\mathrm{-sSet}$ and $G\mathrm{-Mod}_{u}$ with proper combinatorial model structures.
In particular, note that $G_+\mathrm{-Mod}_{u}$ is a $\mathrm{sSet}_\ast$-model category.
\begin{lemma}
\label{lem:RetSpandLoopSp}
For each reduced simplicial set $X$ there are $\mathrm{sSet}_\ast$-Quillen equivalences
\[
\begin{tikzcd}
\mathrm{sSet}_{\dslash X}
\ar[rr, shift left=1.1ex, "\mathfrak{f}_X"]
\ar[rr, leftarrow, shift left=-1.1ex, "\bot", "\mathfrak{b}_X"']
&&
\mathbb{G}X_+\mathrm{-Mod}_{u}
\end{tikzcd}
  \qquad
  \mbox{ and }
  \qquad
\begin{tikzcd}
\mathbb{G}X_+\mathrm{-Mod}_{u}
\ar[rr, shift left=1.1ex, "\mathfrak{b}_X"]
\ar[rr, leftarrow, shift left=-1.1ex, "\bot", "\mathfrak{s}_X"']
&&
\mathrm{sSet}_{\dslash X}\,.
\end{tikzcd}
\]
\end{lemma}
\begin{proof}
The first adjunction is a composite of $\mathrm{sSet}_\ast$-Quillen equivalences
\[
\begin{tikzcd}
\mathrm{sSet}_{\dslash X}
\ar[rr, shift left=1.1ex, "p(X)^\ast"]
\ar[rr, leftarrow, shift left=-1.1ex, "\bot", "p(X)_!"']
&&
 \mathbb{G}X\mathrm{-sSet}_{\dslash \mathbb{P}X}
\ar[rr, shift left=1.1ex, "\mathbb{P}X_!"]
\ar[rr, leftarrow, shift left=-1.1ex, "\bot", "\mathbb{P}X^\ast"']
&&  
\mathbb{G}X\mathrm{-sSet}_{\dslash \ast}\cong\mathbb{G}X_+\mathrm{-Mod}_{u}\,.
\end{tikzcd}
\]
The functor $p(X)^\ast$ sends a retractive space $(X\to Y\to X)$ to its pullback along the path fibration $p(X)\colon \mathbb{P}X\to X$;
the space $p(X)^\ast Y$ inherits a $\mathbb{G}X$-action such that $(\mathbb{P}X\to p(X)^\ast (Y)\to \mathbb{P}X)$ is a retract in $\mathbb{G}X\mathrm{-sSet}$.
 
The functor $p(X)^\ast$ is part of an adjoint equivalence of categories.
To prove this, first note that all objects of $\mathbb{G}X\mathrm{-sSet}_{\dslash \mathbb{P}X}$ are free $\mathbb{G}X$-spaces since $\mathbb{P}X$ is.
For any map of free $\mathbb{G}X$-spaces $P\to \mathbb{P}X$ (with or without section) there is a natural $\mathbb{G}X$-equivariant map $\eta_P\colon P\to p(X)^\ast(P/\mathbb{G}X)$.
The map $\eta_P$ becomes the identity after taking the quotient by $\mathbb{G}X$, so by freeness $\eta_P$ is an isomorphism of $\mathbb{G}X$-spaces.
On the other hand, for any map of simplicial sets $Y\to X$ there is a natural isomorphism $\epsilon_Y \colon p(X)^\ast Y / \mathbb{G}X \to Y$.
Since $\mathbb{P}X/\mathbb{G}X \cong X$, the functor $p(X)_!$ that takes quotients by $\mathbb{G}X$ determines an adjoint equivalence of categories $(p(X)^\ast\dashv p(X)_!)$.
The map $p(X)$ is a fibration, so that $p(X)^\ast$ preserves weak equivalences.
It is not hard to see that $p(X)^\ast$ also preserves cofibrations and $\mathrm{sSet}_\ast$-tensors, so that $(p(X)^\ast\dashv p(X)_!)$ is a $\mathrm{sSet}_\ast$-Quillen equivalence.
It follows that $(p(X)_!\dashv p(X)^\ast)$ is also a $\mathrm{sSet}_\ast$-Quillen equivalence.

A straightforward adaptation of the proof of Theorem \ref{thm:RetSpStructureTheorem} shows that $(\mathbb{P}X_!\dashv \mathbb{P}X^\ast)$ and $(\mathbb{P}X^\ast\dashv \mathbb{P}X_\ast)$ are $\mathrm{sSet}_\ast$-Quillen equivalences; crucial ingredients in this argument are the facts that $\mathbb{G}X\mathrm{-sSet}$ is proper and $p(X)\colon \mathbb{P}X\to X$ is a fibration (of $\mathbb{G}X$-spaces).
\end{proof}

Any simplicial set $X$ (not necessarily reduced) naturally determines a coalgebra in $\mathrm{sSet}$ with comultiplication given by the diagonal map $\Delta_X\colon X\to X\times X$.
After adjoining disjoint basepoints we obtain a coalgebra $X_+$ in $(\mathrm{sSet}_\ast, \wedge)$.
Hess and Shipley use the forgetful-cofree adjunction
\[
\begin{tikzcd}
X_+\mathrm{-Comod}_{\geq 0}
\ar[rr, shift left=1.1ex, "U"]
\ar[rr, leftarrow, shift left=-1.1ex, "\bot", "X_+\wedge (-)"']
&&
\mathrm{sSet}_\ast
\end{tikzcd}
\]
to equip the category of (left) $X_+$-comodules in $\mathrm{sSet}_\ast$ with a left proper combinatorial $\mathrm{sSet}_\ast$-model structure \cite[Theorem 4.1]{hess_waldhausen_2016}.
Cofibrations, weak equivalences, and $\mathrm{sSet}_\ast$-tensors of $X_+$-comodules are created by the forgetful functor $U$. 
Note that $X_+\mathrm{-Comod}_{\geq 0}$ is locally presentable as the comonad $X_+\wedge(-)$ is accessible \cite[Proposition A.1]{ching_coalgebraic_2014}.

\begin{remark}
\label{rem:InducedCoaction}
The pushforward $X_!\colon \mathrm{sSet}_{\dslash X}\to \mathrm{sSet}_\ast$ factors via the forgetful functor on $X_+$-comodules.
For a retractive space $Y\in \mathrm{sSet}_{\dslash X}$ with projection $p\colon Y\to X$, the map
\begin{align*}
\rho\colon X_! Y&\longrightarrow X_+\wedge X_!Y\\
[y] &\longmapsto p(y)\wedge [y]
\end{align*}
defines an $X_+$-coaction on $X_! Y$, as is readily checked.
\end{remark}
\begin{remark}
Fix a  field $k$.
The projection $p\colon Y\to X$ allows us to regard $H_\ast(Y;k)$ as a (left) $H_k(X;k)$-comodule; $H_\ast(\rho;k)$ then exhibits the $H_\ast(X;k)$-coaction on $\widetilde{H}_\ast(X_!Y;k)$ induced by the split cofibre sequence 
\[
0
\longrightarrow
H_p(X; k)
\longrightarrow
H_p(Y;k)
\longrightarrow
\widetilde{H}_p(X_!Y; k)
\longrightarrow 
0\,.
\]
In \cite{braunack-mayer_strict_2019}, we use this identification in the case $k=\mathbb{Q}$ to produce algebraic models of rational homotopy types of parametrised spectra.
\end{remark}

The functor $X_!\colon \mathrm{sSet}_{\dslash X}\to X_+\mathrm{-Comod}_{\geq 0}$ to $X_+$-comodules has a right adjoint $X\star (-)$.
Given an $X_+$-comodule $N$ with coaction $\rho\colon N\to X_+\wedge N$, the underlying object of $X\star N$ is the simplicial set
\[
X\star N := \mathrm{lim}\,
\left(
\!\begin{tikzcd}
  N \ar[r, "\rho"]
  &
  X_+\wedge N
  \ar[r, leftarrow]
  &
  X\times N
\end{tikzcd}\!
  \right)
\]
with $X\times N \to X_+\wedge N$ the quotient map.
Writing $n_0\in N$ for the basepoint, $X\star N$ is equipped with projection $X\star N\to X\times N \to X$ and section $x\mapsto (x,n_0)$. 
With these maps, $X\star N$ can be regarded as an object of $\mathrm{sSet}_{\dslash X}$.
\begin{lemma}
\label{lem:RetSpandComod}
For each simplicial set $X$, there is a $\mathrm{sSet}_\ast$-Quillen adjunction
\[
\begin{tikzcd}
\mathrm{sSet}_{\dslash X}
\ar[rr, shift left=1.1ex, "X_!"]
\ar[rr, leftarrow, shift left=-1.1ex, "\bot", "X\star(-)"']
&&
X_+\mathrm{-Comod}_{\geq 0}\,.
\end{tikzcd}
\]
\end{lemma}
\begin{proof}
It is not difficult to establish the adjunction directly; the counit $X_!(X\star N)\to N$ sends $[(x,n)]\mapsto n$ and the unit $Y\to X\star (X_! Y)$ sends $y\mapsto (p(x), [y])$.
The left adjoint $X_!$ is Quillen: it preserves cofibrations, weak equivalences and $\mathrm{sSet}_\ast$-tensors as these are created in $X_+\mathrm{-Comod}$ by the forgetful functor to $\mathrm{sSet}_\ast$.
\end{proof}

\begin{theorem}[Koszul pre-duality]
\label{thm:ProtoKoszul}
For each reduced simplicial set $X$, there is a $\mathrm{sSet}_\ast$-Quillen adjunction
\[
\begin{tikzcd}
 \mathbb{G}X_+\mathrm{-Mod}_{u}
\ar[rr, shift left=1.1ex, "\mathfrak{K}_X^\ast"]
\ar[rr, leftarrow, shift left=-1.1ex, "\bot", "\mathfrak{K}^X_\ast"']
&&
X_+\mathrm{-Comod}_{\geq 0}\,.
\end{tikzcd}
\]
\end{theorem}
\begin{proof}
Combine the adjunctions of Lemmas \ref{lem:RetSpandLoopSp} and \ref{lem:RetSpandComod}.
\end{proof}

Any map of simplicial sets $f\colon X\to X'$ induces a $\mathrm{sSet}_\ast$-Quillen adjunction
\[
\begin{tikzcd}
X_+\mathrm{-Comod}_{\geq 0}
\ar[rr, shift left=1.1ex, "f_!"]
\ar[rr, leftarrow, shift left=-1.1ex, "\bot", "f^\ast"']
&&
X'_+\mathrm{-Comod}_{\geq 0}\,.
\end{tikzcd}
\]
For an $X_+$-comodule $N$ we set $f_! N = N$ regarded as an  $X'_+$-comodule in the obvious way.
The right adjoint $f^\ast$ sends the $X'_+$-comodule $(M,\rho)$ to the pullback of the cospan
\[
\begin{tikzcd}
  M
  \ar[r, "\rho"]
  &
  X'_+\wedge M
  \ar[r, leftarrow, "{f_+\wedge \mathrm{id}_M}"]
  &
  X_+\wedge M\,,
\end{tikzcd}
\]
which is canonically a left $X_+$-comodule.
Equivalently, $f^\ast NM= X_+\square_{X'_+}M$ is the cotensor product of the $X_+'$-comodules $X_+$ and $M$.
Since cofibrations and weak equivalences are created in $\mathrm{sSet}_\ast$ in both cases, the functor $f_!$ is easily seen to be left Quillen.

Similarly, for a map of \emph{reduced} simplicial sets $f\colon X\to X'$ there is a corresponding map of simplicial groups $\mathbb{G}f\colon \mathbb{G}X\to \mathbb{G}X'$.
There is an induced $\mathrm{sSet}_\ast$-Quillen adjunction
\[
\begin{tikzcd}
\mathbb{G}X_+\mathrm{-Mod}_{u}
\ar[rr, shift left=1.1ex, "\mathbb{G}f_!"]
\ar[rr, leftarrow, shift left=-1.1ex, "\bot", "\mathbb{G}f^\ast"']
&&
\mathbb{G}X'_+\mathrm{-Mod}_{u}\,,
\end{tikzcd}
\]
which is moreover a $\mathrm{sSet}_\ast$-Quillen equivalence if $f$ is a weak equivalence.
It is left to the reader to verify the following
\begin{lemma}
Let $f\colon X\to X'$ be a map of reduced simplicial sets. 
Then there is a natural isomorphism of left $\mathrm{sSet}_\ast$-Quillen functors $\mathfrak{K}^\ast_{X'} \circ \mathbb{G}f_! \cong f_!\circ \mathfrak{K}^\ast_{X}$.
\end{lemma}

\begin{remark}
Let $\mathcal{E}$ be a generalised reduced homology theory.
Hess and Shipley show that $\mathbb{G}X_+\mathrm{-Mod}_{u}$ and $X_+\mathrm{-Comod}_{\geq 0}$ support $\mathcal{E}$-local model structures, and that Koszul pre-duality becomes a Quillen equivalence after $\mathcal{E}$-localisation.
In the particular case $\mathcal{E}= H\mathbb{Q}$ we obtain an equivalence between the rational homotopy theories of connective $\mathbb{G}X_+$-modules and connective $X_+$-comodules
\begin{equation}
\label{eqn:SpaceHQKoszul}
\begin{tikzcd}
\big(\mathbb{G}X_+\mathrm{-Mod}_{u}\big)_{H\mathbb{Q}}
\ar[rr, shift left=1.1ex, "\mathfrak{K}_X^\ast"]
\ar[rr, leftarrow, shift left=-1.1ex, "\bot", "\mathfrak{K}^X_\ast"']
&&
\big(X_+\mathrm{-Comod}_{\geq 0}\big)_{H\mathbb{Q}}\,.
\end{tikzcd}
\end{equation}
When $X$ is moreover simply connected, we can use Quillen's rational homotopy theory \cite{quillen_rational_1969} to produce strict dg models $C_X$ and $\lambda_X$ for the rational chains $H_\ast(X;\mathbb{Q})$ and Whitehead Lie algebra $\mathcal{P}H_\ast(\Omega X;\mathbb{Q})$ of $X$ respectively.
In this setting we have a more familiar Koszul duality Quillen equivalence
\begin{equation}
\label{eqn:AlgKoszul}
\begin{tikzcd}
\lambda_X \mathrm{-Rep}_{\geq 0}
\ar[rr, shift left=1.1ex, "t^!"]
\ar[rr, leftarrow, shift left=-1.1ex, "\bot", "t_\ast"']
&&
C_X\mathrm{-Comod}_{\geq 0}
\end{tikzcd}
\end{equation}
given by forming twisted extensions with respect to a canonical twisting chain $t\colon C_X
\rightsquigarrow \mathcal{U}\lambda_X$.
The derived equivalences of \eqref{eqn:SpaceHQKoszul} and \eqref{eqn:AlgKoszul} are identified via the rational homotopy theory of parametrised spectra \cite{braunack-mayer_strict_2019}.
\end{remark}

\section{Parametrised spectra}
\label{S:ParamSpec}
In this section, we construct several different combinatorial model categories of parametrised spectra.
Working initially over a fixed parameter space $X$, we construct models for $X$-parametrised stable homotopy theory in terms of sequential and symmetric spectrum objects following Hovey \cite{hovey_spectra_2001}.
Though they both present the $\infty$-category of $X$-parametrised spectra, these two alternative models have different mathematical properties, affording a degree of flexibility in applications (for example in \cite{braunack-mayer_strict_2019, braunack-mayer_strict_2019-1} and Section \ref{SS:TwistDiffCoh}).
As in the unstable case, there are base change Quillen adjunctions relating model categories of spectra parametrised over different bases.

As in the unstable case, the various model categories of spectra parametrised over a fixed base glue together via the Grothendieck construction.
The result is a global model category of parametrised spectra, encoding the homotopy theory of parametrised spectra over all possible base spaces.
We construct two Quillen equivalent versions in terms of sequential and symmetric spectra, and show that both model categories are left proper, combinatorial and simplicial.
Moreover, we show that the global model category of parametrised symmetric spectra is equipped with a symmetric monoidal model structure via the external smash product, generalising the smash product of symmetric spectra to the parametrised setting.

\subsection{Local theory}
For a fixed simplicial set $X$, we produce two  model categories presenting the homotopy theory of $X$-parametrised spectra by using the sequential and symmetric stabilisation machines studied in \cite{hovey_spectra_2001}.
In either case, we produce $X$-parametrised spectra as a means of homotopically inverting the fibrewise suspension endofunctor 
\[
\Sigma_X \colon Y \longmapsto S^1\owedge_X Y = (X^\ast S^1)\wedge_X Y
\]
on $\mathrm{sSet}_{\dslash X}$,  
where $S^1 := \Delta^1/\partial\Delta^1$ is pointed by its unique $0$-simplex.
The nomenclature is motivated by the fact that $x^\ast \Sigma_XY\cong \Sigma (x^\ast Y)$ for all points $x\colon \ast \to X$, since pullback functors preserve $\mathrm{sSet}_\ast$-tensors.
The operation of forming $S^1$-cotensors 
\[
\Omega_X \colon Y\longmapsto F_X(X^\ast S^1, Y)
\]
defines a right adjoint to $\Sigma_X$.
All pullback functors are strongly closed, so for all points $x$ we have $x^\ast \Omega_X Y \cong \Omega (x^\ast Y)$.
Finally, since $S^1$ is a cofibrant object of $\mathrm{sSet}_\ast$ it follows from pushout-product axiom for $\mathrm{sSet}_\ast$-tensors that the adjunction
\[
\begin{tikzcd}
\mathrm{sSet}_{\dslash X}
\ar[rr, shift left=1.1ex, "\Sigma_X"]
\ar[rr, leftarrow, shift left=-1.1ex, "\bot", "\Omega_X"']
&&
\mathrm{sSet}_{\dslash X}
\end{tikzcd}
\]
is Quillen.
\begin{remark}
\label{rem:SuspensionEndoonRetSpaces}
The endofunctor $\Sigma_X$ models suspension in $Ho(\mathrm{sSet}_{\dslash X})$.
Since $S^1 = \Delta^1_+/\partial\Delta^1_+$, we can write $\Sigma_X Y$ as the pushout of the span $Y \leftarrow \partial\Delta^1_+ \owedge_X Y \to \Delta^1_+ \owedge_X Y$.
But $\partial\Delta^1_+\owedge_X Y \cong Y\coprod Y$ and $\Delta^1_+\owedge_X Y$ is a very good cylinder object for $Y$, so this pushout presents the suspension of $Y$ in $Ho(\mathrm{sSet}_{\dslash X})$ (cf.~\cite[Chapter 6]{hovey_model_1999}).
\end{remark}

\subsubsection{Sequential stabilisation}
\label{sss:SequentialStabFibrewise}
We apply the sequential stabilisation machine to pass from retractive spaces to parametrised spectra.
The main characters in this part of the story are parametrised sequential spectra:
\begin{definition}
A \emph{sequential $X$-spectrum} is a sequence $A=\{A_n\}_{n\in \mathbb{N}}$ of retractive spaces over $X$ together with a collection of structure maps $\sigma^A_n \colon \Sigma_X A_n\to A_{n+1}$ for each $n\geq 0$. 
A morphism of sequential $X$-spectra $f\colon A\to B$ is a sequence of maps $f_n \colon A_n\to B_n$ such that the diagram
\[
\begin{tikzcd}
  \Sigma_X A_n \ar[r,"\Sigma_X f_n"]
  \ar[d, "\sigma^A_n"']
  &
  \Sigma_X B_n
  \ar[d, "\sigma^B_n"]
  \\
  A_{n+1}
  \ar[r, "f_{n+1}"]
  &
  B_{n+1}
\end{tikzcd}
\]
commutes for each $n\geq 0$.
The category of sequential $X$-spectra is denoted $\mathrm{Sp}^\mathbb{N}_X$.
\end{definition}

\begin{remark}
\label{rem:FibrewiseSeqSpecMonad}
Viewing $\mathbb{N}$ as a discrete category, we may characterise sequential $X$-spectra as algebras over a monad on $\mathrm{Fun}(\mathbb{N},\mathrm{sSet}_{\dslash X})$.
Let $T_X^\mathrm{sp}$ be the endofunctor that sends the sequence $Z:=\{Z_n\}$ to the sequence with $n$-th term
\[
\coprod_{i=0}^n \Sigma^{n-i} Z_i\,.
\]
There is a multiplication transformation $\mu\colon T_X^\mathrm{sp}T_X^\mathrm{sp}\Rightarrow T_X^\mathrm{sp}$ determined by the folding maps
\[
T_X^\mathrm{sp}T_X^\mathrm{sp}(Z)_n =\coprod_{i=0}^n \left(\Sigma^{n-i} Z_i\right)^{\coprod(n+1-i)}
\longrightarrow\coprod_{i=0}^n \Sigma^{n-i} Z_i\,,
\]
and a unit transformation $\eta\colon \mathrm{id}\Rightarrow T_X^\mathrm{sp}$ induced by the coprojection maps $Z_n\to \coprod_{i=0}^n \Sigma^{n-i} Z_i$.
The triple\footnote{\lq\lq Spectral triple''.} $(T_X^\mathrm{sp},\mu,\eta)$ defines a monad on $\mathrm{Fun}(\mathbb{N},\mathrm{sSet}_{\dslash X})$.
For  a $T^\mathrm{sp}$-algebra $A$, the algebra structure map $T^\mathrm{sp}A_n \to A_n$ restricts to a structure map $\sigma_{n-1}\colon \Sigma_X A_{n-1}\to A_n$ on the $(i=n-1)$-th summand.
Via this identification there is a canonical isomorphism of categories between $\mathrm{Sp}^\mathbb{N}_X$ and $T_X^\mathrm{sp}\mathrm{-Alg}$.

The monad $T^\mathrm{sp}$ preserves all colimits, so that that $\mathrm{Sp}^\mathrm{N}_X$ is locally presentable.
Limits and colimits are created by the forgetful functor $\mathrm{Sp}^\mathbb{N}_X\to \mathrm{Fun}(\mathbb{N},\mathrm{sSet}_{\dslash X})$.
\end{remark}

\begin{remark}
\label{rem:FreeSeqSpec}
For each $k\geq 0$, the evaluation functor on sequences $\mathrm{ev}_k \colon \{Z_n\}\mapsto Z_k$ has a left adjoint $i_k$.
For a retractive space $Y\in \mathrm{sSet}_{\dslash X}$, the sequence $i_k(Y)$ has $Y$ as its $k$-th term, with all other terms in the sequence equal to the zero object $0_X$. 
Concatenating with the free-forgetful adjunction of the monad $T^\mathrm{sp}_X$ (Remark \ref{rem:FibrewiseSeqSpecMonad}), we obtain an adjunction
\[
(\Sigma^{\infty-k}_X\dashv \widetilde{\Omega}^{\infty-k}_X)\colon
\begin{tikzcd}
 \mathrm{sSet}_{\dslash X}
\ar[rr, shift left=1.1ex, "i_k"]
\ar[rr, leftarrow, shift left=-1.1ex, "\bot", "\mathrm{ev}_k"']
&&
\mathrm{Fun}(\mathbb{N}, \mathrm{sSet}_{\dslash X})
\ar[rr, shift left=1.1ex, "T^\mathrm{sp}_X"]
\ar[rr, leftarrow, shift left=-1.1ex, "\bot", ""']
&&
\mathrm{Sp}^\mathbb{N}_{X}\,.
\end{tikzcd}
\]
The left adjoint $\Sigma^{\infty-k}_X$ sends a retractive space $Y$ to the sequential $X$-spectrum freely generated by $Y$ at level $k$
\[
\Sigma^{\infty-k}_X Y_n =
\begin{cases}
\;\;\; 0_X & n<k\\
\Sigma^{n-k}_X Y & n\geq k\,,
\end{cases}
\]
equipped with the obvious structure maps.
The right adjoint $\widetilde{\Omega}_X^{\infty-k}$ evaluates the $k$-th term of a sequential $X$-spectrum.
A sequential $X$-spectrum of the form $\Sigma^\infty_X Y$ (or~$\Sigma^{\infty-k}_X Y$ for $k>0$) is called a \emph{fibrewise suspension spectrum} (respectively, a \emph{shifted fibrewise suspension spectrum}).

We reserve the notation $\Omega^{\infty-k}_X$ for the right derived functor $\mathbf{R}\widetilde{\Omega}^{\infty-k}_X$ with respect to the sequential $X$-stable model structure (discussed below).
We do not distinguish between $\Sigma^{\infty-k}_X$ and its derived functor in the notation; all objects of $\mathrm{sSet}_{\dslash X}$ are cofibrant, so $\Sigma^{\infty-k}_X$ is already derived.
\end{remark}

The free-forgetful adjunction of the monad $T^\mathrm{sp}$ induces a left proper combinatorial $\mathrm{sSet}_\ast$-model structure on $\mathrm{Sp}^\mathbb{N}_{X}$, which we call the \emph{sequential $X$-projective model structure}.
Weak equivalences and fibrations for the sequential $X$-projective model structure are determined levelwise by the underlying  maps of sequences.
The $\mathrm{sSet}_\ast$-tensors are computed objectwise, so that for sequential $X$-spectrum $Z$ and pointed simplicial set $K$ we have $(K\owedge_X Z)_n = K\owedge_X Z_n$ with structure maps defined using the symmetry isomorphism $K\wedge S^1\cong S^1 \wedge K$ in $\mathrm{sSet}_\ast$.
The sets
\[
\mathcal{I}^{\mathbb{N}\mathrm{-proj}}_X
:= \bigcup_{k\in \mathbb{N}} \Sigma^{\infty-k}_X (\mathcal{I}^\mathrm{Kan}_X)
\qquad
\mbox{ and }
\qquad
\mathcal{J}^{\mathbb{N}\mathrm{-proj}}_X
:= \bigcup_{k\in \mathbb{N}} \Sigma^{\infty-k}_X (\mathcal{J}^\mathrm{Kan}_X)
\]
are respectively sets of generating cofibrations and generating acyclic cofibrations for the sequential $X$-projective model structure. 
\begin{remark}
As for retractive spaces, the endofunctor $S^1\owedge_X(-)$ models the suspension endofunctor on the sequential $X$-projective homotopy category (cf.~Remark \ref{rem:SuspensionEndoonRetSpaces}).
\end{remark}
\begin{remark}
\label{rem:ProjCofib}
The cofibrations for the sequential $X$-projective model structure are precisely the maps of sequential $X$-spectra $\varphi\colon A\to B$ for which $\varphi_0 \colon A_0 \to B_0$ and 
$  A_{n+1} \coprod_{\Sigma_X A_{n}} \Sigma_X B_{n}\longrightarrow B_{n+1}
$,
$n\geq 0$, are cofibrations in $\mathrm{sSet}_{\dslash X}$ (cf.~\cite[Proposition 1.14]{hovey_spectra_2001}).
\end{remark}

The suspension endofunctor $\Sigma_X$ on $\mathrm{sSet}_{\dslash X}$ prolongs to $\mathrm{Sp}^\mathbb{N}_X$.
For a sequential $X$-spectrum $Z = \{Z_n\}$, we set $(\Sigma_X Z)_n := \Sigma_X Z_n$ with structure maps
$
\Sigma_X \sigma_n\colon \Sigma^2_X Z_n \to \Sigma_X Z_{n+1} 
$.
The functor $\Sigma_X$ on $\mathrm{Sp}^\mathbb{N}_X$ has right adjoint $\Omega_X$,  where $(\Omega_X Z)_n := \Omega_X Z_n$ equipped with structure maps adjunct to $\Omega_X \sigma^\vee_n \colon\Omega_X Z_n \to \Omega^2_X Z_{n+1}$.
The resulting adjunction
\[
(\Sigma_X \dashv \Omega_X)
\colon 
\begin{tikzcd}
\mathrm{Sp}^\mathbb{N}_X
\ar[rr, shift left=1.1ex, ""]
\ar[rr, leftarrow, shift left=-1.1ex, "\bot", ""']
&&
\mathrm{Sp}^\mathbb{N}_X
\end{tikzcd}
\]
is Quillen with respect to the sequential $X$-projective model structure since $\Omega_X$ preserves levelwise fibrations and acyclic fibrations.
\begin{remark}
\label{rem:SmashvSuspendProlong}
For a sequential $X$-spectrum $Z$, the structure maps of $S^1\owedge_X Z$ and $\Sigma_X Z$ differ by the twist isomorphism $S^1\wedge S^1 \to S^1\wedge S^1$ that interchanges smash factors.
\end{remark}

In order to invert $\Sigma_X$ up to homotopy we perform left Bousfield localisation at the set 
\[
S_{\mathbb{N},X} :=\big\{\zeta_{k,X} (C)\colon \Sigma^{\infty-(k+1)}_X (\Sigma_X C)\to \Sigma^{\infty-k}_{X} (C) \big\}\,,
\]
where  $\zeta_{k,X}(Z)$ is the $(\Sigma^{\infty-(k+1)}_X\dashv \widetilde{\Omega}^{\infty-(k+1)}_X)$-adjunct of $\mathrm{id}\colon \Sigma_X Z \to \Sigma_XZ = \widetilde{\Omega}^{\infty-(k+1)}_X \Sigma^{\infty-k}_X Z$,  $C$ ranges over domains and codomains of morphisms in $\mathcal{I}^\mathrm{Kan}_X$, and $k\geq 0$.
The left Bousfield localisation at $S_{\mathbb{N},X}$ exists and is left proper, combinatorial and simplicial \cite[Proposition A.3.7.3]{lurie_higher_2009}.
The localised model structure is called the \emph{sequential $X$-stable model structure}.
Fibrations and weak equivalences in the sequential $X$-stable model structure are referred to as stable fibrations and stable weak equivalences respectively.

\begin{remark}
\label{rem:ShiftFunctors}
The \emph{shift functor} $s_X\colon \mathrm{Sp}^\mathbb{N}_X\to \mathrm{Sp}^\mathbb{N}_X$ is defined by $s_X(Z)_n := Z_{n+1}$, equipped with the obvious structure maps.
The functor $s_X$ has a left adjoint $\ell_X(Z)_n := Z_{n-1}$ (where $\ell_X(Z)_0 := 0_X$) and $(\ell_X\dashv s_X)$ is a $\mathrm{sSet}_\ast$-Quillen adjunction with respect to the sequential $X$-projective model structure.
Since $\ell_X(S_{\mathbb{N},X} )\subset S_{\mathbb{N},X}$, the adjunction is Quillen for the localised model structure as well.
\end{remark}

We now record some important properties of the sequential $X$-stable model structure. 
The following three results are special cases of the more general statements of \cite[Section 3]{hovey_spectra_2001}; we record the proofs here for the sake of completeness and as a warm-up for the global setting.
\begin{lemma}
\label{lem:SeqSpecFib}
Stably fibrant objects of $\mathrm{Sp}^\mathbb{N}_X$ are precisely the \emph{fibrant $\Omega_X$-spectra}, namely the sequential $X$-spectra $A$ such that each $A_n\in \mathrm{sSet}_{\dslash X}$ is fibrant and the adjoint structure maps $\sigma^\vee_n\colon A_n\to \Omega_X A_{n+1}$ are weak equivalences for all $n\geq 0$.
The stable weak equivalences between fibrant $\Omega_X$-spectra are precisely the levelwise weak equivalences.
\end{lemma}
\begin{proof}
The stably fibrant objects are precisely the $S_{\mathbb{N},X}$-local fibrant objects, so that $A$ is stably fibrant if and only if the map of homotopy function complexes
\[
\mathrm{map}(C, \sigma^\vee_n) \colon \mathrm{map}(C, A_n)\longrightarrow \mathrm{map}(C, \Omega_X A_{n+1})
\]
is a weak equivalence for all $n\geq 0$, with $C$ ranging over domain and codomains of morphisms in $\mathcal{I}^\mathrm{Kan}_X$.
A formal argument on cell attachments shows that this is the case precisely if $\sigma^\vee_n$ is a weak equivalence \cite[Proposition 3.2]{hovey_spectra_2001}.
To complete the proof, we remark that by general properties of Bousfield localisation, the stable weak equivalences between stably fibrant objects are precisely the projective, hence levelwise, weak equivalences.
\end{proof}
\begin{lemma}
\label{lem:CofibWE}
For each $Y\in \mathrm{sSet}_{\dslash X}$, the map $\zeta_{k,X}(Y)$ is a stable weak equivalence for all $k\geq 0$.
\end{lemma}
\begin{proof}
The map $\zeta_{k,X}(Y)$ is a stable weak equivalence precisely if $\mathrm{map}(\zeta_{k,X}(Y), A)$ is a weak equivalence for all stably fibrant $A$.
By adjointness, $\mathrm{map}(\zeta_{k,X}(Y), A)$ is equivalent to $\mathrm{map}(Y, \sigma_k^\vee)$, where $\sigma^\vee_k\colon A_k \to \Omega_X A_{k+1}$ is the adjoint spectrum structure map.
But this latter map is an equivalence by the previous result.
Note that we have used that all objects of $\mathrm{sSet}_{\dslash X}$ are cofibrant.
\end{proof}
\begin{lemma}
\label{lem:SeqSpecShifts}
The adjunction $(\Sigma_X\dashv \Omega_X)\colon \mathrm{Sp}^\mathbb{N}_X\to \mathrm{Sp}^\mathbb{N}_X$ is a Quillen equivalence.
Moreover, there are natural isomorphisms $\mathbf{L}\Sigma_X\cong \mathbf{R}s_X$ and $\mathbf{R}\Omega_X\cong \mathbf{L}\ell_X$.
\end{lemma}
\begin{proof}
There is a natural transformation $\mathrm{id}\Rightarrow s_X\Omega_X$.
For a fibrant $\Omega_X$-spectrum $A$, the natural map $A\to s_X\Omega_X A$ is a stable weak equivalence.
Since $s_X$ and $\Omega_X$ commute with each other, there are natural isomorphisms $\mathrm{id} \cong \mathbf{R}(s_X \Omega_X)\cong \mathbf{R}s_X\circ \mathbf{R}\Omega_X\cong \mathbf{R}\Omega_X\circ \mathbf{R}s_X$.
In particular, $\mathbf{R}s_X$ and $\mathbf{R}\Omega_X$ are adjoint equivalences of categories so that $s_X$ and $\Omega_X$ are both right Quillen equivalences.
The result follows by invoking essential uniqueness of adjoints.
\end{proof}

Note that the last result does not yet imply that the sequential $X$-stable model structure is stable (in the sense that suspension is an equivalence on $Ho(\mathrm{Sp}^\mathbb{N}_X)$), since $\Sigma_X$ and $S^1\owedge_X(-)$ differ by the twist map (Remark \ref{rem:SmashvSuspendProlong}).
A result of Hovey \cite[Theorem 10.3]{hovey_spectra_2001} allows us to compare the twisted and untwisted suspensions functors, and we find that the corresponding derived functors are isomorphic provided $S^1$ is \emph{symmetric}, meaning  there is a simplicial homotopy
\[
S^1\wedge S^1 \wedge S^1 \wedge I_+\longrightarrow S^1\wedge S^1 \wedge S^1
\]
between the identity and the cyclic permutation of smash factors, where $I_+$ is a cylinder object of $S^0$.
While these maps are certainly homotopic (both are self-maps of $S^3$ of index $+1$), we have been unable to produce a simplicial homotopy of the prescribed form.
Nevertheless, the model categories $\mathrm{Sp}^\mathbb{N}_X$ are stable (Lemma \ref{lem:SeqSpecisStable}), which fact we show using the base change adjunctions:
\begin{theorem}
\label{thm:SeqSpecStructureTheorem}
For any map of simplicial sets $f\colon X\to X'$ there is an adjoint triple of $\mathrm{sSet}_\ast$-functors
\[
(f_!\dashv f^\ast\dashv f_\ast)\colon
\begin{tikzcd}
\mathrm{Sp}^\mathbb{N}_X
\ar[rr, shift left=2.2ex, ""]
\ar[rr, leftarrow, "\bot"]
\ar[rr,  shift left=-2.2ex, "\bot", ""']
&&
\mathrm{Sp}^\mathbb{N}_{X'}\,.
\end{tikzcd}
\]
With respect to both the projective and  stable model structures, $(f_!\dashv f^\ast)$ is a $\mathrm{sSet}_\ast$-Quillen adjunction that is moreover a $\mathrm{sSet}_\ast$-Quillen equivalence if $f$ is a weak equivalence.
If $f$ is a fibration (or a projection to a factor of a product) then $(f^\ast \dashv f_\ast)$ is a $\mathrm{sSet}_\ast$-Quillen adjunction with respect to either model structure, and is a $\mathrm{sSet}_\ast$-Quillen equivalence if $f$ is an acyclic fibration.
\end{theorem} 
\begin{proof}
The functors $f_!$, $f^\ast$ and $f_\ast$ are given by levelwise application of the respective functors on underlying sequences of retractive spaces.
For $f_!$ and $f^\ast$ this is straightforward as both functors preserve $\mathrm{sSet}_\ast$-tensors.
As for the remaining functor, first note that the natural isomorphism $\Sigma_X f^\ast \cong f^\ast \Sigma_{X'}$ implies a natural isomorphism of right adjoints $f_\ast \Omega_X \cong \Omega_{X'} f_\ast$
Given a a sequential $X$-spectrum $Z$, we set $(f_\ast Z)_n := f_\ast Z_n$ with structure maps adjunct to 
\[
\begin{tikzcd}
  f_\ast Z_n 
  \ar[r, "f_\ast \sigma_n^\vee"]
  &
  f_\ast(
  \Omega_X
  Z_{n+1}
  )
  \cong \Omega_{X'}f_\ast Z_{n+1}\,.
\end{tikzcd}
\]
The results for sequential projective model structures now follow from Theorem \ref{thm:RetSpStructureTheorem}.

Observe that $f_! S_{\mathbb{N},X} \subset S_{\mathbb{N}, X'}$, so that the adjunction $(f_!\dashv f^\ast)$ descends to the localised model structures.
Suppose that $f$ is a weak equivalence, so that $(f_!\dashv f^\ast)$ is a Quillen equivalence of the sequential projective model structures.
To show that $(f_!\dashv f^\ast)$ is a Quillen equivalence with respect to the sequential stable model structures, it is sufficient to show that $f_!$ reflects stable weak equivalences between cofibrant objects and that the derived counit is a natural weak equivalence for all stably fibrant objects.
Let $\mathcal{Q}$ and $\mathcal{R}$ be cofibrant and fibrant replacement functors for the sequential $X$-projective and sequential $X'$-projective model structures respectively. 
The derived counit $f_!\mathcal{Q}f^\ast A\to A$ is thus a levelwise weak equivalence for all fibrant $\Omega_X$-spectra $A\in \mathrm{Sp}^\mathbb{N}_X$.

Now let $\varphi\colon Y\to Z$ be a map of cofibrant $X$-spectra such that $f_!(\varphi)$ is a stable weak equivalence.
Then for all fibrant $\Omega_X$-spectra $A$, the induced map of derived hom-spaces $\mathrm{map}(f_!(\varphi), A)$ is a weak equivalence by Lemma \ref{lem:SeqSpecFib}.
By adjointness, $\mathrm{map}(\varphi, f^\ast A)$ is a weak equivalence for all stably fibrant $A\in \mathrm{Sp}^\mathbb{N}_{X'}$.
We conclude that that $\varphi$ is a stable weak equivalence, since every fibrant $\Omega_X$-spectrum is levelwise weakly equivalent to some $f^\ast A$.
Indeed, suppose that $P$ is a fibrant $\Omega_X$-spectrum so that $\mathcal{Q}P$, too, is stably fibrant.
Then the derived unit $\mathcal{Q}P\to f^\ast \mathcal{R} f_! \mathcal{Q}P$ is a levelwise weak equivalence.
The right Quillen equivalence $f^\ast \colon \mathrm{sSet}_{\dslash X'}\to \mathrm{sSet}_{\dslash X}$ reflects weak equivalences between fibrant objects and we have a natural isomorphism $\Omega_X f^\ast \cong f^\ast\Omega_{X'}$, from which it follows that $\mathcal{R} f_! \mathcal{Q}P$ is a fibrant $\Omega_{X'}$-spectrum and the claim is proven.

Finally, we note that $f^\ast$ sends all morphisms in $S_{\mathbb{N}, X'}$ to stable weak equivalences by Lemma \ref{lem:CofibWE} so that the Quillen adjunction $(f^\ast\dashv f_\ast)$ descends to the localised model structures if $f$ is a fibration.
If $f$ is moreover an acyclic fibration, then there is an isomorphism of derived functors $\mathbf{R}f^\ast\cong \mathbf{L}f^\ast$.
Since $f$ is a weak equivalence, $(\mathbf{L}f_!\dashv \mathbf{R}f^\ast)$ is an adjoint equivalence of categories, so that there is a natural isomorphism $\mathbf{R}f_\ast \cong \mathbf{L}f_!$ by uniqueness of adjoints.
In particular, $(\mathbf{L}f^\ast\dashv \mathbf{R}f_\ast)$ is an adjoint equivalence of categories.
\end{proof}
\begin{remark}[Stable Koszul pre-duality]
\label{rem:StabProtoKoszul}
For each simplicial set $X$, the theorem above implies that we have a $\mathrm{sSet}_\ast$-Quillen adjunction
\[
\begin{tikzcd}
\mathrm{Sp}^\mathbb{N}_{X}
\ar[rr, shift left=1.1ex, "X_!"]
\ar[rr, leftarrow, shift left=-1.1ex, "\bot", "X\overline{\star}(-)"']
&&
X_+\mathrm{-Comod}^\mathbb{N}\,,
\end{tikzcd}
\]
with $X_+\mathrm{-Comod}^\mathbb{N}$ the category of coalgebras in $\mathrm{Sp}^\mathbb{N}$ for the comonad $X_+\wedge(-)$ arising from the $\mathrm{sSet}_{\ast}$-tensoring.
Cofibrations and weak equivalences in $X_+\mathrm{-Comod}^\mathbb{N}$ are created by the forgetful functor to $\mathrm{Sp}^\mathbb{N}$.
In the theorem above, we saw that the functor $X_!\colon \mathrm{Sp}^\mathbb{N}_X\to \mathrm{Sp}^\mathbb{N}$ is defined by levelwise application of $X_!\colon \mathrm{sSet}_{\dslash X}\to \mathrm{sSet}_\ast$.
The symmetry isomorphisms of the smash product and Remark \ref{rem:InducedCoaction} together imply a factorisation
\[
X_!\colon
\begin{tikzcd}
  \mathrm{Sp}^\mathbb{N}_X
  \ar[r, "X_!"]
  &
  X_+\mathrm{-Comod}^\mathbb{N}
  \ar[r]
  &
  \mathrm{Sp}^\mathbb{N}\,,
\end{tikzcd}
\]
where we denote the functor from sequential $X$-spectra to $X_+$-comodule spectra by $X_!$ in a mild abuse of notation.
This functor preserves colimits, so has a right adjoint by the adjoint functor theorem. 
Since $X_!$ also preserves cofibrations, acyclic cofibrations and $\mathrm{sSet}_\ast$-tensors it is left Quillen.

If $X$ is a reduced simplicial set, then a similar argument applied to the $\mathrm{sSet}_\ast$-Quillen equivalence $(\mathfrak{b}_X\dashv \mathfrak{s}_X)$ of Lemma \ref{lem:RetSpandLoopSp} provides a $\mathrm{sSet}_\ast$-Quillen adjunction
\[
\begin{tikzcd}
\mathbb{G}X_+\mathrm{-Mod}^\mathbb{N}
\ar[rr, shift left=1.1ex, "\mathfrak{b}_X"]
\ar[rr, leftarrow, shift left=-1.1ex, "\bot", "\mathfrak{s}_X"']
&&
\mathrm{Sp}^\mathbb{N}_X\,,
\end{tikzcd}
\]
with $\mathbb{G}X_+\mathrm{-Mod}^\mathbb{N}$ the category of algebras for the monad $\mathbb{G}X_+\wedge(-)$ on $\mathrm{Sp}^\mathbb{N}_X$.
By a similar argument to that of the theorem above, the stabilised adjunction $(\mathfrak{b}_X\dashv \mathfrak{s}_X)\colon \mathbb{G}X_+\mathrm{-Mod}^\mathbb{N}\to \mathrm{Sp}^\mathbb{N}_X$ is a $\mathrm{sSet}_\ast$-Quillen equivalence.
Thus for any reduced simplicial set $X$ there is a composite $\mathrm{sSet}_\ast$-Quillen adjunction
\[
\begin{tikzcd}
\mathbb{G}X_+\mathrm{-Mod}^\mathbb{N}
\ar[rr, shift left=1.1ex, "\mathfrak{K}_X^\ast"]
\ar[rr, leftarrow, shift left=-1.1ex, "\bot", "{\mathfrak{K}^X_\ast}"']
&&
X_+\mathrm{-Comod}^\mathbb{N}\,,
\end{tikzcd}
\]
stabilising the adjunction of Theorem \ref{thm:ProtoKoszul}.
This is the sequential spectrum version of the stable Koszul pre-duality adjunction derived by Hess and Shipley (cf.~\cite[Theorem 5.4]{hess_waldhausen_2016}).
\end{remark}

In the classical unparametrised setting, stable weak equivalences of spectra are detected by stable homotopy groups.
That is, a map of (sequential) spectra $\psi\colon E\to F$ is a stable weak equivalence precisely if the induced map on stable homotopy groups is an isomorphism.
The notion corresponding to $\pi^\mathrm{st}_\ast$-isomorphism in the parametrised setting is provided by the following
\begin{definition}
Fix a fibrant replacement functor $\mathcal{R}$ for the sequential $X$-stable model structure.
A map of sequential $X$-spectra $\psi\colon A\to B$ is a \emph{fibrewise $\pi^\mathrm{st}_\ast$-isomorphism} if $x^\ast\mathcal{R}\psi\colon x^\ast \mathcal{R}A\to x^\ast \mathcal{R}B$ is a $\pi^\mathrm{st}_\ast$-isomorphism for all points $x\colon \ast\to X$.
\end{definition}
\begin{remark}
This definition is independent of the particular choice of stably fibrant replacement functor $\mathcal{R}$. 
Moreover, for $A$ a sequential $X$-spectrum and $\gamma\colon \Delta^1 \to X$ a path in $X$, a straightforward argument involving base change functors shows that $\gamma(0)^\ast \mathcal{R}A$ and $\gamma(1)^\ast\mathcal{R}A$ have the same stable homotopy type (see Lemma \ref{lem:ChangeofTwist}).
\end{remark}
\begin{lemma}
\label{lem:SeqSpecFibrewiseStabEquiv}
A map of sequential $X$-spectra $\psi\colon A\to B$ is a stable weak equivalence precisely if it is a fibrewise $\pi^\mathrm{st}_\ast$-isomorphism.
\end{lemma}
\begin{proof}
Stable weak equivalences are fibrewise $\pi^\mathrm{st}_\ast$-isomorphisms by Ken Brown's lemma.

Conversely, suppose that $\psi$ is a fibrewise $\pi^\mathrm{st}_\ast$-isomorphism and choose  a section $i\mapsto x_i$ of the quotient map $X\to \pi_0(X)$.
By the naturality diagram for the fibrant replacement functor $\mathcal{R}$, $\psi$ is a stable weak equivalence precisely if $\mathcal{R}\psi$ is.

For each $k\geq 0$, the maps $\mathcal{R}A_k\to X$ and $\mathcal{R}B_k\to X$ are fibrations.
Fixing $i\in \pi_0(X)$ and $k\geq 0$, we have a morphism of long exact sequences
\[
\begin{tikzcd}[row sep=small]
  \dotsb \ar[r]
  &
  \pi_n (x_i^\ast \mathcal{R}A_k, x_i)
  \ar[r]
  \ar[d]
  &
  \pi_n (\mathcal{R}A_k, x_i)
  \ar[r]
  \ar[d]
  &
  \pi_n(X, x_i)
  \ar[r]
  \ar[d, equal]
  &
  \dotsb
  \\
  \dotsb \ar[r]
  &
  \pi_n (x_i^\ast \mathcal{R}B_k, x_i)
  \ar[r]
  &
  \pi_n (\mathcal{R}B_k, x_i)
  \ar[r]
  &
  \pi_n(X, x_i)
  \ar[r]
  &
  \dotsb
\end{tikzcd}
\]
where we regard $x_i$ as a point in $\mathcal{R}A_k$, $\mathcal{R}B_k$ via the given section.
Since $x_i^\ast \mathcal{R}A$ and $x^\ast_i\mathcal{R}B$ are fibrant $\Omega$-spectra, the map $\pi_n(x_i^\ast \mathcal{R}A_k)\to \pi_n(x_i^\ast \mathcal{R}B_k)$ can be identified with $
\pi^\mathrm{st}_{n-k}(x_i^\ast \mathcal{R}A)\to
\pi^\mathrm{st}_{n-k}(x_i^\ast \mathcal{R}B)
$, which is an isomorphism by assumption.
By the five lemma, $\mathcal{R}A_k\to \mathcal{R}B_k$ is a weak equivalence.
Since this is true for all $k\geq 0$, $\mathcal{R}\psi\colon \mathcal{R}A\to \mathcal{R}B$ is a levelwise weak equivalence of fibrant $\Omega_X$-spectra.
This completes the proof.
\end{proof}

\begin{lemma}
\label{lem:SeqSpecisStable}
For any simplicial set $X$, the sequential $X$-stable model structure is stable.
\end{lemma}
\begin{proof}
We must show that the left Quillen endofunctor $\overline{\Sigma}_X := S^1\owedge_X (-)$ is a Quillen equivalence for the sequential $X$-stable model structure.
We write $\overline{\Omega}_X$ for the right adjoint to $\overline{\Sigma}_X$ given by computing $\mathrm{sSet}_\ast$-cotensors with respect to $S^1$.
Without loss of generality, we may suppose that $X$ is connected and, furthermore, that $X$ is reduced.
For if $|\pi_0(X)|>1$ then we can work over each path component separately since pulling back along the inclusions $X_i\hookrightarrow X$ of path components induces an equivalence of categories
\[
\mathrm{Sp}^\mathbb{N}_X \longrightarrow \prod_{X_i\in \pi_0(X)} \mathrm{Sp}^\mathbb{N}_{X_i}\,.
\]
If $X$ is connected but not necessarily reduced, we first choose a fibrant replacement $X\to X'$ and pass to the first Eilenberg subcomplex $X''\to X'$ so that we have a zig-zag of weak equivalences $X\to X'\leftarrow X''$ with $X''$ reduced.
Theorem \ref{thm:SeqSpecStructureTheorem} implies that $\mathrm{Sp}^\mathbb{N}_{X''}$ is stable precisely if $\mathrm{Sp}^\mathbb{N}_{X}$ is. 

Now suppose that $X$ is reduced, with $p(X)\colon \mathbb{P}X\to X$ the simplicial path fibration.
The pullback $p(X)^\ast\colon \mathrm{Sp}^\mathbb{N}_X\to \mathrm{Sp}^\mathbb{N}_{\mathbb{P}X}$ is both left and right Quillen for the sequential stable model structures, so preserves all stable weak equivalences.
Let $\kappa\colon \ast\to \mathbb{P}X$ be the constant path at the unique $0$-simplex $x$ of $X$.
The simplicial set $\mathbb{P}X$ is weakly contractible, so that $\kappa$ is a weak equivalence.
We now have a diagram of left $\mathrm{sSet}_\ast$-Quillen functors
\[
\begin{tikzcd}
  \mathrm{Sp}^\mathbb{N} 
  \ar[r, "\kappa_!", "\sim"']
  &
  \mathrm{Sp}^\mathbb{N}_{\mathbb{P}X}
  \ar[r, "\mathbb{P}X_!", "\sim"']
  \ar[d, leftarrow, "p(X)^\ast"]
  &
  \mathrm{Sp}^\mathbb{N}
  \\
  &
  \mathrm{Sp}^\mathbb{N}_X
\end{tikzcd}
\]
in which the arrows marked \lq\lq$\sim$'' are Quillen equivalences.
Fix a sequential $X$-spectrum $A$, take cofibrant and fibrant replacements $A^c \to A\to A^f$ for the sequential $X$-stable model structure.
Then $p(X)^\ast A^c \to p(X)^\ast A\to p(X)^\ast A^f$ is a diagram of stable weak equivalences of sequential $\mathbb{P}X$-spectra.
Choose a fibrant replacement $\mathbb{P}X_! p(X)^\ast A^c \to F$ in $\mathrm{Sp}^\mathbb{N}$, then the derived unit map $\eta\colon p(X)^\ast A^c \to \mathbb{P}X^\ast F$ is a stable weak  equivalence since $(\mathbb{P}X_!\dashv \mathbb{P}X^\ast)$ is a Quillen equivalence.
Factoring $\eta$ into a cofibration followed by an acyclic fibration, we have a diagram of stable equivalences
\[
\begin{tikzcd}
  &
  p(X)^\ast A^c
  \ar[d, rightarrowtail]
  \ar[dl, bend left=-19]
  \ar[dr, bend left=20, "\eta"]
  &
  \\
  p(X)^\ast A^f \ar[r, leftarrow]
  &
  (p(X)^\ast A^c)^f 
  \ar[r, twoheadrightarrow]
  & \mathbb{P}X^\ast F
\end{tikzcd}
\]
of $\mathbb{P}X$-spectra.
The bottom left arrow exists by the right lifting property of the terminal map $p(X)^\ast A^f \to 0_{\mathbb{P}X}$ with respect to the cofibration displayed above.
Applying $\kappa^\ast$, we have a zig-zag of stable weak equivalences 
\begin{equation}
\label{eqn:FibreSpecZigZag}
\begin{tikzcd}
  x^\ast A^f = \kappa^\ast p(X)^\ast A^f 
  \ar[r, leftarrow]
  &
  \kappa^\ast (p(X)^\ast A^c )^f
  \ar[r]
  &
  \kappa^\ast \mathbb{P}X^\ast F \cong F
  \ar[r, leftarrow]
  &
  \mathbb{P}X_! p(X)^\ast A^c\,.
\end{tikzcd}
\end{equation}
In particular, the stable homotopy groups of $\mathbb{P}X_! p(X)^\ast A^c$ are the stable homotopy groups of the fibre spectrum $x^\ast A^f$ of $A$.

In the unparametrised setting, suspension $S^1\wedge (-)$ has the effect of shifting stable homotopy groups up by one: $\pi^\mathrm{st}_\ast (\Sigma P)\cong \pi^\mathrm{st}_{\ast - 1} (P)$.
Both $\mathbb{P}X_!$ and $p(X)^\ast$ preserve $\mathrm{sSet}_\ast$-tensors, so \eqref{eqn:FibreSpecZigZag} shows that the suspension functor $\overline{\Sigma}_X$ shifts fibrewise homotopy groups up by one.
Similarly, $x^\ast$ preserves $\mathrm{sSet}_\ast$-cotensors, so that $\overline{\Omega}_X$ shifts fibrewise homotopy groups down by one.

Putting this all together, for any cofibrant $A\in \mathrm{Sp}^\mathbb{N}_X$ we have that the derived unit $
A \to \overline{\Omega}_X \mathcal{R} \overline{\Sigma}_X A
$ is a fibrewise $\pi^\mathrm{st}_\ast$-isomorphism.
The functor $\overline{\Omega}_X$ preserves and reflects stable weak equivalences between stably fibrant objects (this is proven following a similar argument to Lemma \ref{lem:SeqSpecFibrewiseStabEquiv}), so that $(\overline{\Sigma}_X\dashv \overline{\Omega}_X)$ is a Quillen equivalence.
\end{proof}

\begin{remark}
\label{rem:CompactGen}
The homotopy category of a stable model category is triangulated \cite[Chapter 7]{hovey_model_1999}.
In the classical setting, the stable homotopy category $Ho(\mathrm{Sp}
^\mathbb{N})$ is compactly generated by the sphere spectrum $\mathbb{S} =\Sigma^\infty  S^0$ since $\mathbb{S}$ corepresents stable homotopy.
For $X$-parametrised spectra, we fix a collection of points $\{x_i\}$ indexing the path components of $X$.
According to the derived base change adjunctions and Lemma \ref{lem:SeqSpecFibrewiseStabEquiv}, the collection of sequential $X$-spectra $\{(x_i)_! \mathbb{S}\}$ is a set of compact generators for the triangulated category $Ho(\mathrm{Sp}^\mathbb{N}_X)$.
\end{remark}

We conclude our discussion of sequential parametrised spectra by recording a useful technical result and discussing some canonical examples.
\begin{lemma}
Any sequential $X$-spectrum is a sequential homotopy colimit of shifted fibrewise suspension spectra.
\end{lemma}
\begin{proof}
As in Lemma \ref{lem:SeqSpecisStable} we may suppose without loss of generality that $X$ is reduced, with unique $0$-simplex $x\colon \ast \to X$.
Suppose $A$ is a cofibrant sequential $X$-spectrum, and write $A^i$ for the sequential $X$-spectrum with
\[
A^i_n =
\begin{cases}
\;\;\; A_n & {n<i}\\
\Sigma^{n-i}_X A_i & {n\geq i}
\end{cases}
\]
equipped with the obvious structure maps.
For each $i$ there is a cofibration $A^i\to A^{i+1}$, and $A$ is the sequential colimit of the sequence of cofibrations
\[
\begin{tikzcd}
  \dotsb 
  \ar[r, rightarrowtail]
  &
  A^{i-1}
  \ar[r, rightarrowtail]
  &
  A^i
  \ar[r, rightarrowtail]
  &
  A^{i+1}
  \ar[r, rightarrowtail]
  &
  \dotsb
\end{tikzcd}
\]
That is, $A\cong \mathrm{hocolim}_i A^i$.

For each $i$, there is a canonical map of $X$-spectra $c_i\colon \Sigma^{\infty-i}_X A_i \to A^i$.
We now show that each $c_i$ is a stable weak equivalence, which completes the proof.
Arguing as in Lemma \ref{lem:SeqSpecisStable}, we find that the map of sequential spectra $\mathbb{P}X_! p(X)^\ast c_i$ is stably equivalent to $x^\ast\mathcal{R}c_i$, where $\mathcal{R}$ is a fibrant replacement functor for $\mathrm{Sp}^\mathbb{N}_X$.
Applying the  shift functors $s_X$ and $s= s_\ast$ (Remark \ref{rem:ShiftFunctors}), we have that $\mathbb{P}X_! p(X)^\ast s_X^i c_i = s^i \mathbb{P}X_! p(X)^\ast c_i $ is the identity, and hence $s^i x^\ast \mathcal{R} c_i$ is a stable equivalence.
The shift functor reflects stable weak equivalences between fibrant $\Omega$-spectra, so that $x^\ast \mathcal{R}c_i$ is a stable weak equivalence.
By Lemma \ref{lem:SeqSpecFibrewiseStabEquiv}, this in turn implies that $c_i$ is a stable weak equivalence.
\end{proof}

\begin{example}
\label{exa:CanonicalExamplesofParamSpec}
Given a map $p\colon Y\to X$, the fibrewise suspension spectrum $\Sigma^\infty_X Y_{+X}$ is the result of replacing each of the (homotopy) fibres of $p$ by their stabilisations.
Arguing as in Lemma \ref{lem:SeqSpecisStable}, we find that the homotopy fibre spectrum $\mathbf{R}x^\ast \Sigma^\infty_X Y_{+X}$ is stably equivalent to the suspension spectrum $\Sigma^\infty F_+$ of the fibre  $F_x \to Y\to X$ at $x$.
The fibre spectrum $\mathbf{R}x^\ast \Sigma^\infty_X Y_{+X}$ inherits a $\Omega_x X_+$-action by stabilising the $\Omega_x X$-action on $F_x$.

On the other hand, pushing forward along the terminal map gives $\mathbf{L}X_! \Sigma^\infty_{X}Y_{+X} \cong \Sigma^\infty Y_{+}$.
The suspension spectrum $\Sigma^\infty Y_+$ becomes a (left) $X_+$-comodule via the structure map
\[
\begin{tikzcd}
  Y
  \ar[r, "\Delta_Y"]
  &
  Y\times Y
  \ar[r, "{p\times \mathrm{id}_Y}"]
  &
  X\times Y\,,
\end{tikzcd}
\]
 and stable Koszul pre-duality (Remark \ref{rem:StabProtoKoszul}) relates the $X_+$-comodule $\Sigma^\infty Y_+$ to the $\mathbb{G}X_+$-module $\Sigma^\infty F_+$.
We highlight two particular examples:
\begin{enumerate}[label=(\arabic*)]
  \item In the case of the identity map $\mathrm{id}_X\colon X\to X$, the fibrewise suspension spectrum $\Sigma^\infty_X X_{+X}$ is isomorphic to the trivial $\mathbb{S}$-bundle $X^\ast \mathbb{S}$.
 Taking homotopy fibre spectra at any point $x$ of $X$ gives $\mathbb{S}$ regarded as a trivial $\Omega_x X_+$-module.
  The Koszul (pre-)dual is $\Sigma^\infty X_+$ with its canonical $X_+$-comodule structure.
  
  \item For a point inclusion $x\colon \ast \to X$, the fibrewise suspension spectrum $\Sigma^\infty_X x_{+X}\cong x_! \mathbb{S}$ has $n$-th space the wedge sum $S^n \vee_x X$.
  Taking homotopy fibres at $y\in X$ yields either the $\Omega_y X_+$-module $\Sigma^\infty\Omega_y X_+$ or the zero spectrum, depending on whether $y$ is in the same path component as $x$.
  Taking pushforwards yields $X_! x_! \mathbb{S}\cong \mathbb{S}$: the sphere spectrum regarded as a $X_+$-comodule by stabilising the map $x_+\colon \ast_+ \to X_+$. 
\end{enumerate}
\end{example}

\begin{remark}
For a simplicial set $X$, fix a set $\{(x_i)\}$ of points of $X$ indexing the path components.
We saw that the collection $\{(x_i)_!\mathbb{S}\}$ is a set of compact generators for the triangulated category $Ho(\mathrm{Sp}^\mathbb{N}_X)$.
By the main result of \cite{schwede_stable_2003}, $Ho(\mathrm{Sp}^\mathbb{N}_X)$ is equivalent to the homotopy category of  module spectra over 
\[
\bigoplus_{i\in \pi_0(X)} \Omega_{x_i} X_+\,,
\]
where we have used that $(x_i)^\ast (x_i)_! \mathbb{S}\cong \Sigma^\infty \Omega_{x_i} X_+$ by the previous example.
This is a $|\pi_0(X)|  > 0$ generalisation of the equivalence $Ho(\mathrm{Sp}^\mathbb{N}_X)\cong Ho(\mathbb{G}X_+\mathrm{-Mod})$ previously obtained for the reduced case (Remark \ref{rem:StabProtoKoszul}).
\end{remark}

\subsubsection{Symmetric stabilisation}
\label{SS:SymStabLoc}
Our second approach to combinatorial parametrised spectra uses Hovey's symmetric stabilisation machine, which was initially developed in order to obtain a symmetric monoidal model category of spectra.
However, we do not obtain a symmetric monoidal model category of parametrised spectra since the fibrewise smash product already fails to be a homotopically well-behaved in the unstable setting (see Remark \ref{rem:SmashFail}).
Nevertheless, the $\mathrm{sSet}_\ast$-tensoring on retractive spaces prolongs to a bifunctor
\[
\odot_X \colon \mathrm{Sp}^\Sigma\times \mathrm{Sp}^\Sigma_X\longrightarrow \mathrm{Sp}^\Sigma_X
\]
that \emph{is} homotopically well-behaved, which is crucial for certain applications (as in \cite{braunack-mayer_strict_2019}, for example).

Recall that the category of symmetric sequences $\mathrm{Fun}(\mathbf{\Sigma}, \mathrm{sSet}_{\dslash X})$ is symmetric monoidal with respect to the Day convolution tensor product
\[
A\otimes^\mathrm{Day}_X B \colon n\longmapsto \int^{p,q \in \mathbf{\Sigma}} \mathbf{\Sigma}(p+q, n)_+ \owedge_X A(p)\wedge_X B(q)  
\]
for symmetric sequences $A, B \in \mathrm{Fun}(\mathbf{\Sigma}, \mathrm{sSet}_{\dslash X})$.
For any map $f\colon X\to X'$, the pullback functor on retractive spaces prolongs to a strongly monoidal functor $f^\ast\colon \mathrm{Fun}(\mathbf{\Sigma},\mathrm{sSet}_{\dslash X'}) \to \mathrm{Fun}(\mathbf{\Sigma},\mathrm{sSet}_{\dslash X})$ on categories of symmetric sequences.
The category $\mathrm{Fun}(\mathbf{\Sigma}, \mathrm{sSet}_{\dslash X})$ becomes a $\mathrm{Fun}(\mathbf{\Sigma},\mathrm{sSet}_\ast)$-module by taking pullbacks along the terminal map.
Explicitly, the action of $K\in \mathrm{Fun}(\mathbf{\Sigma}, \mathrm{sSet}_\ast)$ on $A\in \mathrm{Fun}(\mathbf{\Sigma}, \mathrm{sSet}_{\dslash X})$ is
\[
K\odot^\mathrm{Day}_X A:= X^\ast K\otimes^\mathrm{Day}_X A\,.
\]
The symmetric monoidal functor $\mathbf{\Sigma}\to\mathrm{sSet}_\ast$ sending $n\mapsto S^n := \bigwedge_{i=1}^n S^1$ defines a commutative monoid $\mathbb{S}$ in $(\mathrm{Fun}(\mathbf{\Sigma}, \mathrm{sSet}_\ast), \otimes_\mathrm{Day})$ called the \emph{symmetric sphere spectrum}.
Symmetric spectra then defined as $\mathbb{S}$-modules in $\mathrm{Fun}(\mathbf{\Sigma}, \mathrm{sSet}_\ast)$, which has an evident generalisation to the parametrised setting:
\begin{definition}
A \emph{symmetric $X$-spectrum} is an $\mathbb{S}$-module in $\mathrm{Fun}(\mathbf{\Sigma},\mathrm{sSet}_{\dslash X})$.
The category of symmetric $X$-spectra is denoted $\mathrm{Sp}^\Sigma_X$.
\end{definition}
\begin{remark}
Unravelling the definitions, a symmetric $X$-spectrum $A$ is the data of a symmetric sequence $\{A(n)\}_{n\in \mathbb{N}}$ of retractive spaces over $X$ together with a $\Sigma_n$-equivariant map of retractive spaces $\sigma_n \colon S^1 \owedge_X A(n) \to A(n+1)$ for each $n\geq 0$.
These data are subject to the condition that for all $p,q\geq 0$ the composite
\[
  S^p\owedge_X A(q)
  \xrightarrow{S^{p-1}\owedge_X \sigma_q}
  S^{p-1}\owedge_X A(q+1)
  \xrightarrow{S^{p-2}\owedge_X \sigma_{q+1}}
  \dotsb
  \xrightarrow{\sigma_{p+q-1}}
  A(p+q)
\]
is $\Sigma_p\times \Sigma_q$-equivariant.
\end{remark}
\begin{remark}
\label{rem:FreeSymSpec}
There is a free-forgetful adjunction
\begin{equation}
\label{eqn:FreeSymmSpec}
\begin{tikzcd}
\mathrm{Fun}(\mathbf{\Sigma},\mathrm{sSet}_{\dslash X})
\ar[rr, shift left=1.1ex, "\mathbb{S}\odot^\mathrm{Day}_X(-)"]
\ar[rr, leftarrow, shift left=-1.1ex, "\bot", ""']
&&
\mathrm{Sp}^\Sigma_X\,.
\end{tikzcd}
\end{equation}
Since the monad $\mathbb{S}\odot^\mathrm{Day}_X(-)$ preserves colimits, the category $\mathrm{Sp}^\Sigma_X$ of symmetric $X$-spectra is locally presentable, with limits and colimits created by the forgetful functor to symmetric sequences.

There is a family of adjunctions relating $\mathrm{sSet}_{\dslash X}$ and $\mathrm{Sp}^\Sigma_X$ indexed by $k\geq 0$.
The construction is analogous to that of Remark \ref{rem:FreeSeqSpec}, though we must take care to remember the symmetric group actions. The evaluation functor on symmetric sequences $\mathrm{ev}_k \colon \{Z_n\}\mapsto Z_k$ has a left adjoint $s_k$: for a retractive space $Y\in \mathrm{sSet}_{\dslash X}$, the symmetric sequence $s_k(Y)$ has $(\Sigma_{k})_+ \owedge_X Y$ as its $k$-th term, with all other terms in the sequence equal to $0_X$. 
We then have the composite adjunction
\[
(\mathbf{\Sigma}^{\infty-k}_X\dashv \widetilde{\mathbf{\Omega}}^{\infty-k}_X)\colon
\begin{tikzcd}
\mathrm{sSet}_{\dslash X}
\ar[rr, shift left=1.1ex, "s_k"]
\ar[rr, leftarrow, shift left=-1.1ex, "\bot", "\mathrm{ev}_k"']
&&
\mathrm{Fun}(\mathbf{\Sigma},\mathrm{sSet}_{\dslash X})
\ar[rr, shift left=1.1ex, "\mathbb{S}\odot^\mathrm{Day}_X(-)"]
\ar[rr, leftarrow, shift left=-1.1ex, "\bot", ""']
&&
\mathrm{Sp}^\Sigma_X\,,
\end{tikzcd}
\]
where boldface symbols are used in order to distinguish from the case of sequential parametrised spectra considered previously.
Explicitly, $\widetilde{\mathbf{\Omega}}^{\infty-k}$ evaluates the $k$-th term of a symmetric $X$-spectrum (forgetting the $\Sigma_k$-action), whereas 
\[
\mathbf{\Sigma}^{\infty-k}_X Y_n =
\begin{cases}
\qquad\qquad\quad\; 0_X & n<k\\
\big((\Sigma_n)_+ \owedge_X  S^{n-k} \owedge_X Y \big)\big/\Sigma_{n-k}& n\geq k\,,
\end{cases}
\]
where $\Sigma_{n-k}$ acts via the canonical inclusion on $\Sigma_n$ and on $S^{n-k}$ by permuting smash factors.
\end{remark}

\begin{remark}
For any map of simplicial sets $f\colon X\to X'$ the base change adjunctions $(f_! \dashv f^\ast \dashv f_\ast)$ prolong to categories of symmetric sequences by levelwise application:
\[
(f_!\dashv f^\ast \dashv f_\ast)\colon
\begin{tikzcd}
\mathrm{Fun}(\mathbf{\Sigma},\mathrm{sSet}_{\dslash X})
\ar[rr, shift left=2.2ex, ""]
\ar[rr, leftarrow, "\bot"]
\ar[rr, shift left=-2.2ex, "\bot", ""']
&&
\mathrm{Fun}(\mathbf{\Sigma},\mathrm{sSet}_{\dslash X'})\,.
\end{tikzcd}
\]
The functors $f_!$ and $f^\ast$  preserve $\mathrm{Fun}(\mathbf{\Sigma},\mathrm{sSet}_\ast)$-tensors, so send $\mathbb{S}$-modules to $\mathbb{S}$-modules.
We thus have a triple of adjoint functors between categories of symmetric parametrised spectra
\begin{equation}
\label{eqn:BChangeSymmSpec}
(f_!\dashv f^\ast \dashv f_\ast)\colon
\begin{tikzcd}
\mathrm{Sp}^\Sigma_X
\ar[rr, shift left=2.2ex, ""]
\ar[rr, leftarrow, "\bot"]
\ar[rr, shift left=-2.2ex, "\bot", ""']
&&
\mathrm{Sp}^\Sigma_{X'}\,,
\end{tikzcd}
\end{equation}
each of which coincides with the functor of the same name on underlying symmetric sequences (justifying the abuse of notation).
\end{remark}

\begin{remark}
\label{rem:FibSymSmash}
The fibrewise smash product on retractive spaces over $X$ prolongs to a closed symmetric monoidal structure on $\mathrm{Sp}^\Sigma_X$.
The \emph{fibrewise smash product} of symmetric $X$-spectra $A$ and $B$ is the symmetric $X$-spectrum
\[
A\otimes_X B =\mathrm{colim}\left(\!
\begin{tikzcd}
  A\odot^\mathrm{Day}_X
  \mathbb{S}
  \odot^\mathrm{Day}_X
  B
  \ar[r, shift left=0.8ex]
  \ar[r, shift left=-0.8ex]
  &
  A\otimes^\mathrm{Day}_X
  B
\end{tikzcd}\!
\right)\,,
\]
where upper and lower arrows are given by the $\mathbb{S}$-actions on $A$ and $B$ respectively, using the symmetry isomorphism of the Day convolution product in the first instance. 
The unit of the fibrewise smash product $\otimes_X$ is the symmetric $X$-spectrum\footnote{This expression has, regrettably, far too many $X$'s.} $\mathbf{\Sigma}^\infty_X (X_{+X}) \cong X^\ast \mathbb{S}$.
The reader can convince themselves that for each $k,l\geq 0$ there is an isomorphism
\begin{equation}
\label{eqn:SmashVSSuspend}
\mathbf{\Sigma}^{\infty-k}_X Y\otimes_X \mathbf{\Sigma}^{\infty-l}_X Z \cong \mathbf{\Sigma}^{\infty-(k+l)}_X (Y\wedge_X Z)
\end{equation}
natural in $Y,Z\in \mathrm{sSet}_{\dslash X}$.
Setting $k=l=0$ implies that $\mathbf{\Sigma}^\infty_X\colon (\mathrm{sSet}_{\dslash}, \wedge_X)\to (\mathrm{Sp}^\Sigma_X, \otimes_X)$ is strongly symmetric monoidal.

While generally incompatible with the model structures we consider, the fibrewise smash product gives us a slick way of understanding $\mathrm{Sp}^\Sigma_X$ as a $\mathrm{Sp}^\Sigma$-module category.
For any map of simplicial sets $f\colon X\to X'$, Lemma \ref{lem:PBisStonglyClosed} implies that $f^\ast$ is strongly closed symmetric monoidal with respect to fibrewise smash products.
In particular, each category $\mathrm{Sp}^\Sigma_X$ is $\mathrm{Sp}^\Sigma$-enriched with $\mathrm{Sp}^\Sigma$-tensors and cotensors with the tensoring given by
\begin{equation}
\label{eqn:psSetTensorSymSpec}
(K, A)\longmapsto K\odot_X A := X^\ast K \otimes_X A\,.
\end{equation}
The base change adjunctions $(f_!\dashv f^\ast)$ and $(f^\ast\dashv f_\ast)$ on parametrised symmetric spectra are $\mathrm{Sp}^\Sigma$-enriched; $f_!$ preserves $\mathrm{Sp}^\Sigma$-tensors whereas $f^\ast$ preserves both tensors and cotensors.
\end{remark}

$\mathrm{Fun}(\mathbf{\Sigma}, \mathrm{sSet}_{\dslash X})$ is a left proper combinatorial $\mathrm{sSet}_\ast$-model category with fibrations and weak equivalences created by the evaluation functors $\mathrm{ev}_k$.
The free-forgetful adjunction \eqref{eqn:FreeSymmSpec} induces a left proper combinatorial $\mathrm{sSet}_\ast$-model structure on $\mathrm{Sp}^\Sigma_X$, which we call the \emph{symmetric $X$-projective model structure}.
A map $f\colon A\to B$ of symmetric $X$-spectra is thus a weak equivalence or fibration for the symmetric $X$-projective model structure precisely if each map $f(n)\colon A(n)\to B(n)$ is a weak equivalence or fibration in $\mathrm{sSet}_{\dslash X}$.
Sets of generating cofibrations and generating acyclic cofibrations are given by
\[
\mathcal{I}^{\Sigma\mathrm{-proj}}_X
:= \bigcup_{k\in \mathbb{N}} \mathbf{\Sigma}^{\infty-k}_X (\mathcal{I}^\mathrm{Kan}_X)
\qquad
\mbox{ and }
\qquad
\mathcal{J}^{\Sigma\mathrm{-proj}}_X
:= \bigcup_{k\in \mathbb{N}} \mathbf{\Sigma}^{\infty-k}_X (\mathcal{J}^\mathrm{Kan}_X)
\]
respectively.
The $\mathrm{sSet}_\ast$-action on $\mathrm{Sp}^\Sigma_X$ is defined by restricting \eqref{eqn:psSetTensorSymSpec} along $\mathbf{\Sigma}^\infty$; for a simplicial set $K$ and symmetric $X$-spectrum $A$ we calculate $(\mathbf{\Sigma}^\infty K \odot_X A)(n) \cong K\owedge_X A(n)\cong X^\ast K\wedge_X A(n)$ with structure maps induced from those of $A$ using the symmetry isomorphism of $\wedge_X$.

\begin{remark}
The left Quillen endofunctor $\mathbf{\Sigma}_X := \mathbf{\Sigma}^\infty S^1\odot_X (-)$ models suspension on the symmetric $X$-projective homotopy category.
Its right adjoint, which we denote $\mathbf{\Omega}_X$, is given by the $\mathrm{sSet}_\ast$-cotensoring with respect to $S^1$. 
Observe that, unlike the case of symmetric spectra, there is no difference between the suspension endofunctor and tensoring with $S^1$.
\end{remark}

For any $Y\in\mathrm{sSet}_{\dslash X}$ and $k\geq 0$, we write $\xi_{k,X}(Y)$ for the $(\mathbf{\Sigma}^{\infty-k}_X\dashv \widetilde{\mathbf{\Omega}}^{\infty-k}_X)$-adjunct of the map
\[
S^1 \owedge_X Y \longrightarrow (\Sigma_{n+1})_+ \owedge_X S^1 \owedge_X Y
\]
induced by the unit $1\colon \ast \to \Sigma_{n+1}$.
The \emph{symmetric $X$-stable model structure} is the left Bousfield localisation of the symmetric $X$-projective model structure at the set
\[
S_{\Sigma,X} :=\big\{\xi_{k,X} (C)\colon \mathbf{\Sigma}^{\infty-(k+1)}_X (S^1 \owedge_X C)\to \mathbf{\Sigma}^{\infty-k}_{X} (C) \big\}\,,
\]
where $C$ ranges over domains and codomains of morphisms in $\mathcal{I}^\mathrm{Kan}_X$ and $k\geq 0$.
The weak equivalences and fibrations of the symmetric $X$-stable model structure are called stable weak equivalences and stable fibrations respectively.

The following properties of the symmetric $X$-stable model structure are proven analogously to the sequential case:
\begin{lemma}
\label{lem:SymStabFib}
Stably fibrant objects of $\mathrm{Sp}^\Sigma_X$ are precisely the \emph{fibrant $\mathbf{\Omega}_X$-spectra}: the symmetric $X$-spectra $A$ such that each $A(n)\in \mathrm{sSet}_{\dslash X}$ is fibrant and the adjoint structure maps $\sigma^\vee_n\colon A(n)\to \Omega_X A(n+1)$ are weak equivalences for all $n\geq 0$.
The stable weak equivalences of fibrant $\mathbf{\Omega}_X$-spectra are precisely the levelwise weak equivalences.
\end{lemma}

\begin{lemma}
\label{lem:XiMapsAreStable}
For each $Y\in \mathrm{sSet}_{\dslash X}$, the map $\xi_{k,X}(Y)$ is a stable weak equivalence for all $k\geq 0$.
\end{lemma}
\begin{lemma}
\label{lem:SymSpecShifts}
The adjunction $(\mathbf{\Sigma}_X\dashv \mathbf{\Omega}_X)\colon \mathrm{Sp}^\Sigma_X\to \mathrm{Sp}^\Sigma_X$ is a Quillen equivalence. In particular, $\mathrm{Sp}^\Sigma_X$ is a stable model category.
\end{lemma}
\begin{theorem}
\label{thm:SymSpecStructureTheorem}
For any simplicial set $X$, $\mathrm{Sp}^\Sigma_X$ is a left proper combinatorial $\mathrm{Sp}^\Sigma$-model category.
For any map of simplicial sets $f\colon X\to X'$ there is a triple of adjoint $\mathrm{Sp}^\Sigma$-functors
\[
(f_!\dashv f^\ast\dashv f_\ast)\colon
\begin{tikzcd}
\mathrm{Sp}^\Sigma_X
\ar[rr, shift left=2.2ex, ""]
\ar[rr, leftarrow, "\bot"]
\ar[rr, shift left=-2.2ex, "\bot", ""']
&&
\mathrm{Sp}^\Sigma_{X'}\,,
\end{tikzcd}
\]
where $(f_!\dashv f^\ast)$ is a $\mathrm{Sp}^\Sigma$-Quillen adjunction that is moreover a $\mathrm{Sp}^\Sigma$-Quillen equivalence if $f$ is a weak equivalence.
If $f$ is a fibration (or a projection to a factor in a product) then $(f^\ast \dashv f_\ast)$ is a $\mathrm{Sp}^\Sigma$-Quillen adjunction, and is a $\mathrm{Sp}^\Sigma$-Quillen equivalence if $f$ is an acyclic fibration.
\end{theorem}
\begin{proof}
We first verify that the $\mathrm{Sp}^\Sigma$-tensoring on $\mathrm{Sp}^\Sigma_X$ is a Quillen bifunctor.
Using \eqref{eqn:SmashVSSuspend} and Theorem \ref{thm:RetSpStructureTheorem}, one verifies that the pushout-products of generating cofibrations are cofibrations.
By a lifting argument, it follows that all pushout-products of cofibrations are cofibrations.

It remains to check that the pushout-product of cofibrations $i\,\square\, j$ is acyclic if either of $i$, $j$ is acyclic. 
For a pointed simplicial set $K$ one checks that $\mathbf{\Sigma}^{\infty-k}K\odot_X \xi_{l,X}(Y) \cong \xi_{l+k,X}(K\owedge_X Y)$ for all $Y\in \mathrm{sSet}_{\dslash X}$ and $k,l\geq 0$.
By Lemma \ref{lem:XiMapsAreStable} the endofunctor $\mathbf{\Sigma}^{\infty-k}K\odot_X (-)$ defines a left Quillen endofunctor for the stable model structure.
Supposing now that $i =\mathbf{\Sigma}^{\infty-k}(\partial\Delta^n_+\to \Delta^n_+)$ is a generating cofibration of $\mathrm{Sp}^\Sigma$, then for any (stably) acyclic cofibration $j\colon A\to B$ of sequential $X$-spectra, we have a diagram
 \[
 \begin{tikzcd}
    \mathbf{\Sigma}^{\infty-k}\partial\Delta^n_+ \odot_X A 
    \ar[r]
    \ar[d, "\sim"'] & 
    \mathbf{\Sigma}^{\infty-k}\Delta^n_+ \odot_X A
    \ar[dr, bend left =15, "\sim"]
    \ar[d, "\sim"] 
    &
    \\
    \mathbf{\Sigma}^{\infty-k}\partial\Delta^n_+ \odot_X B 
    \ar[r]
    &
    P
    \ar[r, "{i\,\square\, j}"]
    &
    \mathbf{\Sigma}^{\infty-k}\Delta^n_+ \odot_X B\,,
 \end{tikzcd}
 \] 
 with stable weak equivalences as marked, where $P$ is the pushout.
 By the $2$-out-of-$3$ property, $i\,\square\, j$ is a stable weak equivalence.
 For all cofibrations $i$ and acyclic cofibrations $j$, it now follows from cofibrant generation (of $\mathrm{Sp}^\Sigma$) that $i\,\square\, j$ is an acyclic cofibration.
 The remaining case is proven similarly, and we conclude that the $\mathrm{Sp}^\Sigma$-tensoring on $\mathrm{Sp}^\Sigma_X$ is a Quillen bifunctor---for $X=\ast$ this also shows that $\mathrm{Sp}^\Sigma$ is a symmetric monoidal model category.
 
 The remaining claims are proven along the same lines as Theorem \ref{thm:SeqSpecStructureTheorem}, using that $f^\ast$ and $f_!$ preserve $\mathrm{Sp}^\Sigma$-tensors.
\end{proof}

We conclude this section by giving a fibrewise characterisation of stable weak equivalences of symmetric parametrised spectra.
Forgetting symmetric group actions gives rise to a functor $\mathrm{Sp}^\Sigma_X\to \mathrm{Sp}^\mathbb{N}_X$; for $X=\ast$ it is an unpleasant but nevertheless true fact that this functor reflects but does not preserve stable weak equivalences \cite[Section 3.1]{hovey_symmetric_2000}.
To remedy this, we pose the following
\begin{definition}
The \emph{naive stable homotopy groups} of a symmetric spectrum $P$ are the stable homotopy groups $\pi^\mathrm{st}_\ast(P)$ of the underlying sequential spectrum.
The \emph{true stable homotopy groups} of $P$ are $\pmb{\pi}^\mathrm{st}_\ast (P) :=\pi^\mathrm{st}_\ast(\mathcal{R}P)$, where $\mathcal{R}$ is a fibrant replacement functor on $\mathrm{Sp}^\Sigma$.
\end{definition}
\begin{remark}
The definition of true stable homotopy groups only depends on $\mathcal{R}$ up to natural isomorphism.
Using Lemma \ref{lem:SymStabFib} it is easy to show that a morphism of symmetric spectra is a stable weak equivalence if and only if it is a $\pmb{\pi}^\mathrm{st}_\ast$-isomorphism.
\end{remark}

\begin{definition}
Fix a fibrant replacement functor $\mathcal{R}_X$ for the symmetric $X$-stable model structure.
A map of symmetric $X$-spectra $\psi\colon A\to B$ is a \emph{fibrewise $\pmb{\pi}^\mathrm{st}_\ast$-isomorphism} if $x^\ast \mathcal{R}_X\psi$ is a $\pmb{\pi}^\mathrm{st}_\ast$-isomorphism for all points $x\colon \ast \to X$.
\end{definition}

The fibrewise characterisation of stable weak equivalences of symmetric parametrised spectra is proven following the same argument as for sequential parametrised spectra (Lemma \ref{lem:SeqSpecFibrewiseStabEquiv}).
\begin{lemma}
A map of symmetric $X$-spectra $\psi\colon A\to B$ is stable weak equivalence precisely if it is a fibrewise $\pmb{\pi}^\mathrm{st}_\ast$-isomorphism.
\end{lemma}

\subsubsection{Comparing models of parametrised spectra}
We conclude our discussion of the local theory of parametrised spectra by comparing various different models.
Our first goal is to compare the sequential and symmetric models for parametrised spectra, which we subsequently compare with the $\infty$-categorical construction.
The first result, a special case of \cite[Theorem 10.1]{hovey_spectra_2001}, compares the sequential and symmetric models by showing that both are equivalent to an auxiliary homotopy theory of \lq\lq double spectra'':
\begin{theorem}
\label{thm:LocalSpecComp}
For any simplicial set $X$, there is a zig-zag of left Quillen equivalences
\[
\begin{tikzcd}
  \mathrm{Sp}^\Sigma_X
  \ar[r, "d(\Sigma)"]
  &
  \mathrm{Sp}^\mathcal{D}_{X}
  \ar[r, leftarrow, "d(\mathbb{N})"]
  &
  \mathrm{Sp}^\mathbb{N}_X
\end{tikzcd}
\]
pseudonatural in $X$.
\end{theorem}
\begin{proof}
The intermediate category of \emph{double $X$-spectra} is $\mathrm{Sp}^\mathcal{D}_X := \mathrm{Sp}^\mathbb{N}(\mathrm{Sp}^\Sigma_X)\cong \mathrm{Sp}^\Sigma(\mathrm{Sp}^\mathbb{N}_X)$ obtained by taking sequential (resp.~symmetric) spectrum objects in symmetric (resp.~sequential) $X$-spectra with respect to the endofunctor $\Sigma_X := S^1\owedge_X(-)$ arising from the $\mathrm{sSet}_\ast$-tensoring.
A double $X$-spectrum is thus a family $\{A_m(n)\}_{m,n\geq 0}$ of retractive spaces over $X$ such that 
\begin{enumerate}[label=(\roman*)]
  \item for each $m\geq 0$, the collection $\{A_{m}(n)\}_{n\geq 0}$ is a symmetric $X$-spectrum; 
  \item for each $n\geq 0$, the collection $\{A_m(n)\}_{m\geq 0}$ is a sequential $X$-spectrum; and
  \item for each $m,n \geq 0$ the structure maps $\Sigma_X A_m(\ast) \to A_{m+1}(\ast)$ and $\Sigma_X A_\bullet(n)\to A_\bullet(n+1)$ are morphisms of symmetric and sequential $X$-spectra respectively. 
\end{enumerate}
The category of double $X$-spectra has a projective model structure that is left proper, combinatorial, and $\mathrm{sSet}_\ast$-enriched. The weak equivalences and fibrations of the projective model structure are precisely the levelwise such maps of retractive spaces over $X$.

For each pair $m,n\geq 0$, the evaluation functor $\mathrm{ev}_{m,n}\colon \mathrm{Sp}^\mathcal{D}_X\to \mathrm{sSet}_{\dslash X}$, $\{A_k(l)\}_{k,l\geq 0}\mapsto A_m(n)$, has a left adjoint $F_{m,n}$ (the first index is \lq\lq sequential'' and the second is \lq\lq symmetric'').
The \emph{twice stable model structure} on double $X$-spectra is the left Bousfield localisation of the projective model structure at the set 
\[
S_{\mathcal{D}, X}:= 
\underbrace{\big\{\zeta_{m,n}(C)\colon F_{m+1, n}(\Sigma_X C)\to F_{m,n}(C)\big\}}_{\vphantom{\big(}=:T_{\mathbb{N},X}}
\cup
\underbrace{
\big\{
\xi_{m,n}(C)\colon F_{m,n+1}(\Sigma_X C)\to F_{m,n}(C)\big\}}_{\vphantom{\big(}=: T_{\Sigma, X}}\,,
\]
with $\zeta$ and $\xi$ as in $S_{\mathbb{N},X}$ and $S_{\Sigma,X}$ respectively.
As usual, $C$ ranges over domains and codomains of morphisms in the set of generating cofibrations $\mathcal{I}^\mathrm{Kan}_X$ and $m,n\geq 0$.
The category $\mathrm{Sp}^\mathcal{D}_X$ has two Quillen endo-adjunctions $(\Sigma_X\dashv \Omega_X)$ and $(\mathbf{\Sigma}_X\dashv \mathbf{\Omega}_X)$ implementing fibrewise suspensions in the \lq\lq sequential'' and \lq\lq symmetric'' directions respectively.
One shows that both are Quillen auto-equivalences for the $S_{\mathcal{D},X}$-local model structure (cf.~Lemmas \ref{lem:SeqSpecShifts}, \ref{lem:SymSpecShifts}).

From each double $X$-spectrum $\{A_m(n)\}$ we extract a sequential and a symmetric $X$-spectrum at level zero by the assignments 
\[
\mathrm{ev}_0^\mathbb{N}\colon A\longmapsto A_\bullet(0)
\qquad
\mbox{ and }
\qquad
\mathrm{ev}_0^\Sigma\colon A \longmapsto A_0(-)
\]
respectively.
Both assignments are functorial and preserve limits, so have left adjoints by the adjoint functor theorem; for example, $d(\Sigma)$ sends the symmetric $X$-spectrum $P$ to the double $X$-spectrum
\[
(d(\Sigma) P)_m(n) = \Sigma^m_X P_n\,.
\]
The resulting zig-zag of adjunctions
\begin{equation}
\label{eqn:DStab}
\begin{tikzcd}
\mathrm{Sp}^\Sigma_X
\ar[rr, shift left=1.1ex, "d(\Sigma)"]
\ar[rr, leftarrow, shift left=-1.1ex, "\bot", "\mathrm{ev}^\Sigma_0"']
&&
\mathrm{Sp}^\mathcal{D}_X
\ar[rr, leftarrow, shift left=1.1ex, "d(\mathbb{N})"]
\ar[rr,  shift left=-1.1ex, "\bot", "\mathrm{ev}^\mathbb{N}_0"']
&&
\mathrm{Sp}^\mathbb{N}_X
\end{tikzcd}
\end{equation}
is Quillen for the stable model structures.
The above model categories are all stable (as can be shown for $\mathrm{Sp}^\mathcal{D}_X$ using our previous methods), and $\mathrm{Sp}^\mathcal{D}_X$ can be viewed either as the stabilisation of $\mathrm{Sp}^\mathbb{N}_X$ with respect to $\Sigma_X$ or of $\mathrm{Sp}^\Sigma_X$ with respect to $\mathbf{\Sigma}_X$.
Since $\Sigma_X$ and $\mathbf{\Sigma}_X$ are already Quillen autoequivalences of their respective domains, we expect the zig-zag \eqref{eqn:DStab} to be a Quillen equivalence.
This is indeed the case, as we now show for the adjunction $(d(\Sigma)\dashv \mathrm{ev}^\Sigma_0)$.

There is a commuting diagram of left Quillen functors
\begin{equation}
\label{eqn:ModelStrucDSpec}
\begin{tikzcd}
\left(\mathrm{Sp}^\Sigma_X\right)_\mathrm{proj}
\ar[r, "d(\Sigma)"]
\ar[d]
&
\left(\mathrm{Sp}^\mathcal{D}_X\right)_\mathrm{proj}
\ar[r]
\ar[d]
&
L_{T_{\mathbb{N},X}}
\left(\mathrm{Sp}^\mathcal{D}_X\right)_\mathrm{proj}
\ar[d]
\\
\left(\mathrm{Sp}^\Sigma_X\right)_\mathrm{stab}
\ar[r, "d(\Sigma)"]
&
\displaystyle{\underbrace{
L_{T_{\Sigma,X}}
\left(\mathrm{Sp}^\mathcal{D}_X\right)_\mathrm{proj}}_{\text{\lq\lq symmetric projective''}}}
\ar[r]
&
\displaystyle{\underbrace{L_{S_{\mathcal{D},X}}
\left(\mathrm{Sp}^\mathcal{D}_X\right)_\mathrm{proj}
=\left(\mathrm{Sp}^\mathcal{D}_X\right)_\mathrm{stab}}_{\text{\lq\lq twice stable''}}}
\end{tikzcd}
\end{equation}
where unlabelled arrows are identity functors implementing left Bousfield localisations.
Observe that the left Bousfield localisation of the projective model structure on $\mathrm{Sp}^\mathcal{D}_X$ at the set $T_{\Sigma,X}$ coincides with the \emph{symmetric projective} model structure on $\mathrm{Sp}^\mathcal{D}_X \cong \mathrm{Sp}^\mathbb{N}(\mathrm{Sp}^\Sigma_X)$, for which the weak equivalences and fibrations are the morphisms of double $X$-spectra $A_\bullet(\ast)\to B_\bullet(\ast)$ such that for each $m\geq 0$, $A_m(\ast)\to B_m (\ast)$ is a stable weak equivalence, respectively a stable fibration, in $\mathrm{Sp}^\Sigma_X$.

The fibrant objects for the twice stable model structure on $\mathrm{Sp}^\mathcal{D}_X$ are the level fibrant double $X$-spectra $A_\bullet(\ast)$ such that for each $m,n\geq 0$, $A_m(\ast)$ is an $\mathbf{\Omega}_X$-spectrum and $A_\bullet(n)$ is an $\Omega_X$-spectrum (cf.~Lemmas \ref{lem:SeqSpecFib} and \ref{lem:SymStabFib}).
Since $\mathbf{\Omega}_X$ is a right Quillen auto-equivalence of $\mathrm{Sp}^\Sigma_X$ it reflects weak equivalences between fibrant $\mathbf{\Omega}_X$-spectra.
Suppose, then, that $f\colon A_\bullet(\ast)\to B_\bullet(\ast)$ is a morphism of fibrant double $X$-spectra for which $\mathrm{ev}^\Sigma_0\colon A_0(\ast)\to B_0(\ast)$ is a stable weak equivalence of symmetric $X$-spectra.
Since $A_\bullet(\ast)$ and $B_\bullet(\ast)$ are fibrant, the induced map $\mathbf{\Omega}_X^n A_n(\ast)\to \mathbf{\Omega}_X^n B_n(\ast)$ is a weak equivalence for all $n$.
As $\mathbf{\Omega}_X$ reflects weak equivalences of fibrant objects, $A_\bullet(\ast)\to B_\bullet(\ast)$ is a level weak equivalence of fibrant double $X$-spectra.
We have thus shown that $\mathrm{ev}^\Sigma_0$ reflects weak equivalences between fibrant objects.

In order to show that $(d(\Sigma)\dashv \mathrm{ev}^\Sigma_0)$ is a Quillen equivalence for the stable model structures, it is sufficient to show that the derived unit of the adjunction is a weak equivalence on cofibrant objects \cite[Corollary 1.3.16]{hovey_model_1999}.
This is done using the diagram of left Quillen functors \eqref{eqn:ModelStrucDSpec} that relates the various model structures at play.
To this end:
\begin{itemize}
  \item Let $\mathcal{R}$ be a fibrant replacement functor for the \emph{symmetric projective} model structure on $\mathrm{Sp}^\mathcal{D}_X$.
  
  \item Let $\mathcal{S}$ be a fibrant replacement functor for the \emph{twice stable} model structure on $\mathrm{Sp}^\mathcal{D}_X$.
\end{itemize}
The goal is to show that the derived unit $P \to \mathrm{ev}^\Sigma_0(\mathcal{S}d(\Sigma)P)$ is a stable weak equivalence  for every cofibrant symmetric $X$-spectrum $P$.

Proceeding along the lines of \cite[Theorem 5.1]{hovey_spectra_2001}, fix a cofibrant symmetric $X$-spectrum $P$ and consider the composite
$
P \to 
\mathrm{ev}^\Sigma_0( d(\Sigma) P)
\to
\mathrm{ev}^\Sigma_0(\mathcal{R} d(\Sigma) P)
$,
which is a stable weak equivalence of symmetric $X$-spectra (by the properties of the symmetric projective model structure).
We claim that $\mathcal{R}d(\Sigma)P$ is already fibrant for the twice stable model structure.
Indeed, by the properties of the symmetric projective model structure, for each $m,n\geq 0$ we have that 
\begin{itemize}
  \item $\mathrm{ev}_{m,n} (\mathcal{R} d(\Sigma)P)$ is a fibrant object of $\mathrm{sSet}_{\dslash X}$
  \item $\{(\mathcal{R}d(\Sigma)P)_m (k)\}_{k\geq 0}$ is an $\mathbf{\Omega}_X$-spectrum,
\end{itemize}
so we must only check that $(\mathcal{R}d(\Sigma)P)_\bullet (n)$ is an $\Omega_X$-spectrum for each $n\geq 0$.
We do this using the fact that the $\mathrm{sSet}_\ast$-tensoring $A\mapsto \Sigma_X A = S^1 \owedge_X A$ is a left Quillen equivalence of $\mathrm{Sp}^\Sigma_X$.
Let $\mathcal{R}'$ be a fibrant replacement functor for the stable model structure on symmetric $X$-spectra, so that for each $n\geq 0$ the morphism
\[
(d(\Sigma) P)_n(\ast)
\cong
\Sigma^n_X P_\ast
\longrightarrow
\Omega_X \mathcal{R}' (\Sigma^{n+1}_X P_\ast)
\]
is a stable weak equivalence of symmetric $X$-spectra.
By the properties of the symmetric projective model structure, the morphism of double $X$-spectra $d(\Sigma) P\to \mathcal{R} d(\Sigma) P$ is a stable weak equivalence in the \lq\lq symmetric'' direction, and for each $m\geq 0$ the symmetric $X$-spectrum $(\mathcal{R} d(\Sigma) P)_m(\ast)$ is stably fibrant.
The lifting diagram of symmetric $X$-spectra
\[
\begin{tikzcd}
\Sigma^m_X P_\ast
\cong (d(\Sigma)P)_m(\ast)
\ar[d, rightarrowtail, "\sim"']
\ar[r, "\sim"]
&
(\mathcal{R}d(\Sigma)P)_m(\ast)
\\
\mathcal{R}' \Sigma^m_X P_\ast
\ar[ur, dashed, "\exists"']
&
\end{tikzcd}
\]
yields a stable weak equivalence of symmetric $X$-spectra $\mathcal{R}'\Sigma^m_XP_\ast\to (\mathcal{R}d(\Sigma)P)_m(\ast)$.
Being a right Quillen equivalence, $\Omega_X$ preserves stable weak equivalences between stably fibrant symmetric $X$-spectra and hence  the composite
\begin{equation}
\label{eqn:DSpecEquiv}
\Sigma^n_X P_\ast
\longrightarrow
\Omega_X (\Sigma^{n+1}_XP_\ast)
\longrightarrow
\Omega_X (\mathcal{R}'\Sigma^{n+1}_X P_\ast)
\longrightarrow
\Omega_X(\mathcal{R} d(\Sigma)P)_{n+1}(\ast)
\end{equation}
is a stable weak equivalence in $\mathrm{Sp}^\Sigma_X$.
The morphism of double $X$-spectra $d(\Sigma) P \to \mathcal{R}d(\Sigma) P $ induces a commuting diagram of symmetric $X$-spectra
\[
\begin{tikzcd}
\Sigma^n_X P_\ast
\ar[r]
\ar[d]
&
\Omega_X (\Sigma^{n+1}_X P_\ast)
\ar[d]
\\
(\mathcal{R}d(\Sigma)P)_n(\ast)
\ar[r]
&
\Omega_X (\mathcal{R}d(\Sigma)P)_{n+1}(\ast)
\end{tikzcd}
\]
for each $n \geq 0$.
The left vertical arrow is a stable weak equivalence of symmetric $X$-spectra, as is the composite of the top horizontal and right vertical arrows since it coincides with \eqref{eqn:DSpecEquiv}.
The bottom horizontal arrow is therefore a stable weak equivalence of stably fibrant symmetric $X$-spectra, thus a level weak equivalence.
In particular, this shows  that $\{(\mathcal{R} d(\Sigma) P)_k(n)\}_{k\geq 0}$ is an $\Omega_X$-spectrum for each $n\geq 0$.

The morphism of double $X$-spectra $d(\Sigma) P \to \mathcal{R}d(\Sigma)P$ is an acyclic cofibration in the symmetric projective model structure, hence also in the twice stable model structure.
The lifting diagram of double $X$-spectra
\[
\begin{tikzcd}
d(\Sigma) P
\ar[d, rightarrowtail, "\sim"']
\ar[r, "\sim"]
&
\mathcal{S}
d(\Sigma)P
\\
\mathcal{R}
d(\Sigma) P
\ar[ur, dashed, "\exists"']
\end{tikzcd}
\]
gives rise to a stable weak equivalence of stably fibrant double $X$-spectra $\mathcal{R}d(\Sigma)P\to \mathcal{S}d(\Sigma)P$.
As such stable weak equivalences are precisely the level weak equivalences of double $X$-spectra, the composite
\[
P
\longrightarrow
\mathrm{ev}^\Sigma_0(\mathcal{R}d(\Sigma)P)
\longrightarrow
\mathrm{ev}^\Sigma_0(\mathcal{S}d(\Sigma)P)
\]
is a stable weak equivalence in $\mathrm{Sp}^\Sigma_X$.
We have thus shown that $(d(\Sigma)\dashv \mathrm{ev}^\Sigma_0)$ is a Quillen equivalence.

To prove that $(d(\mathbb{N})\dashv \mathrm{ev}^\mathbb{N}_0)$ is also a Quillen equivalence we use the same argument \emph{mutatis mutandis}.

Finally, associated to any map of simplicial sets $f\colon X\to X'$ there is an adjunction of categories of double spectra $(f_!\dashv f^\ast)\colon \mathrm{Sp}^\mathcal{D}_X\to\mathrm{Sp}^\mathcal{D}_{X'}$ which is obtained by levelwise application of the functors $(f_!\dashv f^\ast)\colon \mathrm{sSet}_{\dslash X}\to \mathrm{sSet}_{\dslash X'}$. 
As for symmetric and sequential parametrised spectra, the adjunction $(f_!\dashv f^\ast)$ is Quillen for the twice stable model structure and is a Quillen equivalence if $f$ is a weak equivalence; we omit the proof.
Forming pullbacks commutes with $\mathrm{ev}^\Sigma_0$ and $\mathrm{ev}^\mathbb{N}_0$, from which it follows that the zig-zag of Quillen equivalences \eqref{eqn:DStab} has the claimed pseudonaturality properties.
\end{proof}
\begin{remark}
The composite functors $d(\mathbb{N})\circ \Sigma^{\infty}_X$ and $d(\Sigma)\circ \mathbf{\Sigma}^\infty_X$ are both naturally isomorphic to the functor $F_{0,0} \colon \mathrm{sSet}_{\dslash X}\to \mathrm{Sp}^\mathcal{D}_X$ introduced in the proof above.
In particular for any $Y,Z\in \mathrm{sSet}_{\dslash X}$ we have isomorphisms of abelian groups
\[
Ho(\mathrm{Sp}^\Sigma_X)\big(
\mathbf{\Sigma}^{\infty}_X Y, \mathbf{\Sigma}_X^\infty Z\big) \cong 
Ho(\mathrm{Sp}^\mathcal{D}_X)\big(
F_{0,0} Y, F_{0,0} Z\big) 
\cong
Ho(\mathrm{Sp}^\mathbb{N}_X)\big(
{\Sigma}^{\infty}_X Y, {\Sigma}^\infty_X Z\big) 
\]
and, more generally, $Ho(\mathrm{Sp}^\Sigma_X)\big(
\mathbf{\Sigma}^{\infty-k}_X Y, \mathbf{\Sigma}^{\infty-l}_X Z\big) \cong Ho(\mathrm{Sp}^\mathbb{N}_X)\big(
{\Sigma}^{\infty-k}_X Y, {\Sigma}^{\infty -l}_X Z\big) $ for $k,l \geq 0$.
\end{remark}

We conclude this section by showing that $\mathrm{Sp}^\mathbb{N}_X$ and $\mathrm{Sp}^\Sigma_X$ are equivalent presentations of the $\infty$-category of $X$-parametrised spectra.
\begin{remark}[Model categories as presentations of $\infty$-categories]
\label{rem:ooCatsFromModelCats}
To any model category $M$ there is associated an $\infty$-category $M^\infty$ obtained by formally inverting weak equivalences.
If $M$ is a $\mathrm{sSet}$-model category then the full subcategory $M^\circ\hookrightarrow M$ on fibrant-cofibrant objects is enriched in Kan complexes and
$M^\infty = N_\Delta(M^\circ)$ is obtained by applying the homotopy coherent nerve functor $N_\Delta\colon \mathrm{sSetCat}\to\mathrm{sSet}$.
If $M$ is a combinatorial $\mathrm{sSet}$-model category then $M^\infty$ is a presentable $\infty$-category.
For every Quillen adjunction $(L\dashv R)\colon M\to N$ there is an associated adjunction of $\infty$-categories $(L^\infty\dashv R^\infty)\colon M^\infty\to N^\infty$, which is an adjoint equivalence if $(L\dashv R)$ is a Quillen equivalence.
\end{remark}

Theorem \ref{thm:LocalSpecComp} implies that we have a zig-zag of adjoint equivalences of stable presentable $\infty$-categories
\[
\big(\mathrm{Sp}^\mathbb{N}_X\big)^\infty
\longrightarrow
\big(\mathrm{Sp}^\mathcal{D}_X\big)^\infty
\longleftarrow
\big(\mathrm{Sp}^\Sigma_X\big)^\infty
\]
(see Remark \ref{rem:ooCatsFromModelCats}).
Writing $\mathrm{Sp}^\infty_X := (\mathrm{Sp}^\mathbb{N}_X)^\infty$, our goal is now to show that $\mathrm{Sp}^\infty_X$ is equivalent to the $\infty$-categorical stabilisation $\mathrm{Sp}(\mathrm{sSet}_{\dslash X}^\infty)$ in the sense of Lurie  \cite{lurie_higher_2017}.
Since the $\infty$-categories $\mathrm{sSet}_{\dslash X}^\infty$ and $\mathcal{S}_{\dslash X}$ are equivalent, both describing the homotopy theory of retractive spaces over $X$, this implies that the model categories $\mathrm{Sp}^\mathbb{N}_X$ and $\mathrm{Sp}^\Sigma_X$ are both presentations of the $\infty$-category of $X$-parametrised spectra.
\begin{remark}
\label{rem:UniversalProp}
Recall that for a pointed $\infty$-category $\mathcal{C}$ with finite limits, the stabilisation is the $\infty$-category
\[
\mathrm{Sp}(\mathcal{C})\cong \mathrm{lim}
\left(\!
\begin{tikzcd}
  \dotsb 
  \ar[r, "\Omega_\mathcal{C}"]
  &
  \mathcal{C}
  \ar[r, "\Omega_\mathcal{C}"]
  &
  \mathcal{C}
\end{tikzcd}\!\right).
\]
The $\infty$-category $\mathrm{Sp}(\mathcal{C})$ is \emph{stable} in the sense that forming loopspace objects is an auto-equivalence, and the construction furnishes a left exact functor $\Omega^\infty_\mathcal{C}\colon \mathrm{Sp}(\mathcal{C})\to\mathcal{C}$.

If $\mathcal{C}$ is presentable, then $\mathrm{Sp}(\mathcal{C})$ is too and $\Omega^\infty_\mathcal{C}$ admits a left adjoint $\Sigma^\infty_\mathcal{C}$ which enjoys the following universal property: precomposition with $\Sigma^\infty_\mathcal{C}$ induces a natural equivalence of $\infty$-categories
\[
\mathrm{Fun}^L (\mathrm{Sp}(\mathcal{C}), \mathcal{D})\longrightarrow
\mathrm{Fun}^L (\mathcal{C}, \mathcal{D})
\]
for any stable  presentable $\infty$-category $\mathcal{D}$ \cite[Corollary 1.4.4.5]{lurie_higher_2017}.
The superscript \lq\lq$L$'' means functors that are left adjoints or, equivalently, are colimit-preserving.
\end{remark}
\begin{theorem}
\label{thm:FibStabisooStab}
There is a natural equivalence of $\infty$-categories $\mathrm{Sp}(\mathrm{sSet}_{\dslash X}^\infty)\to \mathrm{Sp}^\infty_X$ for all $X$.
\end{theorem}
\begin{proof}
$\mathrm{Sp}^\infty_X$ is stable and presentable, so the universal property of Remark \ref{rem:UniversalProp} applied to the functor $(\Sigma^\infty_X\colon \mathrm{sSet}_{\dslash X}\to \mathrm{Sp}^\mathbb{N}_X)^\infty$ gives rise to a colimit-preserving functor $\kappa_X\colon \mathrm{Sp}(\mathrm{sSet}_{\dslash X}^\infty)\to \mathrm{Sp}^\infty_X$.
If $f\colon X\to X'$ is a map of simplicial sets, applying the universal property to 
\[
\Sigma^\infty\circ f_!^\infty \colon \mathrm{sSet}_{\dslash X}^\infty\longrightarrow \mathrm{Sp}(\mathrm{sSet}^\infty_{\dslash X'})
\] 
gives rise to a colimit-preserving functor $\mathrm{Sp}(f_!)\colon \mathrm{Sp}(\mathrm{sSet}_{\dslash X}^\infty)\to \mathrm{Sp}(\mathrm{sSet}_{\dslash X'}^\infty)$. 
Using the universal property once more, we find that the functors $\kappa_{X'}\circ \mathrm{Sp}(f_!)$ and $f_!^\infty\circ \kappa_X$ are equivalent.
By using the naturality of the universal property, it is possible to choose equivalences $\kappa_{X'}\circ \mathrm{Sp}(f_!)\cong f_!^\infty \circ \kappa_X$ exhibiting $\kappa_X$ as the component at $X$ of a natural transformation of functors $\mathrm{sSet}^\infty\to \mathrm{Cat}_\infty$.
We omit the details.

To see that $\kappa_X$ is an equivalence we give a streamlined version of the proof of \cite[Proposition 4.15]{robalo_noncommutative_2012}.
The category $\mathrm{sSet}^+$ of marked simplicial sets can be endowed with a $\mathrm{sSet}$-model structure for which the functor $u\colon \mathrm{sSet}^+\to \mathrm{sSet}$ forgetting markings is a right Quillen equivalence with respect to the Joyal model structure on simplicial sets \cite[Theorem 3.1.5.1]{lurie_higher_2009}. 
If $\mathcal{C}$ is an $\infty$-category (that is, fibrant object for the Joyal model structure), then we write $\mathcal{C}^\natural$ for the simplicial set $\mathcal{C}$ with equivalences as the marked $1$-simplices.
Then $\mathcal{C}^\natural\in \mathrm{sSet}^+$ is fibrant and is sent to $\mathcal{C}$ under by the forgetful functor.

The homotopy limit defining $\mathrm{Sp}(\mathrm{sSet}_{\dslash X}^\infty)$ can be computed as the pullback
\[
\begin{tikzcd}
  & 
  \displaystyle\prod_{n\in \mathbb{N}}\mathrm{Fun}^\sharp (\Delta^1 , \mathrm{sSet}^\natural_{\dslash X})
  \ar[d, "\mathrm{ev}_0\times \mathrm{ev}_1"]
  \\
  \displaystyle\prod_{n\in\mathbb{N}}
  \mathrm{sSet}_{\dslash X}^\natural
  \ar[r, "{(1,\Omega_X)}"]
  &
  \displaystyle\prod_{n\in\mathbb{N}}
  \mathrm{sSet}_{\dslash X}^\natural
  \times \mathrm{sSet}_{\dslash X}^\natural
\end{tikzcd}
\]
\cite[VI~Lemma 1.12]{goerss_simplicial_2009}
In this diagram, the top right-hand object is the $\Delta^1$-cotensor and $\mathrm{ev}_0\times \mathrm{ev}_1$ is the product of the maps obtained by cotensoring with $\partial\Delta^1 \hookrightarrow \Delta^1$, so is a fibration.
The map $(1,\Omega_X)$ is the product 
\[
\begin{tikzcd}
  \displaystyle\prod_{n\in \mathbb{N}}\mathrm{sSet}_{\dslash X}^\natural
  \ar[r, "{\mathrm{pr}_n}"]
  &
  \mathrm{sSet}_{\dslash X}^\natural
  \ar[r, "{\mathrm{id} \times \Omega_X}"]
  &
  \displaystyle\underbrace{\mathrm{sSet}_{\dslash X}^\natural}_{n}\times \displaystyle\underbrace{\mathrm{sSet}_{\dslash X}^\natural}_{n-1}
\end{tikzcd}
\]
where $\mathrm{pr}_n$ is the projection onto the $n$-th factor and the braces indicate indices in the product over $\mathbb{N}$.
In terms of this pullback, we can present the right adjoint $\rho_X$ to $\kappa_X$ explicitly as follows.
Regard $\mathrm{sSet}^{\Delta^1}_{\dslash X}$ as equipped with the projective model structure and let  $(\mathrm{sSet}^{\Delta^1}_{\dslash X})^\circ_\sim\hookrightarrow (\mathrm{sSet}^{\Delta^1}_{\dslash X})^\circ$ be the full Kan-complex-enriched subcategory on the morphisms that are equivalences. 
Then there is a commuting diagram of categories enriched in Kan complexes
\begin{equation}
\label{eqn:OOSpecCompDiag}
\begin{tikzcd}
  \big(\mathrm{Sp}^\mathbb{N}_{X}\big)^\circ
  \ar[d, "{p_\mathbb{N}}"']
  \ar[r, "{\mathbf{L}\sigma^\vee_\mathbb{N}}"]
  & 
  \displaystyle\prod_{n\in \mathbb{N}}\big(\mathrm{sSet}^{\Delta^1}_{\dslash X}) \big)_\sim^\circ
  \ar[d, "{\mathrm{ev}_0\times \mathrm{ev}_1}"]
  \\
  \displaystyle\prod_{n\in\mathbb{N}}
  \mathrm{sSet}_{\dslash X}^\circ
  \ar[r, "{(1,\mathbf{L}\Omega_X)}"]
  &
  \displaystyle\prod_{n\in\mathbb{N}}
  \mathrm{sSet}_{\dslash X}^\circ
  \times \mathrm{sSet}_{\dslash X}^\circ\,,
\end{tikzcd}
\end{equation}
in which
\begin{itemize}
  \item $p_\mathbb{N}$ is the functor that sends $A\mapsto (A_n)_{n\in \mathbb{N}}$, noting that each $A_n$ is cofibrant-fibrant in $\mathrm{sSet}_{\dslash X}$;
  \item $\mathbf{L}\sigma^{\vee}_\mathbb{N}$ sends the fibrant-cofibrant sequential $X$-spectrum $A$ to the collection of arrows $\overline{\sigma}_n^\vee(A)$ in the functorial factorisation
  \[
  \begin{tikzcd}
    A_n
    \ar[r, rightarrowtail, "{\overline{\sigma}_n^\vee(A)}"]
    &
    \mathcal{Q}\Omega_X A_{n+1}
    \ar[r, twoheadrightarrow, "{\sim}"]
    &
    \Omega_X A_{n+1}
  \end{tikzcd}
  \]
  of the adjoint structure map $\sigma^\vee_n\colon A_n \to\Omega_X A_{n+1}$ of $A$ into a (necessarily acyclic) cofibration followed by an acyclic fibration;
  \item $(1, \mathbf{L}\Omega_X)$ sends $(Y_n)_{n\in \mathbb{N}}$ to the sequence $(Y_n, \mathcal{Q} \Omega_X Y_{n+1})_{n\in \mathbb{N}}$, with $\mathcal{Q}\Omega_X Y_{n+1}$ arising from the functorial factorisation of $0_X\to \Omega_X Y_{n+1}$ into a cofibration followed by an acyclic fibration; and
  \item $\mathrm{ev}_0\times \mathrm{ev}_1$ is the evaluation map $(A_n \to B_n)_{n\in \mathbb{N}}\mapsto (A_n, B_n)_{n\in \mathbb{N}}$.
\end{itemize}
Note that the diagram commutes since we are working with functorial factorisations.
Applying the homotopy coherent nerve $N_\Delta$, there is a weak equivalence 
\[
N_\Delta \Big(\big(\mathrm{sSet}^{\Delta^1}_{\dslash X}) \big)_\sim^\circ\Big)\xrightarrow{\;\;\sim\;\;} u\mathrm{Fun}^\sharp (\Delta^1 , \mathrm{sSet}^\natural_{\dslash X})
\]
for the Joyal model structure \cite[Proposition 4.2.4.4]{lurie_higher_2009},
so we have a commuting diagram of simplicial sets
\begin{equation}
\label{eqn:OOSpecCompDiag2}
\begin{tikzcd}
  \mathrm{Sp}^\infty_X
  \ar[r]
  \ar[d]
  & 
  \displaystyle\prod_{n\in \mathbb{N}}u\mathrm{Fun}^\sharp (\Delta^1 , \mathrm{sSet}^\natural_{\dslash X})
  \ar[d,"{\mathrm{ev}_0\times \mathrm{ev}_1}"]
  \\
  \displaystyle\prod_{n\in\mathbb{N}}
  \mathrm{sSet}_{\dslash X}^\infty
  \ar[r, "{(1,\Omega_X)}"]
  &
  \displaystyle\prod_{n\in\mathbb{N}}
  \mathrm{sSet}_{\dslash X}^\infty
  \times \mathrm{sSet}_{\dslash X}^\infty\,,
\end{tikzcd}
\end{equation}
presenting the functor $\rho_X\colon \mathrm{Sp}^\infty_X\to \mathrm{Sp}(\mathrm{sSet}_{\dslash X}^\infty)$.
From this diagram it is easy to see that $\rho_X$ is essentially surjective.
Indeed, an object $A'$ of $\mathrm{Sp}(\mathrm{sSet}_{\dslash X}^\infty)$ is a collection of fibrant-cofibrant objects $A_n \in \mathrm{sSet}_{\dslash X}$ equipped with weak equivalences $A_n \to \Omega_X A_{n+1}$ of retractive spaces over $X$.
But this is precisely the data of a fibrant $\Omega_X$-spectrum $A$, and taking a cofibrant replacement $\mathcal{Q}A\to A$ gives an object of $\mathrm{Sp}^\infty_X$.
Since the stable equivalence $\mathcal{Q}A \to A$ is a levelwise equivalence (both objects are stably fibrant) we have a choice of equivalence $\rho_X(\mathcal{Q}A)\to A'$ in $\mathrm{Sp}(\mathrm{sSet}_{\dslash X})^\infty$.

To complete the proof, we show that $\rho_X$ is fully-faithful.
For cofibrant-fibrant sequential $X$-spectra $A, B$, the $\infty$-categorical hom space $\mathrm{Hom}_{\mathrm{Sp}(\mathrm{sSet}_{\dslash X}^\infty)}(\rho_XA, \rho_X B)$ is computed, via \eqref{eqn:OOSpecCompDiag2}, as the pullback of
\[
\begin{tikzcd}
  &
  \displaystyle\prod_{n\in \mathbb{N}}  \mathrm{Hom}_{\mathrm{Fun}(\Delta^1,  \mathrm{sSet}_{\dslash X}^\infty)}  \big(\overline{\sigma}^\vee_n (A),   \overline{\sigma}^\vee_n(B)\big)
  \ar[d]
  \\
  \displaystyle\prod_{n\in \mathbb{N}}
  \mathrm{Hom}_{\mathrm{sSet}_{\dslash X}^\infty}(A_n, B_n)
  \ar[r]
  &
  \displaystyle\prod_{n\in \mathbb{N}}
  \mathrm{Hom}_{\mathrm{sSet}_{\dslash X}^\infty}(A_n, B_n)
  \times
  \mathrm{Hom}_{\mathrm{sSet}_{\dslash X}^\infty}(\mathcal{Q}\Omega_X A_{n+1}, \mathcal{Q}\Omega_X B_{n+1})\,.
\end{tikzcd}
\]
By inspection, this pullback is isomorphic to the pullback of
\begin{equation}
\label{eqn:SpooHomSpace}
\begin{tikzcd}
  &
  \displaystyle\prod_{n\in \mathbb{N}} \mathrm{Hom}_{\mathrm{sSet}_{\dslash X}^\infty} (\mathcal{Q}\Omega_X A_{n+1} ,\mathcal{Q}\Omega_X B_{n+1})
  \ar[d]
  \\
  \displaystyle\prod_{n\in \mathbb{N}} \mathrm{Hom}_{\mathrm{sSet}_{\dslash X}^\infty} (A_n ,B_{n})
  \ar[r]
  &
  \displaystyle\prod_{n\in \mathbb{N}} \mathrm{Hom}_{\mathrm{sSet}_{\dslash X}^\infty} (A_n ,\mathcal{Q}\Omega_X B_{n+1})\,,
\end{tikzcd}
\end{equation}
in which the horizontal arrow is given by postcomposing with the maps $\overline{\sigma}^\vee_n(B) \colon B_n \to \mathcal{Q}\Omega_X B_{n+1}$ and the vertical arrow is obtained by precomposing with the maps $\overline{\sigma}^\vee_n (A)\colon A_n \to \mathcal{Q}\Omega_X A_{n+1}$.
Since the $\overline{\sigma}^\vee_n (A)$ are cofibrations, the vertical map is a fibration so that \eqref{eqn:SpooHomSpace} presents a homotopy pullback.

On the other hand, the Kan complex $\mathrm{map}_{\mathrm{Sp}^\mathbb{N}_X}(A,B)$ is the pullback of the cospan
\begin{equation}
\label{eqn:SpKanHom}
  \displaystyle\prod_{n\in \mathbb{N}}\mathrm{map}_{\mathrm{sSet}_{\dslash X}}(A_n, B_n)
  \xrightarrow{\;\;\beta\;\;}
  \displaystyle\prod_{n\in \mathbb{N}}\mathrm{map}_{\mathrm{sSet}_{\dslash X}}(A_n, \Omega_X B_n)
  \xleftarrow{\;\;\alpha\;\;}
  \displaystyle\prod_{n\in \mathbb{N}}\mathrm{map}_{\mathrm{sSet}_{\dslash X}}(A_n, B_n)\,.
\end{equation}
In this diagram, $\beta$ is obtained by postcomposition with the adjoint structure maps of $B$ and $\alpha$ is the product of the maps
\[
\mathrm{map}_{\mathrm{sSet}_{\dslash X}}(A_{n+1}, B_{n+1})\longrightarrow
\mathrm{map}_{\mathrm{sSet}_{\dslash X}}(\Sigma_X A_n, B_{n+1})\cong 
\mathrm{map}_{\mathrm{sSet}_{\dslash X}}(A_n, \Omega_X B_{n+1})
\]
each of which is a fibration, since $\sigma_n\colon \Sigma_X A_n \to A_{n+1}$ is a cofibration (Remark \ref{rem:ProjCofib}).
Using the functorial factorisation of morphisms in $\mathrm{sSet}_{\dslash X}$ into cofibrations followed by acyclic fibrations, the cospan of \eqref{eqn:SpKanHom} is termwise weakly equivalent to the cospan
\[
   \prod_{n\in \mathbb{N}}
   \mathrm{map}_{\mathrm{sSet}_{\dslash X}}(A_n, B_n)
\xrightarrow{\beta}
\prod_{n\in \mathbb{N}}\mathrm{map}_{\mathrm{sSet}_{\dslash X}}(A_n, \mathcal{Q}\Omega_X B_n)
\xleftarrow{\alpha}
 \prod_{n\in \mathbb{N}}\mathrm{map}_{\mathrm{sSet}_{\dslash X}}(\mathcal{Q}\Omega_X A_{n+1}, \mathcal{Q}\Omega_X B_{n+1})\,,
\]
using that $\overline{\sigma}^\vee_n(A)$ and $\overline{\sigma}^\vee_n(B)$ are acyclic cofibrations.
Now, for cofibrant-fibrant objects $x,y$ in a simplicial model category $\mathcal{C}$, there is a natural weak equivalence $\mathrm{map}_{\mathcal{C}}(x,y)\to \mathrm{Hom}_{\mathcal{C}^\infty}(x,y)$ \cite[Theorem 2.2.0.1]{lurie_higher_2009}.
Since the pullbacks of \eqref{eqn:SpooHomSpace} and \eqref{eqn:SpKanHom} are also homotopy pullbacks, there is a diagram of weak equivalences
\[
\begin{tikzcd}
  \mathrm{Hom}_{\mathrm{Sp}^\infty_X}(A,B) 
  \ar[r, leftarrow]
  \ar[rr, "{\rho_X}", bend left =15]
  &
  \mathrm{map}_{\mathrm{Sp}^\mathbb{N}_X}(A,B)
  \ar[r]
  &
  \mathrm{Hom}_{\mathrm{Sp}(\mathrm{sSet}_{\dslash X}^\infty)} (\rho_XA, \rho_XB)\,.
\end{tikzcd}
\]
This shows that $\rho_X$ is fully-faithful and hence is an equivalence of $\infty$-categories.
\end{proof}

\begin{remark}
The combinatorial model categories $\mathrm{Sp}^\mathbf{N}_X$ and $\mathrm{Sp}^\Sigma_X$ discussed in this article are additionally
Quillen equivalent to the model categories of parametrised spectra studied in \cite{may_parametrized_2006}.
May and Sigurdsson work with topological spaces, rather than simplicial sets, and their framework is related to ours by stabilising the adjunction
\[
\begin{tikzcd}
\mathrm{sSet}_{\dslash X}
\ar[rr, shift left=1.1ex, "|-|"]
\ar[rr, leftarrow, shift left=-1.1ex, "\bot", "\mathrm{Sing}"']
&&
\mathrm{Top}_{\dslash |X|}
\end{tikzcd}
\]
induced by the Quillen equivalence $(|-|\dashv \mathrm{Sing})\colon \mathrm{sSet}\to \mathrm{Top}$. We do not go into the details here, though note that both approaches present the same $\infty$-category of parametrised spectra by Theorem \ref{thm:FibStabisooStab} and \cite[Appendix B]{ando_parametrized_2018}.
\end{remark}

\subsection{Global theory}
In our work of the last three sections, we have encountered three pseudofunctors $\mathrm{sSet}\to \mathbf{Model}$ encoding parametrised spectra:
\[
\mathrm{Sp}^{\bullet}_{-}\colon 
\big(f\colon X\to Y\big)
\longmapsto
\left(\!
\begin{tikzcd}
\mathrm{Sp}^\bullet_{ X}
\ar[rr, shift left=1.1ex, "f_!"]
\ar[rr, leftarrow, shift left=-1.1ex, "\bot", "f^\ast"']
&&
\mathrm{Sp}^\bullet_{Y}
\end{tikzcd}
\!\right)
\]
for $\bullet = \mathbb{N}, \Sigma$ or $\mathcal{D}$, taking the stable model structure in each case.
According to Theorems \ref{thm:SeqSpecStructureTheorem}, \ref{thm:SymSpecStructureTheorem} and \ref{thm:LocalSpecComp} each of these pseudofunctors sends weak equivalences to Quillen equivalences so is relative (in the terminology of \cite{harpaz_grothendieck_2015}).
\begin{lemma}
\label{lem:SpecPseudoFunAreProper}
For $\bullet\in\{\mathbb{N},\Sigma,\mathcal{D}\}$, the pseudofunctors $\mathrm{Sp}^\bullet_{-}\colon \mathrm{sSet}\to\mathbf{Model}$ are proper.
\end{lemma}
\begin{proof}
We treat the case of sequential parametrised spectra, the arguments for the other two cases being completely analogous.
For right properness, we observe that if $f\colon X\to X'$ is an acyclic fibration then $f^\ast \colon \mathrm{Sp}^\mathbb{N}_{X'}\to\mathrm{Sp}^\mathbb{N}_X$ is right and left Quillen, so preserves all stable weak equivalences.

For any map $f\colon X\to X'$, the pushforward functor $f_!\colon \mathrm{sSet}_{\dslash X}\to \mathrm{sSet}_{\dslash X'}$ preserves weak equivalences.
Consequently, the induced functor $f_!\colon \mathrm{Sp}^\mathbb{N}_X\to \mathrm{Sp}^\mathbb{N}_{X'}$ preserves levelwise weak equivalences.
Let $\mathcal{R}$ be a fibrant replacement functor for the sequential $X$-stable model structure such that the natural weak equivalence $A\to \mathcal{R}$ is a cofibration for all $A$.
If $\psi\colon A\to B$ is a stable weak equivalence of sequential $X$-spectra, then $\mathcal{R}(\psi)$ is a levelwise weak equivalence of fibrant $\Omega_X$-spectra.
Applying $f_!$ to the naturality square of the natural acyclic cofibration $\mathrm{id}\Rightarrow\mathcal{R}$ yields
\[
\begin{tikzcd}
  f_! A
  \ar[r, "f_!\psi"] 
  \ar[d, rightarrowtail, "\sim"'] 
  &  
  f_! B
  \ar[d, rightarrowtail, "\sim"]
  \\
  f_! \mathcal{R}(A)
  \ar[r, "f_! \mathcal{R}(\psi)"]
  &
  f_! \mathcal{R}(B)\,,
\end{tikzcd}
\]
where the vertical arrows are acyclic cofibrations since $f_!$ is left Quillen.
But $f_!$ preserves levelwise weak equivalences, so that $f_! \mathcal{R}(\psi)$ is a levelwise weak equivalence (though $f_!\mathcal{R}(A)$ and $f_!\mathcal{R}(B)$ are generally not $\Omega_{X'}$-spectra).
It follows that all pushforward functors $f_!$ preserve stable weak equivalences, implying left properness as a special case.
\end{proof}
\begin{definition}
For $\bullet\in \{\mathbb{N},\Sigma, \mathcal{D}\}$ we define the \emph{global category of parametrised $\bullet$-spectra}  as the Grothendieck construction
\[
\mathrm{Sp}^\bullet_\mathrm{sSet} :=\int_{X\in \mathrm{sSet}}\mathrm{Sp}^\bullet_{X}\,.
\]
Objects of $\mathrm{Sp}^\bullet_\mathrm{sSet}$ are pairs $(X,A)$ where $X\in \mathrm{sSet}$ and $A\in \mathrm{Sp}^\bullet_X$
Morphisms $(X,A)\to (X', B)$ are pairs of a morphism of simplicial sets $f\colon X\to X'$ and a morphism of $X'$-spectra $\psi\colon f_! A\to B$ or, equivalently, a morphism of $X$-spectra $\psi^\vee\colon A\to f^\ast B$.
\end{definition}
\begin{theorem}
\label{thm:GlobParamSpec}
For $\bullet\in \{\mathbb{N},\Sigma, \mathcal{D}\}$, $\mathrm{Sp}^\bullet_{\mathrm{sSet}}$ has a \emph{global model structure} with respect to which a morphism $(f,\psi)\colon (X,A)\to (X',B)$ is
\begin{itemize}
  \item a weak equivalence if $f$ and $\psi$ are weak equivalences in $\mathrm{sSet}$ and $\mathrm{Sp}^\bullet_{X'}$ respectively;
  \item a cofibration if $f$ and $\psi$ are cofibrations in $\mathrm{sSet}$ and $\mathrm{Sp}_{X'}^\bullet$ respectively; and
  \item a fibration if $f$ and $\psi^\vee$ are fibrations in $\mathrm{sSet}$ and $\mathrm{Sp}^\bullet_{X}$ respectively. 
\end{itemize}
Moreover, there is a zig-zag of Quillen equivalences
\[
\begin{tikzcd}
 \mathrm{Sp}^\Sigma_\mathrm{sSet}
\ar[rr, shift left=1.1ex, ""]
\ar[rr, leftarrow, shift left=-1.1ex, "\bot", ""']
&&
\mathrm{Sp}^\mathcal{D}_\mathrm{sSet}
\ar[rr, leftarrow, shift left=1.1ex, ""]
\ar[rr, shift left=-1.1ex, "\bot", ""']
&&
\mathrm{Sp}^\mathbb{N}_\mathrm{sSet}
\end{tikzcd}
\]
with respect to the global model structures.
\end{theorem}
\begin{proof}
Apply Theorems 3.0.12 and 4.1.3 of \cite{harpaz_grothendieck_2015} to Theorem \ref{thm:LocalSpecComp}.
\end{proof}

\begin{corollary}
Let $\bullet\in\{\mathbb{N},\Sigma, \mathcal{D}\}$.
The projection functor $p\colon \mathrm{Sp}^\bullet_\mathrm{sSet}\to \mathrm{sSet}$ is left and right Quillen.
\end{corollary}
\begin{proof}
The projection functor $p\colon (X,A)\mapsto X$ preserves cofibrations, fibrations and weak equivalences. 
The functor $0_{-}\colon X\mapsto (X, 0_X)$ is a two-sided adjoint.
\end{proof}
\begin{corollary}
\label{cor:GlobalSpFibreInclusion}
Let $\bullet\in \{\mathbb{N},\Sigma\}$.
For any simplicial set $X$, the fibre inclusion $\imath_X \colon \mathrm{Sp}^\bullet_X\to \mathrm{Sp}^\bullet_\mathrm{sSet}$ is left and right Quillen.
\end{corollary}
\begin{proof}
The fibre inclusion $\imath_X$ preserves colimits, limits, cofibrations, fibrations and weak equivalences.
$\mathrm{Sp}^\bullet_\mathrm{sSet}$ is locally presentable (Corollaries \ref{cor:GlobSeqSpecLocPres} and \ref{cor:GlobSymSpecLocPres}) so that $\imath_X$ has a left and a right adjoint.
\end{proof}
\begin{remark}
\label{rem:RAtoInisGlobSect}
The right adjoint to $\imath \colon \mathrm{Sp}^\bullet\to \mathrm{Sp}^\bullet_\mathrm{sSet}$ is the functor
$
(X,A)\mapsto X_\ast A
$
computing global section spectra.
\end{remark}

\begin{corollary}
\label{cor:GlobalSigmaOmega}
Let $\bullet\in \{\mathbb{N},\Sigma\}$.
For any $k\geq 0$ there is a Quillen adjunction
\[
\begin{tikzcd}
\mathrm{sSet}_{\dslash \mathrm{sSet}}
\ar[rr, shift left=1.1ex, "\Sigma^{\infty-k}_{-}"]
\ar[rr, leftarrow, shift left=-1.1ex, "\bot", "\widetilde{\Omega}^{\infty-k}_{-}"']
&&
\mathrm{Sp}^\bullet_{\mathrm{sSet}}
\end{tikzcd}
\]
with respect to the global model structures that preserves projection functors to $\mathrm{sSet}$.
\end{corollary}
\begin{proof}
Taking $\bullet=\mathbb{N}$ for the sake of concreteness, observe that for any map of simplicial sets $f\colon X\to X'$ there is a natural isomorphism of functors $\Sigma^{\infty-k}_{X'}\circ f_!\cong f_! \circ \Sigma^{\infty-k}_X$ since $f_!$ preserves $\mathrm{sSet}_\ast$-tensors.
The assignment
\[
X\longmapsto \left((\Sigma^{\infty-k}_X\dashv \widetilde{\Omega}^{\infty-k}_X)\colon \mathrm{sSet}_{\dslash X}\to \mathrm{Sp}^\mathbb{N}_X\;\right)
\]
defines the components of a pseudonatural transformation $\mathrm{sSet}_{\dslash -}\Rightarrow \mathrm{Sp}^\bullet_{-}$ (the requisite coherence conditions can be verified using the universal property of the pullback and essential uniqueness of adjoints).
Applying the Grothendieck construction in view of \cite[Theorem 4.1.3]{harpaz_grothendieck_2015} completes the proof.
\end{proof}

\begin{remark}
For a simplicial set $X$, its \emph{cotangent complex} $\mathbb{L}_X$ is the result of the applying the following composite left Quillen functor to $X$:
\[
\begin{tikzcd}
\mathrm{sSet}
\ar[rr, shift left =1.1ex, "\mathrm{const.}"]
\ar[rr, leftarrow, shift left=-1.1ex, "\bot", "\mathrm{ev}_0"']
&&
\mathrm{Fun}(\Delta^1, \mathrm{sSet})
\ar[rr, shift left=1.1ex, "(\mathtt{0})_{+(\mathtt{1})}"]
\ar[rr, leftarrow, shift left=-1.1ex, "\bot", "U"']
&&
\mathrm{sSet}_{\dslash \mathrm{sSet}}
\ar[rr, shift left=1.1ex, "\Sigma^{\infty}_{-}"]
\ar[rr, leftarrow, shift left=-1.1ex, "\bot", "\widetilde{\Omega}^{\infty}_{-}"']
&&
\mathrm{Sp}^\mathbb{N}_{\mathrm{sSet}}
\end{tikzcd}
\]
(see Section \ref{S:CombTangent}, \cite[Chapter 7]{lurie_higher_2017}).
The cotangent complex of $X$ is therefore simply the fibrewise suspension spectrum $\mathbb{L}_X \cong \Sigma^\infty_X X_{+X}$ obtained by replacing each of the fibres of the identity fibration $\mathrm{id}_X\colon X\to X$ with the sphere spectrum.
Pushing forward along the terminal map $X\colon X\to \ast$, we find that $X_!\mathbb{L}_X$ is stably equivalent to the suspension spectrum $\Sigma^\infty_+ X$.

In algebraic contexts, such as for $\mathbf{E}_\infty$-ring spectra \cite[Section 7.4]{lurie_higher_2017}, the cotangent complex governs infinitesimal deformations.
In the setting of parametrised spectra, one might reasonably expect that there is some sense in which $\mathbb{L}_X$ governs infinitesimal deformations of the space $X$ (see also Remark \ref{rem:Goo}).
As a nod in this direction, observe that in the present setting the analogue of the space of derivations of $X$ with values in a (fibrant) $X$-spectrum $\mathcal{E}$ is 
\[
\mathrm{Sp}^\mathbb{N}_X \big(\mathbb{L}_X,
\mathcal{E} \big)
\cong
\left\{
\begin{tikzcd}[row sep= small]
&
\Omega^\infty_X \mathcal{E}
\ar[d]
\\
X
\ar[ur, dashed, bend left =15]
\ar[r, equals]
&
X
\end{tikzcd}
\right\}\,,
\]
the space of sections
computing the cohomology of $X$ with coefficients in the local system on $X$ determined by $\mathcal{E}$.
We plan to return to these ideas in future work.
\end{remark}

\subsubsection{Sequential stabilisation}
In this section we give another construction of the global category of parametrised sequential spectra which is more algebraic in nature.
We use this second construction to show that the global category of parametrised sequential spectra is combinatorial and simplicial.

\begin{construction}
Recalling the notation of Construction \ref{cons:GlobRetSpaceAlg}, let $\mathrm{RetSeq}$ be the pullback category
\[
\begin{tikzcd}
  \mathrm{RetSeq}\ar[r]\ar[d]
  &
  \mathrm{Fun}(\mathbb{N}\times \partial\Delta^2,\mathrm{sSet})
  \ar[d, "{(\mathrm{id}\times i_{\{0,2\}})^\ast}"]
  \\
  c(\mathrm{sSet})
  \ar[r, hookrightarrow]
  &
  \mathrm{Fun}(\mathbb{N}\times \Delta^{\{0,2\}},\mathrm{sSet})\,,
\end{tikzcd}
\]
with $c(\mathrm{sSet})\hookrightarrow \mathrm{Fun}(\mathbb{N}\times \partial\Delta^{\{0,2\}},\mathrm{sSet})$ the full subcategory on those functors $\mathbb{N}\times \Delta^1\to \mathrm{sSet}$ that are constant on the identity $(\mathrm{id}_X\colon X\to X)$ for some simplicial set $X$.
Objects of $\mathrm{RetSeq}$ are thus sequences of retractive spaces over a fixed base space.

There is an accessible monad $T^\mathrm{sp}$ on $\mathrm{RetSeq}$ which sends $\{(X,Y_n)\}_{n\in \mathbb{N}}$ to the sequence of retractive spaces over $X$ whose $n$-th term is the colimit of the diagram
\[
\begin{tikzcd}[sep =small]
  X
  \ar[d]
  \ar[dr, bend left= 5]
  \ar[drrr, bend left= 10]
  \ar[drrrr, bend left =15]
  \\
  \Sigma^n_X Y_0 
  &
  \Sigma^{n-1}_X Y_1
  &
  \dotsb
  &
  \Sigma_X Y_{n-1}
  &
  Y_n\,,
\end{tikzcd}
\]
which is a retractive space over $X$.
This defines the functor $T^\mathrm{sp}$ on objects; the action on morphisms is defined in the obvious way, using naturality of the fibrewise suspension functors.
The monad multiplication $T^\mathrm{sp}\circ T^\mathrm{sp}\Rightarrow T^\mathrm{sp}$ is determined at level $n$ by the folding maps
\[
\underbrace{\Sigma^{n-i}_X Y_i\coprod_X \dotsb \coprod_X \Sigma^{n-i}_X Y_i}_{\text{$(n+1-i)$ copies}}
\longrightarrow 
\Sigma^{n-i}_X Y_i\,,
\]
and the monad unit $\mathrm{id}\Rightarrow T^\mathrm{sp}$ is given at level $n$ by the inclusion of the summand $Y_n$.

Let $\{(X,A_n)\}_{n\in \mathbb{N}}$ be a $T^\mathrm{sp}$-algebra object. At level $(n+1)$, the $T^\mathrm{sp}$-algebra structure gives rise to morphisms 
\[
\rho^i_{n+1}\colon \Sigma^{n-i}_X A_i\longrightarrow A_{n+1}
\]
for $0\leq i\leq n+1$.
Writing $\overline{\sigma}_n:= \rho^n_{n+1}\colon (X,\Sigma_X A_n)\to (X, A_{n+1})$, the multiplication axiom implies that $
\rho^i_{n+1} = \overline{\sigma}_n \circ \Sigma_X \overline{\sigma}_{n-1}\circ \dotsb \circ \Sigma^{i-1}_X \overline{\sigma}_{n+1-i}
$  for all $0\leq i \leq n$, whereas the unit axiom implies that each $\rho^n_n$ is the identity.
Specifying a $T^\mathrm{sp}$-algebra on $\{(X, A_n)\}_{n\in \mathbb{N}}$ is the same as saying that the sequence $\{A_n\}_{n\in \mathbb{N}}$ assembles into a sequential $X$-spectrum (see Remark \ref{rem:FibrewiseSeqSpecMonad})---using the $T^\mathrm{sp}$-algebra axioms, one shows that the would-be spectrum structure maps $\overline{\sigma}_n$ cover the identity on $X$.
\end{construction}

\begin{lemma}
The categories $T^\mathrm{sp}\mathrm{-Alg}$ and $\mathrm{Sp}^\mathbb{N}_{\mathrm{sSet}}$ are canonically equivalent.
\end{lemma}

\begin{remark}
\label{rem:SeqSpecGlobalMorphisms}
According to the Lemma, morphisms $(f,\psi)\colon (X,A)\to (X',B)$ in $\mathrm{Sp}^\mathbb{N}_\mathrm{sSet}$ may be equivalently regarded as a collection of commutative diagrams
\[
\begin{tikzcd}
  X
  \ar[r]
  \ar[d, "f"']
  \ar[rr, "\mathrm{id}_X", bend left =25]
  &
  A_ n 
  \ar[r]\ar[d, "\Psi_n"']
  &
  X\ar[d, "f"]
  \\
  X'
  \ar[rr, "\mathrm{id}_{X'}"', bend left=-25]
  \ar[r]
  &
  B_n
  \ar[r]
  &
  X'
\end{tikzcd}
\]
such that the diagrams
\[
\begin{tikzcd}
  X
  \ar[r]\ar[d, "{f}"']
  &
  \Sigma_X A_ n
  \ar[r]
  \ar[d, "{\Sigma_f \Psi_{n}}"']
  &
  A_{n+1} 
  \ar[r]\ar[d, "{\Psi_{n+1}}"]
  &
  X\ar[d, "{f}"]
  \\
  X'
  \ar[r]
  &
  \Sigma_{X'} B_n
  \ar[r]
  &
  B_{n+1}
  \ar[r]
  &
  X'
\end{tikzcd}
\]
commute for all $n\in \mathbb{N}$ ($\Sigma_{f}\Psi_n$ corresponds to $f_! \Sigma_X A_n \to \Sigma_{X'} B_n$ via Remark \ref{rem:GlobalMorphisms}).
\end{remark}

\begin{corollary}
\label{cor:GlobSeqSpecLocPres}
$\mathrm{Sp}^\mathbb{N}_{\mathrm{sSet}}$ is locally presentable.
\end{corollary}
\begin{proof}
$\mathrm{RetSeq}$ is accessible by \cite[Corollary A.2.6.5]{lurie_higher_2009} and closed under small colimits, hence is locally presentable. 
The monad $T^\mathrm{sp}$ preserves all colimits, so that $T^\mathrm{sp}\mathrm{-Alg}$ is accessible \cite[2.78]{adamek_locally_1994} and colimits in $T^\mathrm{sp}\mathrm{-Alg}$ are created by the forgetful functor.
$T^\mathrm{sp}\mathrm{-Alg}$, and hence $\mathrm{Sp}^\mathbb{N}_\mathrm{sSet}$, is thus locally presentable.
\end{proof}

The bulk of the rest of this section is devoted to the proof of the following
\begin{theorem}
\label{thm:SeqSpecGlobalComb}
$\mathrm{Sp}^\mathbb{N}_{\mathrm{sSet}}$ is a left proper combinatorial $\mathrm{sSet}_\ast$-model category.
\end{theorem}
\begin{proof}
Left properness of the global model structure  on $\mathrm{Sp}^\mathbb{N}_{\mathrm{sSet}}$ is proven as in Lemma \ref{lem:GlobRetSpaceComb}, since $\mathrm{sSet}$ and $\mathrm{Sp}^\mathbb{N}_X$ are left proper for all $X$.

As $\mathrm{Sp}^\mathbb{N}_\mathrm{sSet}$ is locally presentable, in order to show that the global model structure is combinatorial we must demonstrate that it is cofibrantly generated.
The small object argument automatically holds in a locally presentable category so it is sufficient to find sets $\mathcal{I}_{\mathbb{N}\mathrm{-Sp}}$ and $\mathcal{J}_{\mathbb{N}\mathrm{-Sp}}$ characterising acyclic fibrations and fibrations respectively via the right lifting property.

As a first approximation, let 
\[
\mathcal{I}_{\mathbb{N}\mathrm{-Sp}}:=
\bigcup_{k\geq 0}\Sigma^{\infty-k}_-(\mathcal{I}^\mathrm{Kan}_{\mathrm{sSet}})
\;\;
\mbox{ and }
\;\;
\mathcal{J}'_{\mathbb{N}\mathrm{-Sp}}:=
\bigcup_{k\geq 0}\Sigma^{\infty-k}_-(\mathcal{J}^\mathrm{Kan}_{\mathrm{sSet}})
\]
be the set of (acyclic) cofibrations defined by applying $\Sigma^{\infty-k}_- \colon \mathrm{sSet}_{\dslash \mathrm{sSet}}\to \mathrm{Sp}^\mathbb{N}_\mathrm{sSet}$ to the set of generating (acyclic) cofibrations of Lemma \ref{lem:GlobRetSpaceComb}.
The global \emph{projective} model structure on $\mathrm{Sp}^\mathbb{N}_\mathrm{sSet}$ is the result of applying the Grothendieck construction for model categories to the pseudofunctor that sends $X$ to the sequential $X$-projective model category.
Arguing as in the proof of Lemma \ref{lem:GlobRetSpaceComb}, $\mathcal{I}_{\mathbb{N}\mathrm{-Sp}}$ and $\mathcal{J}'_{\mathbb{N}\mathrm{-Sp}}$ are seen to be respectively generating cofibrations and generating acyclic cofibrations for the global projective model structure, which is therefore combinatorial and left proper (by the above).

We identify the global model structure as a left Bousfield localisation of the global projective model structure. 
To do this, we first show that the global projective model structure is a $\mathrm{sSet
}_\ast$-enriched in order that we can use the  existence result for Bousfield localisations of left proper combinatorial $\mathrm{sSet}$-model categories.
The category $\mathrm{Sp}^\mathbb{N}_\mathrm{sSet}$ admits $\mathrm{sSet}_\ast$-tensors
\[
K\otimes (X,A) := (X, K\owedge_X A)
\]
inherited from the $\mathrm{sSet}_\ast$-tensoring on the fibre categories $\mathrm{Sp}^\mathbb{N}_X$.
The tensoring bifunctor preserves colimits in each variable and so participates in an adjunction of two variables by the adjoint functor theorem.
For a map of simplicial sets $i\colon K\to L$ and a map $g\colon (X, Y) \to (X',Z)$ in $\mathrm{sSet}_{\mathrm{sSet}}$ there is a natural isomorphism
\[
\Sigma^{\infty-k}_{-} \big(i\,\square\,g\big)
\cong i\, \square\, \Sigma^{\infty-k}_{-} g
\]
for all $k\geq 0$, the pushout-product on the left-hand side being computed with respect to the $\mathrm{sSet}_\ast$-tensoring on $\mathrm{sSet}_{\dslash \mathrm{sSet}}$.
Since the $\mathrm{sSet}_\ast$-tensoring on $\mathrm{sSet}_{\dslash \mathrm{sSet}}$ is a Quillen bifunctor, we deduce that the $\mathrm{sSet}_\ast$-tensoring on $\mathrm{Sp}^\mathbb{N}_\mathrm{sSet}$ also satisfies the pushout-product axiom.
\vspace{5mm}

\noindent {\bf Claim:} the global model structure on $\mathrm{Sp}^\mathbb{N}_\mathrm{sSet}$ is the left Bousfield localisation of the global projective model structure at the set of morphisms
\[
S_{\mathbb{N}} :=\left\{\zeta_{k,\Delta^n } \big((\mathrm{id}|_{\partial\Delta^n})_{+\Delta^n}\big), \zeta_{k,\Delta^n } \big(\mathrm{id}_{+\Delta^n}\big)\right\}_{n,k\geq 0}\,.
\]

\vspace{5mm}

Assuming the veracity of this claim for the moment, it follows that the global model structure is combinatorial and left proper.
To see that the $\mathrm{sSet}_\ast$-tensoring descends to a Quillen bifunctor after Bousfield localisation, note that  there are natural isomorphisms $K\otimes \zeta_{k,X}(Y)\cong \zeta_{k, X}(K\owedge_X Y)$.
It follows that the tensoring $K\otimes (-)$ on $\mathrm{Sp}^\mathbb{N}_{\mathrm{sSet}}$ carries morphisms in $S_{\mathbb{N}}$ to stable weak equivalences, so descends to a left Quillen endofunctor for the global model structure.
For $i\colon K\to L$ a cofibration of pointed simplicial sets and $j\colon (X, A)\to (X',B)$ a weak equivalence for the global model structure on $\mathrm{Sp}^\mathbb{N}_\mathrm{sSet}$, we have a commuting diagram
 \[
 \begin{tikzcd}
    K\otimes (X,A)
    \ar[r]\ar[d, "{\sim}"'] & 
    L\otimes (X,A)
    \ar[dr, bend left =15,"{\sim}"]
    \ar[d, "{\sim}"'] &
    \\
    K\otimes (X',B) \ar[r]
    &
    P
    \ar[r, "{i\,\square\, j}"]
    &
    L\otimes (X',B) \,,
 \end{tikzcd}
 \] 
 with $P$ the pushout, where all morphisms are cofibrations and with acyclic cofibrations as marked.
 By the $2$-out-of-$3$ property the pushout product $i\, \square\, j$ is an acyclic cofibration.
This completes the proof of the theorem. 
\end{proof}

\begin{proof}[Proof of Claim.]
The fibrant objects for the $S_\mathbb{N}$-local model structure are precisely those projectively fibrant pairs $(X, A)$ that are moreover $S_\mathbb{N}$-local.
A pair $(X,A)$ is projectively fibrant iff $X$ is a Kan complex and $A$ is a levelwise fibrant sequential $X$-spectrum, and $S_\mathbb{N}$-locality is the condition that the maps of homotopy function complexes
\begin{align*}
\mathrm{map}_{\mathrm{Sp}^\mathbb{N}_{\mathrm{sSet}}}\left(\Sigma^{\infty-k}_{\Delta^n}(\Delta^n_{+\Delta^n}), (X,A)\right)
&\longrightarrow 
\mathrm{map}_{\mathrm{Sp}^\mathbb{N}_{\mathrm{sSet}}}\left(\Sigma^{\infty-(k+1)}_{\Delta^n}(\Sigma_{\Delta^n} \Delta^n_{+\Delta^n}), (X,A)\right)
\\
\mathrm{map}_{\mathrm{Sp}^\mathbb{N}_{\mathrm{sSet}}}\left(\Sigma^{\infty-k}_{\Delta^n}(\partial\Delta^n_{+\Delta^n}), (X,A)\right)
&\longrightarrow 
\mathrm{map}_{\mathrm{Sp}^\mathbb{N}_{\mathrm{sSet}}}\left(\Sigma^{\infty-(k+1)}_{\Delta^n}(\Sigma_{\Delta^n} \partial\Delta^n_{+\Delta^n}), (X,A)\right)
\end{align*}
are weak equivalences for all $n,k\geq 0$.
By adjointness, this amounts to the condition that the maps of homotopy function complexes induced by the adjoint structure maps $\sigma^\vee_n \colon A_k \to \Omega_{X} A_{k+1}$
\begin{equation}
\label{eqn:GlobSeqStableMpSpCondition}
\mathrm{map}_{\mathrm{Fun}(\Delta^1, \mathrm{sSet})}
\left(
\!
\raisebox{4pt}{
\begin{tikzcd}[sep =small]
  K(n)
  \ar[d]
  \\
  \Delta^n 
\end{tikzcd}}
,
\raisebox{4pt}{
\begin{tikzcd}[sep=small]
  A_k  
  \ar[d]
  \\
  X 
\end{tikzcd}}\! 
\right)
\longrightarrow 
\mathrm{map}_{\mathrm{Fun}(\Delta^1, \mathrm{sSet})}
\left(
\raisebox{4pt}{
\begin{tikzcd}[sep=small]
  K(n) 
  \ar[d]
  \\
  \Delta^n 
\end{tikzcd}
}
,
\raisebox{4pt}{
\begin{tikzcd}[sep=small]
  \Omega_X A_{k+1}  
  \ar[d]
  \\
  X 
\end{tikzcd}
} 
\right)
\end{equation}
are equivalences for all $n,k\geq 0$, where $K(n) =\partial \Delta^n$ or $\Delta^n$.

For a cofibration $i\colon K\to L$ of simplicial sets and a fibration $f\colon E\to F$ between Kan complexes, the map of homotopy function complexes
$
 \mathrm{map}_{\mathrm{Fun}(\Delta^1, \mathrm{sSet})}(i,f)
\to
\mathrm{map}_\mathrm{sSet}(L, F)
$
induced by evaluating arrows at their codomains is a fibration.
To see this, observe that 
\[
\text{lifts of}\quad
\begin{tikzcd}
  \Lambda^m_l
  \ar[r, "{\tau}"]
  \ar[d] &
  \mathrm{map}_{\mathrm{Fun}(\Delta^1, \mathrm{sSet})}(i,f)
  \ar[d]
  \\
  \Delta^m 
  \ar[ur, dashed]
  \ar[r,"{\sigma}"']
  &
  \mathrm{map}_\mathrm{sSet}(L,F)
\end{tikzcd}
\longleftrightarrow
\text{ arrows $\theta$ in}
\quad
\begin{tikzcd}[sep=small]
  &
  \Delta^m \times K 
  \ar[rr, "{\theta}", dashed]
  \ar[dd] && 
  E
  \ar[dd, "{f}"]
  \\
  \Lambda^m_l \times K \ar[ur] \ar[dd]
  \ar[urrr, bend right=15, crossing over]
  &&&
  \\
  &
  \Delta^m \times L
  \ar[rr, "{\overline{\sigma}}"] &&
  F\,,
  \\
  \Lambda^m_l \times L 
  \ar[ur]\ar[urrr, bend right=15]
\end{tikzcd}
\]
where $\overline{\sigma}$ corresponds to the $m$-simplex $\Delta^m \to \mathrm{map}_\mathrm{sSet}(L,F)$ and the front face of the right-hand diagram corresponds to $\tau\colon \Lambda^m_l \to \mathrm{map}_{\mathrm{Fun}(\Delta, \mathrm{sSet})}(i,f)$.
That such an arrow $\theta$ exists is guaranteed by the lifting diagram
\[
\begin{tikzcd}
  \Lambda^m_l \times K\ar[rr] \ar[d, rightarrowtail, "\sim"'] && E\ar[d,"{f}", two heads]
  \\
  \Delta^m \times K
  \ar[r]
  \ar[urr, dashed,"{\exists \theta}"]
  &
  \Delta^m \times L
  \ar[r, "{\overline{\sigma}}"]
  &
  F\,,
\end{tikzcd}
\]
as $f$ is a fibration and the left-hand vertical arrow is an acyclic cofibration.

The map \eqref{eqn:GlobSeqStableMpSpCondition} is thus an equivalence precisely if the induced map on (homotopy) fibres over $\sigma\colon \ast \to \mathrm{map}_{\mathrm{sSet}}(\Delta^n , X)$ is an equivalence for all $n$-simplices $\sigma$ of $X$.
The induced map on fibres over $\sigma$ is isomorphic to
\[
\mathrm{map}_{\mathrm{sSet}_{\dslash X}}(\sigma_! K(n), \sigma^\vee_k)\colon
\mathrm{map}_{\mathrm{sSet}_{\dslash X}}(\sigma_! K(n), A_k)
\longrightarrow 
\mathrm{map}_{\mathrm{sSet}_{\dslash X}}(\sigma_! K(n), \Omega_X A_{k+1})\,,
\]
which, by Lemma \ref{lem:SeqSpecFib}, is a equivalence for all $K(n)$, $\sigma\colon \Delta^n \to X$ and $n,k\geq 0$ precisely if $A$ is a fibrant $\Omega_X$-spectrum---the $\sigma_! K(n)$ range over the sets of domains and codomains of the set $\mathcal{I}^\mathrm{Kan}_X$ of generating cofibrations of $\mathrm{sSet}_{\dslash X}$).

We conclude that the $S_\mathbb{N}$-local fibrant objects are precisely fibrant loop spectra parametrised over Kan complexes.
In particular, the fibrant objects of the $S_\mathbb{N}$-local model structure coincide with those of the global stable model structure.
For objects $(X,A)$ and $(X',B)$, the homotopy function complex of maps  $(X,A)\to (X',B)$ can be presented both in the global and in the $S_\mathbb{N}$-local model structures by taking a fibrant replacement $(X',B)\to \mathcal{R}(X',B)$ in the global model structure and a Reedy cosimplicial coframe $\mathcal{Q}^\bullet(X,A)$ of $(X,A)$ with respect to the global \emph{projective} model structure:
\[
\underbrace{\mathrm{map}_{\mathrm{Sp}^\mathbb{N}_\mathrm{sSet}}\big((X,A), (X',B)}_{\text{global or $S_\mathbb{N}$-local}}\big)
\sim_{\mathrm{w.e.}} \left([n]\mapsto\mathrm{Sp}^\mathbb{N}_{\mathrm{sSet}}\big(\mathcal{Q}^n(X,A), \mathcal{R}(X',B)\big)\right)\,.
\]
The upshot is that the homotopy function complexes are the same in both model structures, and since a morphism $f$ in a model category $M$ is a weak equivalence precisely if $\mathrm{map}_M(f, T)$ is a weak equivalence for all fibrant objects $T$, we conclude that the global and $S_\mathbb{N}$-local model structures have the same classes of weak equivalences.
In any model category, any two out of the three classes of cofibrations, fibrations and weak equivalences is sufficient to determine the third.
\end{proof}

\begin{remark}[Basewise simplicial homotopy]
\label{rem:SeqSpecBasewiseEnrich}
$\mathrm{Sp}^\mathbb{N}_\mathrm{sSet}$ carries another \lq\lq basewise'' $\mathrm{sSet}$-enrichment, obtained by prolonging the tensoring of Corollary \ref{cor:GlobRetSpBaseEnrich} to sequential spectra.
The tensor bifunctor $\mathrm{sSet}\times \mathrm{Sp}^\mathbb{N}_\mathrm{sSet}\to  \mathrm{Sp}^\mathbb{N}_\mathrm{sSet}$ sends
$
\big(K, (X,A)\big)\mapsto
\big(K\times X, p_2^\ast A\big)$,
with $p_2\colon K\times X\to X$ the projection.
To show that this is a Quillen bifunctor, we use that for $p$ a projection to a factor in a product, the associated pullback functor $p^\ast$ on sequential parametrised spectra preserves cofibrations and stable weak equivalences (Theorem \ref{thm:SeqSpecStructureTheorem}).
\end{remark}

\subsubsection{Symmetric stabilisation}
\label{SS:SymStabGlob}
As we have just seen for sequential parametrised spectra, there is also an algebraic construction of the global category of parametrised symmetric spectra.
Together with allowing us to deduce that the global model structure is combinatorial and simplicial, this construction allows us to simply define a symmetric monoidal external smash product on $\mathrm{Sp}^\Sigma_\mathrm{sSet}$.
With respect to the external smash product, $\mathrm{Sp}^\Sigma$ is a symmetric monoidal model category.
\begin{construction}
Let $\mathrm{SymSeq}$ be the pullback category
\[
\begin{tikzcd}
  \mathrm{SymSeq}
  \ar[r]\ar[d]
  &
  \mathrm{Fun}(\mathbf{\Sigma}, \mathrm{sSet}_{\dslash \mathrm{sSet}})
  \ar[d,"{p}"]
  \\
  c(\mathrm{sSet})\,
  \ar[r, hookrightarrow]
  &
  \mathrm{Fun}(\mathbf{\Sigma},\mathrm{sSet})\,,
\end{tikzcd}
\]
with $c(\mathrm{sSet})$ the full subcategory on constant symmetric sequences.
That is, objects of $\mathrm{SymSeq}$ are symmetric sequences of retractive spaces over a fixed (but arbitrary) base space.
\end{construction}
\begin{lemma}
The Day convolution product on $\mathrm{Fun}(\mathbf{\Sigma},\mathrm{sSet}_{\dslash \mathrm{sSet}})$ restricts to a symmetric monoidal structure on $\mathrm{SymSeq}$.
\end{lemma}
\begin{proof}
The monoidal unit for Day convolution is the symmetric sequence that sends $0\mapsto S^0$ and all other terms to $0_\ast$---this is an object of $\mathrm{SymSeq}$.
On the other hand, for $(X, A), (X',B)\in \mathrm{SymSeq}$ we have
\[
(X, A)\otimes^\mathrm{Day} (X',B) \colon 
n\longmapsto
\int^{p,q\in \mathbf{\Sigma}}
(\ast,\mathbf{\Sigma}(p+q, n)_+ )\esmash (X, A(p))\esmash (X',B(q))\,,
\]
Since the projection functor $p\colon \mathrm{sSet}_{\dslash \mathrm{sSet}}\to \mathrm{sSet}$ is a two-sided adjoint
\[
p\Big((X, A)\otimes^\mathrm{Day} (X',B) \Big)(n)
=
\int^{p,q\in \mathbf{\Sigma}} \ast \times X\times X' \cong X\times X'\,,
\]
so that the sequence of base spaces is constant.
\end{proof}

The symmetric sequence $\mathbb{S} \colon n\mapsto (\ast, S^n)$ is a commutative monoid with respect to the Day convolution product on $\mathrm{SymSeq}$.
Writing $\Sigma\mathrm{-Sp}$ for the category of $\mathbb{S}$-modules in $\mathrm{SymSeq}$, by unwinding the definitions we have the following
\begin{lemma}
The categories $\Sigma\mathrm{-Sp}$ and $\mathrm{Sp}^\Sigma_{\mathrm{sSet}}$ are canonically equivalent.
\end{lemma}
\begin{corollary}
\label{cor:GlobSymSpecLocPres}
$\mathrm{Sp}^\Sigma_{\mathrm{sSet}}$ is locally presentable.
\end{corollary}
\begin{proof}
$\mathrm{SymSeq}$ is locally presentable by the same argument as for $\mathrm{RetSeq}$ (Corollary \ref{cor:GlobSeqSpecLocPres}).
The Day convolution product on $\mathrm{SymSeq}$ preserves colimits separately in each variable, so that $\mathbb{S}\otimes^\mathrm{Day}(-)$ is an accessible monad.
The category of algebras over this monad is $\Sigma\mathrm{-Sp}$, which is therefore accessible \cite[2.78]{adamek_locally_1994}.
Since $\Sigma\mathrm{-Sp}$ has small colimits, the result is proven.
\end{proof}
Since $\mathbb{S}$ is a commutative monoid, the category $\mathrm{Sp}^\Sigma_\mathrm{sSet}$ inherits a symmetric monoidal product, the \emph{(symmetric) external smash product}, via
\[
(X,A) \esmash (X',B) :=
\mathrm{colim}\left(\!
\begin{tikzcd}
  (X,A)\otimes^\mathrm{Day}
  \mathbb{S}
  \otimes^\mathrm{Day}
  (X',B)
  \ar[r, shift left=0.8ex]
  \ar[r, shift left=-0.8ex]
  &
  (X,A)\otimes^\mathrm{Day} (X', B)
\end{tikzcd}\!
\right)\,.
\]
Since the Day convolution product on $\mathrm{SymSeq}$ preserves colimits in each variable, we have
\begin{lemma}
The external smash product on $\mathrm{Sp}^\Sigma_\mathrm{sSet}$ preserves colimits separately in each variable.
\end{lemma}
\begin{corollary}
$(\mathrm{Sp}^\Sigma_\mathrm{sSet},\vartriangle)$ is closed symmetric monoidal.
The projection functor $p\colon \mathrm{Sp}^\Sigma_\mathrm{sSet}\to \mathrm{sSet}$ is strongly closed monoidal.
\end{corollary}
\begin{proof}
For any $(X,A)\in \mathrm{Sp}^\Sigma_\mathrm{sSet}$, the colimit-preserving functor $(X,A)\esmash(-)$ has a right adjoint by the adjoint functor theorem.
We have already seen that the projection functor on $\mathrm{SymSeq}$ is strongly symmetric monoidal, so that $\vartriangle$ covers the Cartesian product.
That $p\colon \mathrm{Sp}^\Sigma_\mathrm{sSet}\to \mathrm{sSet}$ is strongly closed monoidal follows as in Lemma \ref{lem:RetSpProjClosedMonoidal}.
\end{proof}

In Remark \ref{rem:FibSymSmash} we observed that the shifted fibrewise suspension functors are compatible with the fibrewise smash product of symmetric $X$-spectra.
There is also a global version of this result, using the fibrewise suspension functors of Corollary \ref{cor:GlobalSigmaOmega}:
\begin{lemma}
\label{lem:SmashVSSuspend}
For $(X,Y), (X', Z)\in\mathrm{sSet}_{\dslash \mathrm{sSet}}$ and $k,l\geq 0$ there is a natural isomorphism
\[
\mathbf{\Sigma}^{\infty-k}_- (X, Y)\esmash \mathbf{\Sigma}^{\infty-l }_- (X',Z) \cong \mathbf{\Sigma}^{\infty- (k+l)}_- \big((X,Y)\esmash (X',Z)\big)
\]
of parametrised symmetric spectra covering the identity on $X\times X'$.
\end{lemma}
\begin{proof}
The fibrewise suspension spectrum $\mathbf{\Sigma}^{\infty-k}_- (X,Y) \cong \mathbf{\Sigma}^{\infty-k}_X Y$ is as in Remark \ref{rem:FreeSymSpec}.
For $n < k+l$ we thus have isomorphisms of retractive spaces
\[
\left[\mathbf{\Sigma}^{\infty-k}_- (X, Y)\esmash \mathbf{\Sigma}^{\infty-l }_- (X',Z)\right] (n) \cong 0_{X\times X' } \cong\left[ \mathbf{\Sigma}^{\infty- (k+l)}_- \big((X,Y)\esmash (X',Z)\big)\right](n)
\]
over $X\times X'$, since we always have at least one factor of the form $(X,0_X)$ or $(X', 0_{X'})$ in the coend expression defining the left-hand side.
For $n \geq k$, we have
\begin{align*}
\mathbf{\Sigma}^{\infty-k}_- (X, Y)(n) &\cong \left(X, \big((\Sigma_n)_+\owedge_X S^{n-k} \owedge_X Y\big)\big/ \Sigma_{n-k} \right)
\\
&\cong 
 \big(\ast, ((\Sigma_n)_+\wedge S^{n-k})/ \Sigma_{n-k}\big)\esmash (X,Y)
 \\
 &
 \cong (\mathbf{\Sigma}^{\infty-k} S^0)(n) \esmash (X,Y)\,,
\end{align*}
so that for $n\geq k+l$  the $n$-th term of $\mathbf{\Sigma}^{\infty-k}_- (X, Y)\esmash \mathbf{\Sigma}^{\infty-l }_- (X',Z)$ is the retractive space over $X\times X'$ obtained by externally smashing $(X,Y)\esmash (X',Z)$ with the pointed simplicial set
\[
\left[\mathbf{\Sigma}^{\infty-k} S^0 
\wedge \mathbf{\Sigma}^{\infty-l} S^0
\right](n) \overset{\text{\eqref{eqn:SmashVSSuspend}}}{\cong} (\mathbf{\Sigma}^{\infty-(k+l)} S^0)(n) \cong (\Sigma_{n})_+ \wedge S^{n-(k+l)}\big/ \Sigma_{n-(k+l)}\,.
\]
But this is precisely the $n$-th term of $\mathbf{\Sigma}^{\infty- (k+l)}_- \big((X,Y)\esmash (X',Z)\big)$.
\end{proof}
\begin{corollary}
\label{cor:SuspenVSSmashII}
The functor $\mathbf{\Sigma}^{\infty}_- \colon \mathrm{sSet}_{\dslash \mathrm{sSet}}\to \mathrm{Sp}^\Sigma_\mathrm{sSet}$ is strongly symmetric monoidal.
\end{corollary}

We now come to the main result of this section.
\begin{theorem}
\label{thm:ExternalSmash}
$(\mathrm{Sp}^\Sigma_\mathrm{sSet}, \vartriangle)$ is a left proper combinatorial symmetric monoidal model category.
\end{theorem}
\begin{proof}
The model categorical considerations (combinatoriality and left properness) are proven along the same lines as Theorem \ref{thm:SeqSpecGlobalComb}, by showing that
\begin{enumerate}[label=(\roman*)]
  \item $\mathrm{Sp}^\Sigma_\mathrm{sSet}$ carries a global projective model structure with generating cofibrations and acyclic cofibrations given respectively by the sets
  \[
\mathcal{I}_{\Sigma\mathrm{-Sp}}:=
\bigcup_{k\geq 0}\mathbf{\Sigma}^{\infty-k}_-(\mathcal{I}^\mathrm{Kan}_{\mathrm{sSet}})
\;\;
\mbox{ and }
\;\;
\mathcal{J}'_{\Sigma\mathrm{-Sp}}:=
\bigcup_{k\geq 0}\mathbf{\Sigma}^{\infty-k}_-(\mathcal{J}^\mathrm{Kan}_{\mathrm{sSet}})\,.
\]
  \item The global model structure is the left Bousfield localisation at the set of morphisms
  \[
  S_{\Sigma} :=\left\{\xi_{k,\Delta^n } \big((\mathrm{id}|_{\partial\Delta^n})_{+\Delta^n}\big), \xi_{k,\Delta^n } \big(\mathrm{id}_{+\Delta^n}\big)\right\}_{n,k\geq 0}\,.
  \] 
\end{enumerate} 
Corollary \ref{cor:SuspenVSSmashII} and Theorem \ref{thm:RetSpExtSmash} together imply that the sets of pushout-products $\mathcal{I}_{\Sigma\mathrm{-Sp}}\,\square\, \mathcal{I}_{\Sigma\mathrm{-Sp}}$ and $\mathcal{I}_{\Sigma\mathrm{-Sp}}\, \square\, \mathcal{J}'_{\Sigma\mathrm{-Sp}}$ consist of cofibrations and acyclic cofibrations respectively, form which it follows so the external smash product is a Quillen bifunctor for the global projective model structure. 

To show that the external smash product determines a Quillen bifunctor on the localised model structure, we use the fact that for any $(X,Y)\in \mathrm{sSet}_{\dslash \mathrm{sSet}}$ and $n\geq 0$ the comparison map
\[
\xi_{n,X}(X,Y)\colon \mathbf{\Sigma}^{\infty-(n+1)}_- (X, \Sigma_X Y)\longrightarrow \mathbf{\Sigma}^{\infty-n}_- (X, Y)
\]
is a weak equivalence for the global model structure (Lemma \ref{lem:XiMapsAreStable}).
By Lemma \ref{lem:SmashVSSuspend} there is an isomorphism
\[
\Sigma^{\infty-k}_- (X,Y) \esmash \xi_{l,X'}(X',B)\cong \xi_{k+l,X\times X'}(X,A)\esmash (X',B)\,.
\]
Applied to the set $S_{\Sigma}$, this  implies that the functors $\Sigma^{\infty-k}_- (X, Y)\esmash(-)$ are left Quillen for the global model structure for all $(X,Y)\in \mathrm{sSet}_{\dslash \mathrm{sSet}}$ and $k\geq 0$. 
If $s\colon (X,A)\to (X',B)$ is an acyclic cofibration for the global model structure and $i\colon (K,Y)\to (L,Z)$ is a cofibration in $\mathrm{sSet}_{\dslash \mathrm{sSet}}$, the diagram of cofibrations 
\[
 \begin{tikzcd}
    \mathbf{\Sigma}^{\infty-k}_- (K, Y) \esmash (X,A)
    \ar[r]\ar[d, "{\sim}"'] 
    & 
    \mathbf{\Sigma}^{\infty-k}_- (L, Z) \esmash (X,A)
    \ar[dr, "{\sim}", bend left=15]
    \ar[d, "{\sim}"'] &
    \\
    \mathbf{\Sigma}^{\infty-k}_- (K, Y) \esmash (X',B) \ar[r]
    &
    P
    \ar[r, "{i\,\square\, s}"]
    &
    \mathbf{\Sigma}^{\infty-k}_- (L, Z) \esmash (X',B) \,,
 \end{tikzcd}
 \]
has weak equivalences (for the global model structure on $\mathrm{Sp}^\Sigma_\mathrm{sSet}$) as marked, so that $i\,\square\, s$ is an acylic cofibration.
Invoking cofibrant generation once more completes the proof.
\end{proof}

\begin{corollary}
\label{cor:ExtSmashMonoidalFun}
Each of the functors
\[
\begin{tikzcd}
  \mathrm{sSet}_{\dslash \mathrm{sSet}}
  \ar[r, "{\mathbf{\Sigma}^{\infty}_-}"]
  &
  \mathrm{Sp}^\Sigma_\mathrm{sSet}
  \ar[r, "{p}"]
  &
  \mathrm{sSet}
\end{tikzcd}
\]
is strongly monoidal left Quillen.
\end{corollary}
\begin{corollary}
\label{cor:GlobSymSpecRetStEnrich}
$\mathrm{Sp}^\Sigma_\mathrm{sSet}$ is a $\mathrm{sSet}_{\dslash \mathrm{sSet}}$-model category.
\end{corollary}

\begin{corollary}
\label{cor:GlobSymSpecSpEnrich}
$\mathrm{Sp}^\Sigma_\mathrm{sSet}$ is a $\mathrm{Sp}^\Sigma$-model category.
\end{corollary}
\begin{remark}[Enrichments of $\mathrm{Sp}^\Sigma_\mathrm{sSet}$]
\label{rem:EnrichSymSpec}
The $\mathrm{Sp}^\Sigma$-enriched hom-spaces of Corollary \ref{cor:GlobSymSpecSpEnrich} compute global sections of fibrewise mapping spectra.
For parametrised symmetric spectra $(X,A)$ and $(X', B)$ (cofibrant and fibrant respectively) we write $\underline{F}\{(X,A), (X',B)\}$ and $F_\vartriangle \{(X,A), (X',B)\}$ for the $\mathrm{Sp}^\Sigma$-enriched hom-space and the internal hom respectively, where $F_\vartriangle \{(X,A), (X',B)\}$ is parametrised by $[X,X']$.
Then for any (cofibrant) symmetric spectrum $E$:
\begin{align*}
\mathrm{Sp}^\Sigma\Big( E, \underline{F}\{ (X,A), (X',B)\}\Big) 
&\cong 
\mathrm{Sp}^\Sigma_\mathrm{sSet}\Big((\ast, E)\esmash(X,A), (X',B)\Big)
\\
&
\cong 
\mathrm{Sp}^{\Sigma}_\mathrm{sSet}\Big(\imath E , F_\vartriangle\left\{(X,A), (X',B)\right\}\Big)
\\
&\cong
\mathrm{Sp}^\Sigma\Big( E, [X,X']_\ast F_\vartriangle \{(X,A), (X',B)\}
 \Big) \qquad\qquad\text{(Remark \ref{rem:RAtoInisGlobSect})}
\end{align*}
so that $\underline{F}\{(X,A), (X',B)\}\cong [X,X']_\ast F_\vartriangle \{(X,A), (X',B)\}$ is the global section spectrum. 
Similarly, the $\mathrm{sSet}_{\dslash \mathrm{sSet}}$-enriched hom-spaces of Corollary \ref{cor:GlobSymSpecRetStEnrich} compute the fibrewise infinite loop space o $F_\vartriangle \{(X,A), (X',B)\}$ as a retractive space over $[X,X']$.
Via the two symmetric monoidal Quillen functors $\mathrm{sSet}\to \mathrm{sSet}_{\dslash \mathrm{sSet}}$
\[
K \longmapsto (\ast, K_+)   \qquad \mbox{and}
\qquad
K\longmapsto (K, K^\ast S^0)\qquad\qquad \text{(Corollary \ref{cor:GlobRetSpBaseEnrich})}
\]
we obtain \lq\lq fibrewise'' and \lq\lq basewise'' simplicial enrichments of $\mathrm{Sp}^\Sigma_\mathrm{sSet}$ respectively.
\end{remark}

The final result of this section is a consistency check, which shows that the external smash product of parametrised symmetric spectra has the correct behaviour on homotopy fibre spectra.
In the following we write $F_x A := \mathbf{R}x^\ast A$ for the homotopy fibre spectrum of $A\in \mathrm{Sp}^\Sigma_X$ at $x\colon \ast \to X$ and $(X, A)\esmash(Y,B) = (X\times Y, A\esmash B)$ for the components of the external smash product.
\begin{lemma}
\label{lem:ESmash=FSmash}
For cofibrant parametrised symmetric spectra $(X,A), (Y, B)\in \mathrm{Sp}^\Sigma_\mathrm{sSet}$, there is a stable equivalence $F_x A\wedge F_y B \cong F_{(x,y)} (A\esmash B)$ for all $(x,y)\in X\times Y$.
\end{lemma}
\begin{proof}
By the cofibrancy assumption, the external smash product $(X, A)\esmash (Y, B) = (X\times Y, A\esmash B)$ is itself cofibrant.
In particular, this implies that $A\esmash B\in \mathrm{Sp}^\Sigma_{X\times Y}$ is cofibrant for the symmetric $(X\times Y)$-stable model structure.

Factor the maps $x\colon \ast \to X$ and $y\colon \ast \to Y$ into acyclic cofibrations followed by fibrations
\[
\begin{tikzcd}
  \ast
  \ar[r, "{\sim}", rightarrowtail]
  &
  P_x X
  \ar[r, two heads, "{p_X}"]
  &
  X
\end{tikzcd}
  \qquad
  \mbox{ and }
  \qquad
\begin{tikzcd}
  \ast
  \ar[r, "{\sim}", rightarrowtail]
  &
  P_y Y
  \ar[r, "{p_Y}", two heads]
  &
  Y
\end{tikzcd}
\]
so that the fibration $p_x\times p_y\colon P_x X\times P_y Y \to X\times Y$ is a model for the path fibration of $X\times Y$ at $(x,y)$.
Arguing as in Lemma \ref{lem:SeqSpecisStable}, see \eqref{eqn:FibreSpecZigZag}, the homotopy fibre spectrum $F_{(x,y)} (A\esmash B)$ is seen to be stably equivalent to the cofibrant symmetric spectrum $(P_xX\times P_yY)_! (p_X\times p_Y)^\ast (A\esmash B)$.
Pullback-stability of colimits in $\mathrm{sSet}$ applied to the definition of the external smash product $\vartriangle$ yields an isomorphism of symmetric $P_xX\times P_y Y$-spectra
\[
(p_X\times p_Y)^\ast (A\esmash B) \cong p_X^\ast A\esmash p_Y^\ast B\,.
\]
Commutation of colimits amongst themselves implies an isomorphism of symmetric $P_yY$-spectra
\[
(P_xX\times \mathrm{id}_{P_yY})_! \big(p_X^\ast A\esmash p_Y^\ast B\big)\cong 
(P_x X_! p_X^\ast A ) \esmash p_Y^\ast B\,,
\]
where the right-hand side coincides with the $\mathrm{Sp}^\Sigma$-tensoring of $p_Y^\ast B$ by the (unparametrised) symmetric spectrum $P_x X_! p_X^\ast A$.
Pushing forward along the terminal map $P_y Y\to \ast$ preserves $\mathrm{Sp}^\Sigma$-tensors, so that 
\[
(P_xX\times P_yY)_! \big(p_X^\ast A\esmash p_Y^\ast B\big)\cong P_y Y_! (P_xX\times \mathrm{id}_{P_yY})_! \big(p_X^\ast A\esmash p_Y^\ast B\big)
\cong (P_x X_! p_X^\ast A ) \wedge (P_y Y_! p_Y^\ast B )\,.
\]
The terms in the smash product on the right-hand side are stably equivalent to $F_x A$ and $F_y B$, which completes the proof.
\end{proof}

\section{Combinatorial models for tangent $\infty$-categories}
\label{S:CombTangent}
In the final part of this article, we apply our results to obtain combinatorial models for tangent $\infty$-categories.
Our main result is an extension of Simpson's theorem, characterising the tangent $\infty$-categories of presentable $\infty$-categories as accessible localisations of $\infty$-categories of presheaves of parametrised spectra.
As an example of the utility of this result, in Section \ref{SS:TwistDiffCoh} we specialise to the $\infty$-topos $\mathbf{H}$ of smooth $\infty$-stacks and show that twisted differential cohomology theories are classified by sheaves of parametrised spectra.

We begin with some general remarks on tangent categories, which the reader may find more familiar than their $\infty$-categorical siblings.
A common theme in mathematics is the dichotomy between linearity and non-linearity.
Broadly speaking, \lq\lq linear'' structures such as vector spaces and chain complexes are often much easier to study and analyse than \lq\lq non-linear'' structures such as groups  and spaces.  
Non-linear objects are often effectively studied by a process of linearisation; for example studying groups in terms of their categories of modules, or studying spaces by their (co)homological invariants.
One way to make this more concrete is the following: for an object $c$ of a locally presentable category $C$ (of \lq\lq non-linear structures''), a \emph{Beck module} over $c$ is an abelian group object of the slice $C_{/c}$.
For each $c\in C$ there is a free-forgetful adjunction
\[
\begin{tikzcd}
C_{/c}
\ar[rr, shift left=1.1ex, "\mathrm{free}"]
\ar[rr, leftarrow, shift left=-1.1ex, "\bot", ""']
&&
\mathrm{Ab}(C_{/c})\,,
\end{tikzcd}
\]
with the left adjoint is interpreted as \lq\lq linearisation''. For example, if $C=\mathrm{CRing}$ is the category of commutative rings, then $\mathrm{Ab}(\mathrm{CRing}_{/R})$ is equivalent to the category of (left) $R$-modules and $\mathrm{CRing}_{/R}\to \mathrm{Ab}(\mathrm{CRing}_{/R})$ sends $S\to R$ to the $R$-module of relative K\"{a}hler differentials $\Omega_{R/S}$.
The tangent bundle of $C$ is the categorical fibration $TC\to C$ whose fibre at $c$ is $\mathrm{Ab}(C_{/c})$.

Working with $\infty$-categories, it is more natural to replace abelianisation by \emph{stabilisation}.
More precisely, for an $\infty$-category $\mathcal{C}$ the \emph{tangent $\infty$-bundle} is, if it exists, the functor $p\colon T\mathcal{C}\to \mathcal{C}$ obtained by applying Lurie's unstraightening construction \cite[Chapter 2]{lurie_higher_2009} to
\begin{align*}
T_{-}\mathcal{C}\colon \mathcal{C}
&
\longrightarrow \widehat{\mathrm{Cat}}_\infty\\
c&\longmapsto \mathrm{Sp}(\mathcal{C}_{\dslash c})\cong \mathrm{Sp}(\mathcal{C}_{/c})\,,
\end{align*}
each object $c$ being sent to the stable $\infty$-category $\mathrm{Sp}(\mathcal{C}_{\dslash c})$ of \emph{$c$-parametrised spectra} (see \cite[Section 7.3.1]{lurie_higher_2017} or, for $\mathcal{C}$ an $\infty$-topos, \cite[Section 5]{ando_parametrized_2018}).
The domain of the tangent $\infty$-bundle is the \emph{tangent $\infty$-category} $T\mathcal{C}$.
The tangent $\infty$-category of a presentable $\infty$-category always exists and is itself presentable.
In this case, there is in addition to the projection $p\colon T\mathcal{C}\to \mathcal{C}$ a functor $q\colon T\mathcal{C}\to \mathcal{C}$ corresponding to 
\[
c\longmapsto \big(\!
\begin{tikzcd}
  \mathrm{Sp}(\mathcal{C}_{\dslash c})
  \ar[r, "{\Omega^\infty}"]
  &
  \mathcal{C}_{\dslash c}
  \ar[r]
  &
  \mathcal{C}
\end{tikzcd}
\! 
\big)
\]
under (un)straightening.
This functor has a left adjoint 
$\mathbb{L}\colon \mathcal{C}\to T\mathcal{C}$
sending each object of $\mathcal{C}$ to its \emph{cotangent complex}.
 The cotangent complex $\mathbb{L}$ is a section of the projection functor, that is $p\circ \mathbb{L}\cong \mathrm{id}_\mathcal{C}$.

\begin{remark}
The tangent $\infty$-category ought to be thought of as parametrising first-order infinitesimal thickenings of objects of $\mathcal{C}$.
For example, if $\mathcal{C}\cong \mathrm{Alg}_\mathcal{P}(\mathcal{D})$ is the $\infty$-category of algebras over an $\infty$-operad $\mathcal{P}$ in a stable presentable $\infty$-category $\mathcal{D}$, there are identifications of the fibre $T_A \mathcal{C}$  with the $\infty$-category of $A$-modules and of $q|_{T_A\mathcal{C}}\colon T_A \mathcal{C}\to  \mathcal{C}$ with the functor that sends an $A$-module $M$ to the square-zero extension of $A$ by $M$ \cite{lurie_higher_2017}.
\end{remark}

\begin{remark}[Goodwillie calculus and deformations]
\label{rem:Goo}
For a fixed presentable $\infty$-category $\mathcal{C}$, the tangent $\infty$-category is the first stage in a tower of $\infty$-categories
\begin{equation}
\label{eqn:GoodwillieTower}
  \mathcal{C}^{\Delta^1}
\longrightarrow
  \dotsb
\longrightarrow
  J^n\mathcal{C}
\longrightarrow
  \dotsb
\longrightarrow
  J^2 \mathcal{C}
\longrightarrow
  T \mathcal{C}
\longrightarrow
  \mathcal{C}
\end{equation}
factoring the codomain projection functor $\mathrm{ev}_{[1]} \colon \mathcal{C}^{\Delta^1}\to \mathcal{C}$.
In this tower, the \emph{$n$-th jet $\infty$-bundle} $p^n\colon J^{n} \mathcal{C}\to \mathcal{C}$ is the result of unstraightening the functor
$c\mapsto \mathcal{P}_n (\mathcal{C}_{\dslash c})$ that
sends $c\in \mathcal{C}$ to the $n$-th excisive approximation of the $\infty$-category $\mathcal{C}_{\dslash c}$ (the general theory of excisive approximations of $\infty$-categories was worked out by Heuts in \cite{heuts_goodwillie_2015}).
The first excisive approximation corresponds to passing to the stabilisation, and the various excisive approximations fit together into a tower of left adjoints
\[
\mathcal{C}_{\dslash c}
\longrightarrow
\dotsb
\longrightarrow
\mathcal{P}_n(\mathcal{C}_{\dslash c})
\longrightarrow
\dotsb
\longrightarrow
\mathcal{P}_2(\mathcal{C}_{\dslash c})
\longrightarrow
\mathcal{P}_1(\mathcal{C}_{\dslash c})\cong \mathrm{Sp}(\mathcal{C}_{\dslash c})
\]
interpolating between $\mathcal{C}_{\dslash c}$ and its stabilisation.
The projection functors $p^n\colon J^n \mathcal{C}\to \mathcal{C}$ are compatible in the obvious sense and for each $n$ there is a functor $q^n\colon J^n\mathcal{C}\to \mathcal{C}$ generalising the functor $q\colon T\mathcal{C}\to\mathcal{C}$.
Each of the $q^n$ has a left adjoint $\mathbb{L}^n$ such that $p^n\circ \mathbb{L}^n\cong\mathrm{id}_{\mathcal{C}}$, and there is a diagram of left adjoints commuting up to natural isomorphism:
\[
\begin{tikzcd}
  \mathcal{C}
  \ar[r]
  \ar[rr, bend left=15, "{\mathbb{L}^n}"]
  \ar[rrrr, bend right=15,"{\mathbb{L}^2}"']
  \ar[rrrrr, bend left= 28 ,"{\mathbb{L}}"]
  &
  \dotsb
  \ar[r]
  &
  J^n\mathcal{C}
  \ar[r]
  &
  \dotsb
  \ar[r]
  &
  J^2 \mathcal{C}
  \ar[r]
  &
  T\mathcal{C}\,.
\end{tikzcd}
\]
Passing to the limit of the tower \eqref{eqn:GoodwillieTower}, we get a diagram of left adjoint functors
\[
\begin{tikzcd}
  \mathcal{C}
  \ar[r, "{\mathbb{L}^{\infty}}"]
  \ar[rr, "{\mathbb{L}}"', bend right= 20]
  &
  J^{\infty}\mathcal{C} := \underset{\longleftarrow}{\mathrm{lim}}\, J^n \mathcal{C}
  \ar[r]
  &
  T\mathcal{C}\,,
\end{tikzcd}
\]
in which the \emph{jet $\infty$-bundle} $p^{\infty}\colon J^{\infty}\mathcal{C}\to \mathcal{C}$ is interpreted as parametrising infinitesimal thickenings of objects of $\mathcal{C}$ to all orders.
By results of Lurie \cite[Section 6.3]{lurie_higher_2017}, objects in the fibre of $p^{\infty}$ over $c\in \mathcal{C}$ are coalgebras over the Goodwillie coderivatives of the identity $\partial^\ast(\mathrm{id}_{\mathcal{C}_{\dslash c}})$.
In particular, $\mathbb{L}^{\infty} c$ is a coalgebraic lift of the cotangent complex $\mathbb{L}c$ through the forgetful functor $J^{\infty}\mathcal{C}\to T\mathcal{C}$.

Armed with a suitable notion of Koszul  duality for stable $\infty$-operads (such as has been worked out by Ching \cite{ching_bar_2012} for operads and cooperads in spectra), the coalgebraic structure of $\mathbb{L}^\infty c$ ought to give rise to an algebra $\mathbb{T}^\infty c$  over the Goodwillie derivatives of the identity $\partial_\ast(\mathrm{id}_{\mathcal{C}_{\dslash c}})$.
Provided that this can be made precise, we expect $\mathbb{T}^\infty c$ to play a central role in the deformation theory of $c$ inside $\mathcal{C}$, analogous to the Lie algebras controlling derived deformations problems in derived algebraic geometry over a field of characteristic zero (we recommend Lurie's ICM address \cite{lurie_moduli_2010} for a compelling overview of this perspective on formal moduli problems).
This analogy is borne out for the $\infty$-category of 1-connected rational spaces, where for a 1-connected rational space $X$, $\mathbb{T}^\infty X$ is the rational Whitehead Lie algebra \cite{quillen_rational_1969, heuts_goodwillie_2015}.
We intend to undertake a detailed study of these ideas in future work.
\end{remark}

Our first result of this section is the observation that $\mathrm{Sp}^\mathbb{N}_\mathrm{sSet}$ is a presentation of the tangent $\infty$-category of spaces:
\begin{lemma}
\label{lem:PresentTS}
There are equivalences of $\infty$-categories $T\mathcal{S}\cong (\mathrm{Sp}^\mathbb{N}_\mathrm{sSet})^\infty \cong (\mathrm{Sp}^\Sigma_\mathrm{sSet})^\infty$.
\end{lemma}
\begin{proof}
This is essentially the statement of Theorem \ref{thm:FibStabisooStab}, using that the Grothendieck construction of model categories models Lurie's unstraightening construction for $\infty$-categories.
In more detail: given a model category $M$ and a  proper relative pseudofunctor $F\colon M\to \mathbf{Model}$, there is an associated functor of $\infty$-categories
\[
F^\infty\colon M^\infty\longrightarrow \widehat{\mathrm{Cat}}_\infty\,,
\]
with $\widehat{\mathrm{Cat}}_\infty$ the large $\infty$-category of $\infty$-categories (with respect to some choices of Grothendieck universes, which we do not specify but regard as fixed throughout). 
Lurie's unstraightening construction assigns a co-Cartesian fibration
\[
\int_{M^\infty} F^\infty\longrightarrow M^\infty\,,
\]
and there is an equivalence of $\infty$-categories
\[
\begin{tikzcd}[column sep=small]
  \displaystyle{\Big(\int_{M} F\Big){\vphantom{\big)}}^\infty}
  \ar[rr, "{\cong}"]
  \ar[dr]
  &&
  \displaystyle\int_{M^\infty} F^\infty
  \ar[dl]
  \\
  &M^\infty
\end{tikzcd}
\]
fibred over $M^\infty$ \cite[Proposition 3.1.2]{harpaz_grothendieck_2015}.
Applying this argument to the pseudofunctors pseudofunctors $\mathrm{Sp}^\mathbb{N}_-, \mathrm{Sp}^\Sigma_-\colon \mathrm{sSet}\to \mathbf{Model}$ in view of Theorem \ref{thm:FibStabisooStab} completes the proof.
\end{proof}

We are now in a position to prove the main result of this section:
\begin{theorem}
\label{thm:Tangent}
For any presentable $\infty$-category $\mathcal{C}$, there is a small $\infty$-category $\mathcal{K}$ such that the tangent $\infty$-category $T\mathcal{C}$ is an accessible localisation of the $\infty$-category $\mathrm{Fun}(\mathcal{K}^\mathrm{op},T\mathcal{S})$ of $T\mathcal{S}$-valued presheaves.
\end{theorem}

\begin{remark}
Let $M$ be a left proper combinatorial model category (so that $M^\infty$ is presentable) and let $(\mathbb{N}\times \mathbb{N})_\star$ be the category obtained by freely adjoining a zero object.
The main result of \cite{harpaz_tangent_2018} shows that the tangent $\infty$-category $TM^\infty$ is presented by a left Bousfield localisation of the Reedy model structure on functors $(\mathbb{N}\times \mathbb{N})_\star\to M$.
In \cite{harpaz_tangent_2016}, this construction of the tangent $\infty$-category is used prove an equivalence between parametrised spectra in the $\infty$-category of algebras over an $\infty$-operad and operadic modules.
\end{remark}

Fixing a presentable $\infty$-category $\mathcal{C}$, we first find a (small) simplicial category $K$ together with a set $S$ of morphisms of $\mathrm{Fun}(K^\mathrm{op}, \mathrm{sSet})$ such that $\mathcal{C}$ is presented by the left Bousfield localisation
\[
\begin{tikzcd}
L_S \mathrm{Fun}_\Delta(K^\mathrm{op}, \mathrm{sSet})
\ar[rr, leftarrow, shift left=1.1ex, ""]
\ar[rr, hookrightarrow, shift left=-1.1ex, "\bot", ""']
&&
\mathrm{Fun}_\Delta(K^\mathrm{op}, \mathrm{sSet})
\end{tikzcd}
\]
of the enriched category of simplicial presheaves equipped with the projective model structure \cite[Proposition A.3.7.6]{lurie_higher_2009}.
We may suppose without loss of generality that the domains and codomains of all morphisms in $S$ are cofibrant.
We prove Theorem \ref{thm:Tangent} by showing that the tangent $\infty$-bundle $p\colon T\mathcal{C}\to \mathcal{C}$ is presented by a simplicial Quillen functor between (left Bousfield localisations of) enriched functor categories
\[
L_{TS} \mathrm{Fun}_\Delta \big(K^\mathrm{op}, \mathrm{Sp}^\mathbb{N}_\mathrm{sSet}\big)
\longrightarrow
L_S \mathrm{Fun}_\Delta (K^\mathrm{op}, \mathrm{sSet})\,,
\]
with the prolonged set of morphisms $TS$ (Definition \ref{defn:ProlongLoc}). 
Throughout, $\mathrm{Sp}^\mathbb{N}_\mathrm{sSet}$ is regarded as a simplicial category with respect to the \lq\lq basewise'' simplicial enrichment (Remark \ref{rem:SeqSpecBasewiseEnrich}), which guarantees that the base projection adjunction
\begin{equation}
\label{eqn:BaseProjectPSheaves}
\begin{tikzcd}
 \mathrm{Fun}_\Delta \big(K^\mathrm{op}, \mathrm{Sp}^\mathbb{N}_\mathrm{sSet}\big)
\ar[rr, shift left=1.2ex, "p"]
\ar[rr, leftarrow, shift left=-1.1ex, "\bot/\top", "0_-"']
&&
\mathrm{Fun}_\Delta(K^\mathrm{op}, \mathrm{sSet})
\end{tikzcd}
\end{equation}
is simplicial. 
For $k\geq 0$ let $\Sigma^{\infty-k}_+$ be the composite of simplicial functors
\begin{equation}
\label{eqn:PSheafProlongation}
  \mathrm{Fun}_\Delta(K^\mathrm{op},\mathrm{sSet})
  \xrightarrow{(\mathtt{0})_{+(\mathtt{0})}}
  \mathrm{Fun}_\Delta(K^\mathrm{op},\mathrm{sSet}_{\dslash \mathrm{sSet}})
  \xrightarrow{\Sigma^{\infty-k}_-}
  \mathrm{Fun}_\Delta(K^\mathrm{op},\mathrm{Sp}^\mathbb{N}_\mathrm{sSet})\,,
\end{equation}
with
$(\mathtt{0})_{+(\mathtt{0})}\colon X \to (X, X_{+X})$ and $\Sigma^{\infty-k}_- \colon (X,Z) \mapsto (X, \Sigma^{\infty-k}_X Z)$, defined objectwise in $k\in K$.
There are simplicial adjunctions
\begin{equation}
\label{eqn:PSheavesSuspendSpec}
(\Sigma^{\infty-k}_+, \widetilde{\Omega}^{\infty-k}_+)\colon
\begin{tikzcd}
\mathrm{Fun}_\Delta(K^\mathrm{op}, \mathrm{sSet})
\ar[rr, shift left=1.1ex, ""]
\ar[rr, leftarrow, shift left=-1.1ex, "\bot", ""']
&&
\mathrm{Fun}_\Delta(K^\mathrm{op},\mathrm{Sp}^\mathbb{N}_\mathrm{sSet})
\end{tikzcd}
\end{equation}
of enriched functor categories.
\begin{lemma}
\label{lem:PSheafFibEquiv}
The fibre of $p\colon \mathrm{Fun}_\Delta \big(K^\mathrm{op}, \mathrm{Sp}^\mathbb{N}_\mathrm{sSet}\big)
\to
\mathrm{Fun}_\Delta(K^\mathrm{op}, \mathrm{sSet})$ at $X$ is equivalent to the category $\mathrm{Sp}^\mathbb{N}\big(\mathrm{Fun}_\Delta(K^\mathrm{op}, \mathrm{sSet})_{/X}\big)$ of sequential spectrum objects in the slice category over $X$.
\end{lemma}
\begin{proof}
We first analyse the fibre $F_X$ of $ \mathrm{Fun}_\Delta(K^\mathrm{op},\mathrm{sSet}_{\dslash \mathrm{sSet}})\to \mathrm{Fun}_\Delta (K^\mathrm{op},\mathrm{sSet})$ at $X$.
Objects of the fibre category are identified with simplicial functors $Z\colon K^\mathrm{op}\to \mathrm{sSet}_{\mathrm{sSet}}$ such that each $Z(k)$ is a retractive space over $X(k)$, this identification being compatible with the $\mathrm{sSet}$-enrichment. 
This identification of objects gives rise to an equivalence of simplicial categories between the fibre category $F_X$ and $\mathrm{Fun}_\Delta (K^\mathrm{op}, \mathrm{sSet})_{/X}$.
This equivalence is manifestly compatible with (fibrewise) $\mathrm{sSet}_\ast$-tensors, so that passing to underlying sequences of $\mathrm{sSet}_{\dslash \mathrm{sSet}}$-presheaves and unwinding the definitions proves the result.
\end{proof}

\begin{remark}
For a simplicial presheaf $X\in \mathrm{Fun}_\Delta (K^\mathrm{op}, \mathrm{sSet})$, its \emph{enriched category of elements} $\underline{\mathrm{el}}(X)$ is the full simplicial subcategory of $\mathrm{Fun}_{\Delta}(K^\mathrm{op}, \mathrm{sSet})_{/X}$ on those morphisms with domain in the image of the simplicial Yoneda embedding $\mathcal{Y}\colon K\hookrightarrow \mathrm{Fun}_\Delta(K^\mathrm{op},\mathrm{sSet})$.
There is an equivalence of simplicial categories between $ 
\mathrm{Fun}_\Delta(\underline{\mathrm{el}}(X)^\mathrm{op},\mathrm{sSet})$ and the fibre of $ \mathrm{Fun}_\Delta(K^\mathrm{op},\mathrm{sSet}_{\dslash \mathrm{sSet}})\to \mathrm{Fun}_\Delta (K^\mathrm{op},\mathrm{sSet})$ at $X$, and
upon taking sequential spectrum objects there is an equivalence of simplicial categories of 
$\mathrm{Fun}_\Delta\big(\underline{\mathrm{el}}(X)^\mathrm{op},\mathrm{Sp}^\mathbb{N}\big)$ with the fibre of $\mathrm{Fun}_\Delta(K^\mathrm{op},\mathrm{Sp}^\mathbb{N}_\mathrm{sSet})$ at $X$.
Working with enriched categories of elements in this manner is tantamount to regarding parametrised spectra as spectrum-valued local coefficient system, which is the perspective adopted in \cite{ando_parametrized_2018}.
\end{remark}

We now turn to the study of model structures on $\mathrm{Fun}_\Delta(K^\mathrm{op},\mathrm{Sp}^\mathbb{N}_\mathrm{sSet})$ and its fibres.
\begin{lemma}
\label{lem:EnrichedProjective}
The enriched functor category $\mathrm{Fun}_\Delta \big(K^\mathrm{op},\mathrm{Sp}^\mathbb{N}_\mathrm{sSet}\big)$ carries a \emph{projective model structure} with respect to which the fibrations and weak equivalences are objectwise.
The projective model structure is left proper, combinatorial and simplicial.
\end{lemma}
\begin{proof}
The projective model structure on functors  exists by \cite[Proposition A.3.3.2]{lurie_higher_2009}.
Projective cofibrations are in particular objectwise cofibrations, so that left properness is inherited objectwise from $\mathrm{sSet}$.
\end{proof}
As in Section \ref{sss:SequentialStabFibrewise}, \cite{hovey_spectra_2001}, we can equip each fibre category $\mathrm{Sp}^\mathbb{N}\big(\mathrm{Fun}_\Delta(K^\mathrm{op},\mathrm{sSet})_{/X}\big)$ with a stable $\mathrm{sSet}_\ast$-model structure that is left proper and  combinatorial.
The fibre inclusion functor
\[
\mathrm{Sp}^\mathbb{N}\big(\mathrm{Fun}_\Delta(K^\mathrm{op},\mathrm{sSet})_{/X}\big)\longrightarrow 
\mathrm{Fun}_\Delta\big(K^\mathrm{op}, \mathrm{Sp}^\mathbb{N}_\mathrm{sSet}\big)
\]
induced by Lemma \ref{lem:PSheafFibEquiv}
preserves cofibrations, fibrations, weak equivalences and $\mathrm{sSet}_\ast$-tensors.
\begin{definition}
\label{defn:ProlongLoc}
For a (small) simplicial category $K$, the \emph{tangent prolongation} of a set of morphisms $S$ in $\mathrm{Fun}_\Delta (K^\mathrm{op},\mathrm{sSet})$ is the set of morphisms
\begin{equation}
TS :=\bigcup_{k\geq 0} \Sigma^{\infty-k}_+ (S) \cup
0_- (S) \subset \mathrm{Mor}\big(\mathrm{Fun}_\Delta(K^\mathrm{op},\mathrm{Sp}^\mathbb{N}_\mathrm{sSet})\big)
\end{equation}
obtained using the simplicial functors of \eqref{eqn:BaseProjectPSheaves} and \eqref{eqn:PSheavesSuspendSpec}.
The \emph{retractive prolongation} of $S$ is the set of morphisms
\begin{equation}
S_{\dslash S} := \big((\mathtt{0})_{+(\mathtt{0})}\big) (S)\subset \mathrm{Mor}\big(\mathrm{Fun}_\Delta(K^\mathrm{op},\mathrm{sSet}_{\dslash \mathrm{sSet}})\big)\,.
\end{equation}
For a simplicial functor $X\colon K^\mathrm{op}\to \mathrm{sSet}$, the \emph{fibrewise stabilisation} of $S$ over $X$ is the set of morphisms 
\begin{equation}
\mathrm{Sp}(S)_{/X} := \bigcup_{k\geq 0} \Sigma^{\infty-k}_X (S_{/X})\subset \mathrm{Mor}\Big(\mathrm{Sp}^\mathbb{N}\big(
  \mathrm{Fun}_\Delta(K^\mathrm{op},\mathrm{sSet})_{/X}\big)\Big)
\end{equation}
obtained by applying the functors 
\[
\begin{tikzcd}
  \mathrm{Fun}_\Delta(K^\mathrm{op},\mathrm{sSet})_{/X}
  \ar[r, "{(-)_{+X}}"]
  &
  \mathrm{Fun}_\Delta(K^\mathrm{op},\mathrm{sSet})_{\dslash X}
  \ar[r, "{\Sigma^{\infty-k}_X}"]
  &
  \mathrm{Sp}^\mathbb{N}\big(
  \mathrm{Fun}_\Delta(K^\mathrm{op},\mathrm{sSet})_{/X}\big)
\end{tikzcd}
\]
to the set of commuting diagrams
\[
S_{/X}:=
\left\{
\!
\begin{tikzcd}[sep=small]
  K\ar[dr]\ar[rr, "\in S"]&& L
  \ar[dl]
  \\
  &X
\end{tikzcd}
\!\right\}
\] 
viewed as morphisms in $\mathrm{Fun}_\Delta (K^\mathrm{op},\mathrm{sSet})_{/X}$.
Finally, the \emph{fibrewise prolongation} of $S$ over $X$ is the set of morphisms 
\begin{equation}
S_{\dslash X} := (-)_{+X}\big(S_{/X}\big)
\end{equation}
in the functor category $\mathrm{Fun}_\Delta(K^\mathrm{op}, \mathrm{sSet})_{/X}$.
\end{definition}
\begin{lemma}
\label{lem:PSheafStabFib}
Let $X\in \mathrm{Fun}_\Delta(K^\mathrm{op}, \mathrm{sSet})$ be an $S$-local fibrant simplicial functor.
Under the equivalence of Lemma \ref{lem:PSheafFibEquiv}, the $\mathrm{Sp}(S)_{/X}$-local fibrant objects of $\mathrm{Sp}^\mathbb{N}(\mathrm{Fun}_\Delta(K^\mathrm{op},\mathrm{sSet})_{/X})$ are identified with the $TS$-local fibrant objects of the fibre of $p\colon \mathrm{Fun}_\Delta(K^\mathrm{op}, \mathrm{Sp}^\mathbb{N}_\mathrm{sSet})\to \mathrm{Fun}_\Delta(K^\mathrm{op}, \mathrm{sSet})$ over $X$.
\end{lemma}
\begin{proof}
The proof is an elaboration of the method used in Lemma \ref{lem:GlobRetSpaceComb} to compare fibrewise and global model structures.
We use the following characterisation of local fibrant objects: for a set $S'$ of morphisms in a simplicial model category $M$, a fibrant object $m\in M$ is $S'$-local precisely if the terminal morphism $m\to \ast$ has the right lifting property with respect to the set of $S'$-horns
$
\Lambda(S') := \{ \Delta^n \otimes K\coprod_{\,\partial\Delta^n \otimes K} \partial \Delta^n \otimes L \to \Delta^n \otimes L \mid (f\colon K\to L) \in S', \; n\geq 0\}$
\cite[Proposition 4.2.4]{hirschhorn_model_2003}.
Suppose $(X,A)\in \mathrm{Fun}_\Delta(K^\mathrm{op},\mathrm{Sp}^\mathbb{N}_\mathrm{sSet})$ is projectively fibrant.
The set of $TS$-horns splits as
\[
\Lambda (TS) = \Lambda \Bigg(\bigcup_{k\geq 0} \Sigma^{\infty-k}_+ (S)\Bigg) \cup \Lambda \big(0_-(S)\big)\,,
\]
where the hypothesis on $X$ guarantees that $(X,A)\to (\ast, 0_\ast)$ has the right lifting property with respect to the set $\Lambda(0_-(S))$.
As in the proof of Lemma \ref{lem:GlobRetSpaceComb}, by pushing forward along the map of base functors the morphism $(\sigma, \Psi)\colon (K, \Sigma^{\infty-k}_K L_{+K})\to (X,A)$ is equivalent to the datum of a morphism $\psi\colon \Sigma^{\infty-k}_X L_{+X} \to A$ covering the identity on $X$ (see also Remark \ref{rem:SeqSpecGlobalMorphisms}).
Under the simplicial equivalence of categories of Lemma \ref{lem:PSheafFibEquiv}, we have that $(X,A)\to (\ast, 0_\ast)$ has the right lifting property with respect to  $\Lambda (\bigcup_{k\geq 0} \Sigma^{\infty-k}_+ (S))$ precisely if the terminal morphism $A\to \ast$  in $ \mathrm{Sp}^\mathbb{N}(\mathrm{Fun}_\Delta(K^\mathrm{op},\mathrm{sSet})_{/X})$ has the right lifting property with respect to $\Lambda(\mathrm{Sp}(S)_{/X})$.
\end{proof}

We also require the unstable analogue of the last result, which has a similar (but easier) proof.
\begin{lemma}
\label{lem:PSheafUStabFib}
Let $X\in \mathrm{Fun}_\Delta(K^\mathrm{op},\mathrm{sSet})$ be an $S$-local fibrant simplicial functor.
Under the equivalence of Lemma \ref{lem:PSheafFibEquiv}, the $S_{/X}$-local fibrant objects of $\mathrm{Fun}_\Delta (K^\mathrm{op}, \mathrm{sSet})_{/X}$ are identified with the $S_{\dslash S}$-local fibrant objects of the fibre of $\mathrm{Fun}_\Delta(K^\mathrm{op},\mathrm{sSet}_{\dslash \mathrm{sSet}})\to \mathrm{Fun}_\Delta (K^\mathrm{op},\mathrm{sSet})$ over $X$.
\end{lemma}

We have thus far obtained a relative Quillen adjunction of localised simplicial model categories:
\[
\begin{tikzcd}
  L_{TS} \mathrm{Fun}_{\Delta}(K^\mathrm{op},\mathrm{Sp}^\mathbb{N}_\mathrm{sSet})
  \ar[r, shift left=1.1ex, "{\widetilde{\Omega}_-^{\infty}}", "{\top}"']
  \ar[r, leftarrow, shift left=-1.1ex,"{\Sigma^{\infty}_-}"']
  \ar[dr, bend right=15] &
  L_{S_{\dslash S}}\mathrm{Fun}_{\Delta}(K^\mathrm{op},\mathrm{sSet}_{\dslash\mathrm{sSet}})
  \ar[d]
  \\
  &
  L_S \mathrm{Fun}_\Delta(K^\mathrm{op},\mathrm{sSet})\,.
\end{tikzcd}
\]
By Lemmas \ref{lem:PSheafFibEquiv}, \ref{lem:PSheafStabFib}, and \ref{lem:PSheafUStabFib} (together with the recognition principle for simplicial Quillen adjunctions \cite[Corollary A.3.7.2]{lurie_higher_2009}) imply that the adjunction on fibres over $X\colon K^\mathrm{op}\to \mathrm{sSet}$ is presented by the simplicial Quillen adjunction
\[
\begin{tikzcd}
   L_{\mathrm{Sp}(S)_{/X}}\mathrm{Sp}^\mathbb{N}\big(\mathrm{Fun}_{\Delta}(K^\mathrm{op},\mathrm{sSet})_{/X}\big)
  \ar[rr, "{\widetilde{\Omega}_X^{\infty}}", "{\top}"', shift left =1.1ex]
  \ar[rr, leftarrow, "{\Sigma^{\infty}_X}"', shift left=-1.1ex]
  &&
  L_{S_{\dslash X}}\mathrm{Fun}_{\Delta}(K^\mathrm{op},\mathrm{sSet})_{\dslash X}
\end{tikzcd}
\]
exhibiting $L_{\mathrm{Sp}(S)_{/X}}\mathrm{Sp}^\mathbb{N}\big(\mathrm{Fun}_{\Delta}(K^\mathrm{op},\mathrm{sSet})_{/X}\big)$ as the stabilisation of $L_{S_{\dslash X}}\mathrm{Fun}_{\Delta}(K^\mathrm{op},\mathrm{sSet})_{\dslash X}$.

Recalling that $\mathcal{C}\cong L_S \mathrm{Fun}_\Delta (K^\mathrm{op},\mathrm{sSet})^\infty$, for a cofibrant-fibrant object $X\in L_S \mathrm{Fun}_\Delta(K^\mathrm{op}, \mathrm{sSet})$ an easy argument shows that
\[
\mathcal{C}_{\dslash X}  \cong L_{S_{\dslash X}}\mathrm{Fun}_\Delta (K^\mathrm{op}, \mathrm{sSet})_{\dslash X}^\infty\,,
\]
and hence that the restriction of the adjunction
\[
\begin{tikzcd}
   L_{TS} \mathrm{Fun}_{\Delta}(K^\mathrm{op},\mathrm{Sp}^\mathbb{N}_\mathrm{sSet})^\infty
  \ar[rr, "{\Omega_-^{\infty}}", "{\top}"', shift left =1.1ex]
  \ar[rr, leftarrow, "{\Sigma^{\infty}_-}"', shift left=-1.1ex]
  &&
  L_{S_{\dslash S}}\mathrm{Fun}_{\Delta}(K^\mathrm{op},\mathrm{sSet}_{\dslash\mathrm{sSet}})^\infty
\end{tikzcd}
\]
to fibres over $X$ is equivalent to the stabilisation $(\Sigma^\infty_X\dashv \Omega^\infty_X)\colon \mathcal{C}_{\dslash X}\to \mathrm{Sp}(\mathcal{C}_{\dslash X})\cong \mathrm{Sp}(\mathcal{C}_{/X})$.
From this it follows that the simplicial left Quillen functor  $p\colon L_{TS}\mathrm{Fun}_\Delta(K^\mathrm{op}, \mathrm{Sp}^\mathbb{N}_\mathrm{sSet})\to L_S \mathrm{Fun}_\Delta(K^\mathrm{op},\mathrm{sSet})$ presents the tangent $\infty$-bundle $T\mathcal{C}\to \mathcal{C}$.
In particular, $T\mathcal{C}$ is presented in terms of simplicial model categories as the left Bousfield localisation
\[
\begin{tikzcd}
  L_{TS} \mathrm{Fun}_\Delta(K^\mathrm{op}, \mathrm{Sp}^\mathbb{N}_\mathrm{sSet})
  \ar[rr, leftarrow, shift left =1.1ex]
  \ar[rr, hookrightarrow, shift left=-1.1ex, "{\bot}"]
  &&
 \mathrm{Fun}_\Delta(K^\mathrm{op}, \mathrm{Sp}^\mathbb{N}_\mathrm{sSet})\,.
\end{tikzcd}
\]
This completes the proof of Theorem \ref{thm:Tangent}.\qed
\begin{remark}
\label{rem:TCatSymorSeq}
Any presentable $\infty$-category $\mathcal{C}$ has finite limits and so determines a Cartesian symmetric monoidal $\infty$-category $\mathcal{C}^\times$.
This in turn induces the structure of a  symmetric monoidal $\infty$-category $T\mathcal{C}^\otimes$  on $T\mathcal{C}$ as well a symmetric monoidal functor $p^\otimes \colon T\mathcal{C}^\otimes\to \mathcal{C}^\times$ \cite{lurie_higher_2017}.
Generalising the case for spaces, the symmetric monoidal structure on $T\mathcal{C}^\otimes$ encodes external smash products of parametrised spectrum objects in $\mathcal{C}$.

Throughout this section we have been using sequential spectra to model $T\mathcal{S}$, but the same arguments carry over \emph{mutatis mutandis} with $\mathrm{Sp}^\Sigma_\mathrm{sSet}$ in place of $\mathrm{Sp}^\mathbb{N}_\mathrm{sSet}$.
In favourable situations, working with symmetric spectra enables us to present $p^\otimes \colon T\mathcal{C}^\otimes\to \mathcal{C}^\times$ in terms of symmetric monoidal model categories, for example in Theorem \ref{thm:SmoothESmash} below.
\end{remark}

\subsection{Twisted differential cohomology}
\label{SS:TwistDiffCoh}
In this brief epilogue, we apply our results to the smooth $\infty$-topos $\mathbf{H}$ to obtain model categories for twisted differential cohomology.
The objects of $\mathbf{H}$ are the \emph{smooth $\infty$-stacks}; homotopical presheaves of spaces
\[
\mathrm{CartSp}^\mathrm{op}\longrightarrow \mathcal{S}
\]
on the category $\mathrm{CartSp}$ of smooth Cartesian spaces $\mathbb{R}^n$ ($n\geq 0$) satisfying homotopical descent with respect to good open covers.
The $\infty$-topos $\mathbf{H}$ combines differential geometry with homotopy theory, providing a robust framework for handling many of the higher structures of quantum field theory and string theory \cite{schreiber_differential_2017}.
Smooth $\infty$-stacks are naturally interpreted as functors-of-points of \emph{smooth $\infty$-groupoids}; classifying objects for higher structures on manifolds.
In the same vein, the objects of $\mathrm{Sp}(\mathbf{H})$ are \emph{smooth spectra} representing differential cohomology theories on manifolds (\cite{bunke_differential_2016}, based on earlier observations of Schreiber \cite{schreiber_differential_2017}).
Subsuming both of these, the tangent $\infty$-category $T\mathbf{H}$ is the $\infty$-category of \emph{smooth parametrised spectra}, representing \emph{twisted} differential cohomology (see Definition \ref{defn:TwistDiffCoh} below).

\begin{construction}[The smooth $\infty$-topos]
For a Cartesian space $U\cong \mathbb{R}^n$ and good open cover $\mathcal{U}= \{U_i \to U\}_{i\in \mathcal{I}}$, we write $U_{i_0\dotsc i_k}:=U_{i_0}\cap \dotsb \cap U_{i_k}$ for $i_0, \dotsc, i_k \in \mathcal{I}$.
By the goodness of $\mathcal{U}$, each $U_{i_0\dotsc i_k}$ is either empty or a Cartesian space diffeomorphic to $U$. 
Let $\underline{U}$ denote the image of $U$ under the Yoneda embedding, then the \emph{\v{C}ech nerve of $\mathcal{U}$} is the simplicial presheaf
\[
\text{\v{C}}_\mathcal{U} \colon [k]\longmapsto \coprod_{i_0, \dotsc, i_k} \underline{U}_{i_0\dotsc i_k}\,,
\]
with simplicial structure maps induced by inclusions and coprojections.
The inclusions $U_i \to U$ induce a map of simplicial presheaves $\text{\v{c}}_\mathcal{U}\colon \text{\v{C}}_\mathcal{U}\to \underline{U}$ and we write
\[
\check{C} :=\big\{\text{\v{c}}_\mathcal{U}\colon \text{\v{C}}_\mathcal{U}\to \underline{U}\,\big|\, \text{$\mathcal{U}$ is a good open cover of $U\in \mathrm{CartSp}$} \big\} \subset \mathrm{Mor}\big(\mathrm{Fun}(\mathrm{CartSp}^\mathrm{op},\mathrm{sSet})\big)
\]
for the set of all such maps.

A model categorical presentation for the smooth $\infty$-topos $\mathbf{H}$ is given as by the left Bousfield localisation
\[
L_{\check{C}} \mathrm{Fun}(\mathrm{CartSp}^\mathrm{op},\mathrm{sSet})
\]
of the projective model structure on simplicial presheaves at the set of \v{C}ech nerves. 
The fibrant objects for the \v{C}ech-local model structure, namely the \emph{smooth $\infty$-stacks}, are precisely the simplicial presheaves $X\colon \mathrm{CartSp}^\mathrm{op}\to \mathrm{sSet}$ such that
\begin{itemize}
  \item $X$ takes values in Kan complexes; and
  \item $X$ satisfies \v{C}ech descent---for any Cartesian space $U$ and good open cover $\mathcal{U}=\{U_i \to U\}_{i\in \mathcal{I}}$, the canonical morphism
  \[
  X(U) 
  \longrightarrow 
  \mathrm{holim}
  \left(
    \!
     \begin{tikzcd}
  \displaystyle\prod_{i} X(U_i)
  \ar[r, shift left= 0.8ex]
  \ar[r, shift left= -0.8ex]
  &
  \displaystyle\prod_{i,j}X(U_{ij})
  \ar[r]
  \ar[r, shift left= 1.6ex]
  \ar[r, shift left= -1.6ex]
  &
  \;\dotsb
  \end{tikzcd}
    \!
  \right)
  \]
  is a weak equivalence.
\end{itemize}
The projectively cofibrant objects are generally more difficult to characterise, however degreewise coproducts of representables whose degenerate cells split off are projectively cofibrant \cite[Corollary 9.4]{dugger_universal_2001}.
In particular for a (paracompact) manifold $M$, which defines a simplicial presheaf via the external Yoneda embedding
$
\underline{M}\colon U\mapsto\mathrm{Map}(U,M)
$,
and any good open cover $\mathcal{V}$ of $M$,
the natural map of simplicial presheaves $\text{\v{C}}_\mathcal{V}\to \underline{M}$ exhibits the \v{C}ech nerve $\text{\v{C}}_\mathcal{V}$ as a cofibrant resolution of $\underline{M}$ with respect to the \v{C}ech-local model structure.
\end{construction}

\begin{theorem}
\label{thm:SmoothTangentooTopos}
The smooth tangent $\infty$-bundle $T\mathbf{H}\to \mathbf{H}$ is presented by the simplicial Quillen functor
\[
p\colon
L_{T\check{C}}\mathrm{Fun}(\mathrm{CartSp}^\mathrm{op}, \mathrm{Sp}^\mathbb{N}_\mathrm{sSet})
\longrightarrow 
L_{\check{C}} \mathrm{Fun}(\mathrm{CartSp}^\mathrm{op},\mathrm{sSet})\,.
\]
\end{theorem}
\begin{proof}
The follows from the proof of Theorem \ref{thm:Tangent}.
\end{proof}
\begin{definition}
The fibrant objects for the $T\check{C}$-local model structure are called \emph{smooth parametrised spectra}.
For a smooth parametrised spectrum $(\mathcal{X},\mathcal{A})\colon \mathrm{CartSp}^\mathrm{op}\to \mathrm{Sp}^\mathbb{N}_\mathrm{sSet}$, its  \emph{moduli $\infty$-stack of twists} is the smooth $\infty$-stack $p(\mathcal{X},\mathcal{A}) =\mathcal{X}$.
\end{definition}
By a straightforward adjointness argument, we find that the smooth parametrised spectra are precisely the functors $(\mathcal{X},\mathcal{A})\colon \mathrm{CartSp}^\mathrm{op}\to \mathrm{Sp}^\mathbb{N}_\mathrm{sSet}$ for which
\begin{itemize}
  \item $(\mathcal{X}(U), \mathcal{A}(U))\in \mathrm{Sp}^\mathbb{N}_\mathrm{sSet}$ is fibrant for each Cartesian space $U$; and
  \item for each good open cover $\mathcal{U}=\{U_i\hookrightarrow U\}$  of a Cartesian space $U$ and for each $n\geq 0$, the induced maps of spaces
  \[
\mathcal{X}(U)\longrightarrow  
\mathrm{holim}
  \left(
    \!
     \begin{tikzcd}
  \displaystyle\prod_{i} \mathcal{X}(U_i)
  \ar[r, shift left= 0.8ex]
  \ar[r, shift left= -0.8ex]
  &
  \displaystyle\prod_{i,j}\mathcal{X}(U_{ij})
  \ar[r]
  \ar[r, shift left= 1.6ex]
  \ar[r, shift left= -1.6ex]
  &
  \;\dotsb
  \end{tikzcd}
    \!
  \right)
\]
and
\[
\mathcal{A}_n(U)\longrightarrow
\mathrm{holim}
  \left(
  \!
   \begin{tikzcd}
  \displaystyle\prod_{i} \mathcal{A}_n(U_i)
  \ar[r, shift left= 0.8ex]
  \ar[r, shift left= -0.8ex]
  &
  \displaystyle\prod_{i,j}\mathcal{A}_n(U_{ij})
  \ar[r]
  \ar[r, shift left= 1.6ex]
  \ar[r, shift left= -1.6ex]
  &
  \;\dotsb
  \end{tikzcd}
  \!
  \right)
\]
are weak equivalences.
\end{itemize}
Fix a manifold $M$ and a smooth parametrised spectrum $(\mathcal{X},\mathcal{A})$.
Choosing a good open cover $\mathcal{U}= \{U_i \hookrightarrow M\}_{i\in \mathcal{I}}$, we set
\[
\mathcal{X}(M) :=
\mathrm{holim}
  \left(\!
 \begin{tikzcd}
  \displaystyle\prod_{i} \mathcal{X}(U_i)
  \ar[r, shift left= 0.8ex]
  \ar[r, shift left= -0.8ex]
  &
  \displaystyle\prod_{i,j}\mathcal{X}(U_{ij})
  \ar[r]
  \ar[r, shift left= 1.6ex]
  \ar[r, shift left= -1.6ex]
  &
  \;\dotsb
  \end{tikzcd}
  \!
  \right)
\]
and for $n\geq 0$
\[
\mathcal{A}_n(M):=
\mathrm{holim}
  \left(
  \!
   \begin{tikzcd}
  \displaystyle\prod_{i} \mathcal{A}_n(U_i)
  \ar[r, shift left= 0.8ex]
  \ar[r, shift left= -0.8ex]
  &
  \displaystyle\prod_{i,j}\mathcal{A}_n(U_{ij})
  \ar[r]
  \ar[r, shift left= 1.6ex]
  \ar[r, shift left= -1.6ex]
  &
  \;\dotsb
  \end{tikzcd}
    \!
  \right)\,.
\]
The spaces $\mathcal{X}(M)$ and $\mathcal{A}_n(M)$ are independent of the choice of good open cover $\mathcal{U}$ up to weak equivalence;
each $\mathcal{A}_n(M)$ is a retractive space over $\mathcal{X}(M)$ and together they assemble into a fibrant $\Omega_{\mathcal{X}(M)}$-spectrum.
\begin{remark}
For a smooth manifold $M$ and smooth $\infty$-stack $\mathcal{X}$, the space $\mathcal{X}(M)$ defined above is weakly equivalent to the $\infty$-categorical hom-space $\mathbf{H}(\underline{M}, \mathcal{X})$.
By definition, the homotopy groups $\pi_\ast \mathbf{H}(\underline{M}, \mathcal{X})$ compute the non-abelian differential $\mathcal{X}$-cohomology of $M$.
\end{remark}
\begin{definition}
\label{defn:TwistDiffCoh}
For a manifold $M$ and smooth parametrised spectrum $(\mathcal{X},\mathcal{A})$, fix an element $\tau\colon \ast \to \mathcal{X}(M)$ (equivalently a map of smooth $\infty$-stacks $\tau\colon \underline{M}\to \mathcal{X}$).
The abelian groups
\[
\mathcal{A}^{\tau + \bullet}(M) := \pi^\mathrm{st}_{-\bullet} \big(\tau^\ast \mathcal{A}(M)\big)
\]
are the \emph{$\tau$-twisted differential $\mathcal{A}$-cohomology groups of $M$}.
\end{definition}
\begin{remark}
$\tau^\ast\mathcal{A}$ is a fibrant $\Omega$-spectrum, so 
$\mathcal{A}^{\tau+n} (M) \cong \pi_{k-n} (\tau^\ast \mathcal{A}_{k}(M))$
for  $\min\{k-n,k\} \geq 0$.
\end{remark}

\begin{lemma}
\label{lem:ChangeofTwist}
Suppose that $\tau, \xi\colon M\to \mathcal{X}$ represent the same class in non-abelian differential cohomology, that is $[\tau]= [\xi]\in \pi_0\mathbf{H}(\underline{M}, \mathcal{X})$.
Then there is a (non-canonical) isomorphism
\[
\mathcal{A}^{\tau+\bullet}(M) \cong \mathcal{A}^{\xi+\bullet}(M)
\]
of twisted differential $\mathcal{A}$-cohomology groups.
\end{lemma}
\begin{proof}
Since $X(M) \cong \mathbf{H}(\underline{M},\mathcal{X})$ is a Kan complex, there is a $1$-simplex $\sigma\colon \Delta^1 \to \mathbf{H}(\underline{M}, \mathcal{X})$ restricting to $\tau$ and $\xi$ at the vertices $0$ and $1$ respectively.
Writing $i_0, i_1\colon \ast \to \Delta^1$ for the vertex inclusions and $p\colon \Delta^1 \to \ast$ for the terminal map, by Theorem \ref{thm:SeqSpecStructureTheorem} we have Quillen equivalences 
\[
\begin{tikzcd}
\mathrm{Sp}^\mathbb{N}
\ar[rr, shift left=1.1ex, "(i_{0/1})_!"]
\ar[rr, leftarrow, shift left=-1.1ex, "\bot", "(i_{0/1})^\ast"']
&&
\mathrm{Sp}^\mathbb{N}_{\Delta^1}
\end{tikzcd}
  \qquad
  \mbox{ and }
  \qquad
\begin{tikzcd}
\mathrm{Sp}^\mathbb{N}
\ar[rr, leftarrow, shift left=1.1ex, "p_!"]
\ar[rr, shift left=-1.1ex, "\bot", "p^\ast"']
&&
\mathrm{Sp}^\mathbb{N}_{\Delta^1}
\end{tikzcd}
\]
in which all left adjoints preserve stable weak equivalences (by Lemma \ref{lem:SpecPseudoFunAreProper}).
The $((i_0)_!\dashv i_0^\ast)$-counit
\[
(i_0)_! \tau^\ast \mathcal{A}\cong (i_0)_! i_0^\ast \sigma^\ast \mathcal{A}\longrightarrow \sigma^\ast \mathcal{A}
\]
is a stable weak equivalence of sequential $\Delta^1$-spectra, and applying $p_!$ gives a stable weak equivalence $\tau^\ast \mathcal{A}\to p_! \sigma^\ast \mathcal{A}$.
Choosing a fibrant replacement $p_! \sigma^\ast \mathcal{A}\to (p_! \sigma^\ast \mathcal{A})^f$, there is thus a zig-zag of stable weak equivalences of fibrant $\Omega_{\Delta^1}$-spectra
\[
p^\ast \tau^\ast \mathcal{A}\xrightarrow{\;\hphantom{\text{unit}}\;} 
p^\ast (p_! \sigma^\ast \mathcal{A})^f \xleftarrow{\;\text{unit}\;}
\sigma^\ast \mathcal{A}\,.
\]
Under the pullback functor $i_1^\ast$, this becomes a zig-zag of stable weak equivalences 
\[
\tau^\ast \mathcal{A}
\longrightarrow (p_! \sigma^\ast \mathcal{A})^f
\longleftarrow
\xi^\ast \mathcal{A}
\]
of fibrant $\Omega$-spectra.
Passing to stable homotopy groups completes the proof.
\end{proof}

\begin{remark}[Descent spectral sequences]
\label{rem:DescSS}
Fix a choice of smooth parametrised spectrum $(\mathcal{X},\mathcal{A})$, a manifold $M$ with good open cover $\mathcal{U}=\{U_i\hookrightarrow M\}_{i\in \mathcal{I}}$, and a twist $\tau\colon \ast \to\mathcal{X}(M)$.
Then the cosimplicial simplicial sets
\[
 \begin{tikzcd}
  \displaystyle\prod_{i} \mathcal{X}(U_i)
  \ar[r, shift left= 0.8ex]
  \ar[r, shift left= -0.8ex]
  &
  \displaystyle\prod_{i,j}\mathcal{X}(U_{ij})
  \ar[r]
  \ar[r, shift left= 1.6ex]
  \ar[r, shift left= -1.6ex]
  &
  \;\dotsb
  \end{tikzcd}
\qquad\mbox{ and }
\qquad
 \begin{tikzcd}
  \displaystyle\prod_{i} \mathcal{A}_n(U_i)
  \ar[r, shift left= 0.8ex]
  \ar[r, shift left= -0.8ex]
  &
  \displaystyle\prod_{i,j} \mathcal{A}_n(U_{ij})
  \ar[r]
  \ar[r, shift left= 1.6ex]
  \ar[r, shift left= -1.6ex]
  &
  \;\dotsb
  \end{tikzcd}
\]
are Reedy fibrant for all $n\geq 0$.
The inclusions $U_{i_0\dotsb i_k}\hookrightarrow M$ induce a compatible family of twists $\tau_{i_0\dotsb i_k}\in \mathcal{X}(U_{i_0\dotsb i_k})$ and after taking fibres we get weak equivalences
\[
\tau^\ast \mathcal{A}_n(M)
\longrightarrow
\mathrm{holim}
  \left(
  \!
  \begin{tikzcd}
  \displaystyle\prod_{i} \tau_i^\ast \mathcal{A}_n(U_i)
  \ar[r, shift left= 0.8ex]
  \ar[r, shift left= -0.8ex]
  &
  \displaystyle\prod_{i,j} \tau^\ast_{ij}\mathcal{A}_n(U_{ij})
  \ar[r]
  \ar[r, shift left= 1.6ex]
  \ar[r, shift left= -1.6ex]
  &
  \;\dotsb
  \end{tikzcd}
  \right),
  \qquad n\geq 0\,.
\]
Reedy fibrancy implies that each of these homotopy limits is computed by the totalisation functor, so assuming that the indexing set $\mathcal{I}$ is finite there are convergent second octant Bousfield--Kan spectral sequences (see, e.g.~\cite[VIII Section 1]{goerss_simplicial_2009})
\begin{equation}
\label{eqn:DescSS}
E^{s,t}_2 \cong \check{H}^s \big(M;  \underline{\mathcal{A}}^{\tau+n-t} \big) \Longrightarrow  \mathcal{A}^{\tau+ n+s-t} (M)\,,\qquad t\geq s\geq 0,\;\; n\geq 0\,,
\end{equation}
with $E_2$-terms given by \v{C}ech cohomology of the presheaves
$
\underline{\mathcal{A}}^{\tau+n-t} \colon U \mapsto \mathcal{A}^{\tau_U +n-t}(U)
$.
\end{remark}

\begin{remark}[Cohesion]
\label{rem:Cohesion}
The $\infty$-topos $\mathbf{H}$ is \emph{cohesive}, meaning that there is a quadruple of adjoint functors
\[
(\Pi\dashv \mathrm{Disc}\dashv \Gamma \dashv \mathrm{coDisc})
\colon
\begin{tikzcd}
\mathbf{H}
\ar[rr, shift left=3.3ex]
\ar[rr, shift left=1.1ex, "\bot", hookleftarrow]
\ar[rr, shift left=-1.1ex, "\bot"]
\ar[rr, shift left=-3.3ex, "\bot", hookleftarrow]
&&
\mathcal{S}\,.
\end{tikzcd}
\]
The induced idempotent (co)modalities $\smallint:= \mathrm{Disc}\circ \Pi$, $\flat := \mathrm{Disc}\circ \Gamma$ and $\sharp := \mathrm{coDisc}\circ \Gamma$ equip $\mathbf{H}$ with a rich notion of intrinsic higher geometry (see \cite{schreiber_differential_2017} for a detailed account).
The fully-faithful functor $\mathrm{Disc}$ sends a space $K$ to the discrete $\infty$-stack $U\mapsto K$, whereas $\Pi$ and $\Gamma$ respectively compute ($\infty$-) colimits and limits of smooth $\infty$-stacks  over $\mathrm{CartSp}^\mathrm{op}$.
The site $\mathrm{CartSp}$ is cosifted (it has finite products), hence $\Gamma X \cong X(\ast)$ is the global sections functor and $\Pi$ preserves finite products.
The colimit functor $\Pi$ is interpreted as geometric realisation, and indeed for any manifold $M$ viewed as a smooth $\infty$-stack, Borsuk's nerve theorem implies a homotopy equivalence between $\Pi\underline{M}$ and the underlying topological space $uM$ of $M$.

The cohesion adjunctions for  $\mathbf{H}$ prolong to a quadruple of functors
\[
(T\Pi\dashv T\mathrm{Disc}\dashv T\Gamma \dashv T\mathrm{coDisc})
\colon
\begin{tikzcd}
T\mathbf{H}
\ar[rr, shift left=3.3ex]
\ar[rr, shift left=1.1ex, "\bot", hookleftarrow]
\ar[rr, shift left=-1.1ex, "\bot"]
\ar[rr, shift left=-3.3ex, "\bot", hookleftarrow]
&&
T\mathcal{S}
\end{tikzcd}
\]
covering the  corresponding functors for $\mathbf{H}$.
Via Theorem \ref{thm:SmoothTangentooTopos}, the functor $T\mathrm{Disc}$ is presented by the simplicial functor
\begin{align*}
\mathrm{const}\colon \mathrm{Sp}^\mathbb{N}_\mathrm{sSet}
&
\longrightarrow  
L_{T\check{C}}\mathrm{Fun}(\mathrm{CartSp}^\mathrm{op}, \mathrm{Sp}^\mathbb{N}_\mathrm{sSet})
\\
K&\longmapsto\left[U\mapsto K\right]\,,
\end{align*}
which is both left and right Quillen (with left and right adjoints given by taking, respectively, colimits and limits over $\mathrm{CartSp}^\mathrm{op}$).
Any parametrised spectrum $(X,A)\in \mathrm{Sp}^\mathbb{N}_\mathrm{sSet}$ thus presents a (discrete) twisted differential cohomology theory.
\end{remark}

\begin{example}
Let $X$ be a Kan complex and $A$ a fibrant $\Omega_X$-spectrum, so that $\mathrm{const}(X,A)$ is a smooth parametrised spectrum.
Write $uM$ for the underlying space of the manifold $M$, then the natural weak equivalences 
$
X(M) \cong \mathbf{H}(\underline{M}, \mathrm{Disc}X)\cong \mathcal{S}(\Pi M, X) \cong \mathcal{S}(uM, X)
$
imply that a twist $\tau\in X(M)$ is equivalently a map of spaces $\tau\colon uM \to X$.
Pulling back $A$, we define the local coefficient system
$
\mathcal{L}_\bullet (A, \tau)\colon m\mapsto \pi^\mathrm{st}_\bullet (\tau(m)^\ast A)
$
on $uM$.
When $M$ (hence $uM$) is compact, the descent spectral sequence \eqref{eqn:DescSS} recovers the twisted Serre spectral sequence \cite[Section 20.4]{may_parametrized_2006}
\[
E^{p,q}_2 \cong H^p\big(uM; \mathcal{L}_{-q}(A, \tau)\big)\Longrightarrow  A^{\tau +p+q}(uM)\,.
\]
\end{example}

\begin{example}[Smooth spectral $\widehat{R}$-line bundles]
\label{ex:SmoothSpecLineBun}
Working in the $\infty$-categorical setting, Bunke and Nikolaus \cite{bunke_twisted_2014} have defined a family of examples of twisted differential cohomology theories by twisting by smooth Picard $\infty$-groupoids.
We briefly recall this approach and outline how these examples define smooth parametrised spectra (that is, objects of $T\mathbf{H}$).

For a smooth spectrum $E\in \mathrm{Sp}(\mathbf{H})$ we consider the cofibre sequence $\flat E\to E\to \flat_\mathrm{dR} E$ generated by the $\flat$-counit (Remark \ref{rem:Cohesion}).
Cohesion implies that the naturality square of the $\smallint$-unit
\begin{equation}
\label{eqn:CohesionFracture}
\begin{tikzcd}
  E \ar[r]\ar[d] & \flat_\mathrm{dR} E \ar[d]
  \\
  \smallint E\ar[r] & \smallint\flat_\mathrm{dR}E
\end{tikzcd}
\end{equation}
is a pullback.
In this case, we say that $E$ is a \emph{differential refinement} of the spectrum $\Pi E$ (regarded as a discrete smooth spectrum via $\smallint E = (\mathrm{Disc}\circ \Pi) E$).
If $E$ is a (commutative) algebra object in $\mathrm{Sp}(\mathbf{H})$ then so are $\smallint E$, $\flat E$, $\flat_\mathrm{dR} E$, and \eqref{eqn:CohesionFracture} becomes a pullback diagram of (commutative) algebra objects.

For an ordinary commutative ring spectrum $R$, the Picard $\infty$-groupoid $\mathrm{Pic}_R$ of invertible $R$-module spectra is equipped with a canonical functor of $\infty$-categories
\begin{align*}
\ell_R \colon \mathrm{Pic}_R& \longrightarrow \mathrm{Sp}(\mathcal{S})\\
M &\longmapsto M\,.
\end{align*}
Under Lurie's unstraightening construction this is equivalent to a $\mathrm{Pic}_R$-parametrised spectrum $\mathrm{Line}_R$.
Fixing a space $X$, then the \emph{space of $R$-twists on $X$} is the mapping space $[X, \mathrm{Pic}_R]$.
For an $R$-twist $\tau\colon X\to \mathrm{Pic}_R$ the composite
\[
\begin{tikzcd}
  X
  \ar[r, "\tau"]
  &
  \mathrm{Pic}_R
  \ar[r, "\ell_R"]
  &
  \mathrm{Sp}(\mathcal{S})
\end{tikzcd}
\]
is equivalent to the $X$-parametrised spectrum $\tau^\ast \mathrm{Line}_R$ under unstraightening (see \cite{ando_parametrized_2018} for more details).
The $\tau$-twisted $R$-cohomology groups of $X$ are then
$R^{\tau+n}(X) := \pi^\mathrm{st}_{-n} (X_! \tau^\ast \mathrm{Line}_R)$ or, equivalently, as 
\[
R^{\tau+n}(X):= 
\pi^\mathrm{st}_{-n}\left[
\mathrm{lim}\Big(\!
\begin{tikzcd}
  X
  \ar[r,"\tau"]
  &
  \mathrm{Pic}_R
  \ar[r, "\ell_R"]
  &
  \mathrm{Sp}(\mathcal{S})
\end{tikzcd}\!
\Big)\right]\,.
\]
In mimicking this construction for twisted differential cohomology, 
Bunke--Nikolaus focus on differential refinements for which $\flat_{\mathrm{dR}}\widehat{R}$ is equivalent to a sheaf of commutative differential graded algebras (CDGAs) of the form
\[
\Omega A\colon U\longmapsto \Omega^\bullet(U) \otimes A\,,
\]
for a CDGA $A$ over $\mathbb{R}$.
CDGAs over $\mathbb{Z}$ are equivalent commutative $H\mathbb{Z}$-algebra spectra by the stable Dold--Kan correspondence $H\colon \mathrm{CDGA}_\mathbb{Z}\to \mathrm{CAlg}(H\mathbb{Z}\mathrm{-Mod})$ \cite{shipley_HZ_2007} and for \eqref{eqn:CohesionFracture} to be a pullback we must have an equivalence of commutative ring spectra $HA \cong R\wedge H\mathbb{R}$.
Bunke--Nikolaus then construct the \emph{$\infty$-stack of differential $\widehat{R}$-twists} as a pullback
\[
\begin{tikzcd}
  \mathrm{Tw}_{\widehat{R}}
  \ar[r]
  \ar[d]
  &
  \mathrm{Pic}^{1, \mathrm{fl,wlc}}_{\Omega A}
  \ar[d]
  \\
  \mathrm{Disc}\big(\mathrm{Pic}_R\big)
  \ar[r]
  &
  \mathrm{Disc}\big(\mathrm{Pic}_{HA}\big)\,,
\end{tikzcd}
\]
where $\mathrm{Pic}^{1, \mathrm{fl,wlc}}_{\Omega A}$ is the Picard $1$-stack of flat invertible $\Omega A$-modules $M$ that are moreover \emph{weakly locally constant} in the sense that for each Cartesian space $U$ and $x\in U$, there is a neighbourhood $V$ of $x$ such that $M|_V$ is weakly equivalent to a constant sheaf on $V$.
The bottom horizontal arrow in the above diagram is induced by smashing with $H\mathbb{R}$ using $HA\cong R\wedge H\mathbb{R}$; the right-hand vertical arrow is abstractly defined using the quasi-isomorphisms $\Omega^\bullet(U) \otimes A \cong A$ for each Cartesian space $U$ together with the weak local constancy condition.

For any Cartesian space $U$, there is a functor 
\[
\ell_{\Omega A(U)}\colon \mathrm{Pic}^{1,\mathrm{fl,wlc}}_{\Omega A}(U)\longrightarrow \mathrm{Sp}(\mathcal{S})
\] 
sending flat invertible $\Omega A(U)$-modules $N(U)\mapsto HN(U)$ to their underlying spectra under the stable Dold--Kan correspondence.
Taking pullbacks with $\ell_R$ and $\ell_{HA}$ we get a functor of $\infty$-categories
\[
\ell_{\widehat{R}}(U)\colon \mathrm{Tw}_{\widehat{R}}(U) \longrightarrow \mathrm{Sp}(\mathcal{S})\,,
\]
which, after unstraightening, defines a $\mathrm{Tw}_{\widehat{R}}(U)$-parametrised spectrum $\mathrm{Line}_{\widehat{R}}(U)$.
These data are natural in $U$ and together define a smooth parametrised spectrum $(\mathrm{Tw}_{\widehat{R}}, \mathrm{Line}_{\widehat{R}})\in T\mathbf{H}$.
For any manifold $M$ we thus have a $\mathrm{Tw}_{\widehat{R}}(M)$-parametrised spectrum $\mathrm{Line}_{\widehat{R}}(M)$, which becomes a functor of $\infty$-categories
\[
\ell_{\widehat{R}}(M)\colon \mathrm{Tw}_{\widehat{R}}(M) \longrightarrow \mathrm{Sp}(\mathcal{S})
\]
after straightening.
A \emph{differential $\widehat{R}$-twist} is a choice
$
\widehat{\tau}\in 
\mathrm{Tw}_{\widehat{R}}(M)\cong \mathbf{H}(\underline{M},\mathrm{Tw}_{\widehat{R}})
$; the corresponding \emph{$\widehat{\tau}$-twisted differential $\widehat{R}$-cohomology of $M$}\footnote{Beware the minor conflict of terminology with Definition \ref{defn:TwistDiffCoh}.} is then computed as the stable homotopy groups of the spectrum
$\widehat{\tau}^\ast \mathrm{Line}_{\widehat{R}}$ or, equivalently, of the spectrum
\[
\begin{tikzcd}
  \ast
  \ar[r, "\widehat{\tau}"]
  &
  \mathrm{Tw}_{\widehat{R}}(M)
  \ar[r, "\ell_{\widehat{R}}(M)"]
  &
  \mathrm{Sp}(\mathcal{S})\,.
\end{tikzcd}
\]
The former is in line with the global perspective taken in this article, whereas Bunke--Nikolaus adopt the latter approach.

In order to deal with homotopy coherence issues, the construction outlined above makes heavy use of $\infty$-categorical machinery.
Nevertheless, by Theorem \ref{thm:SmoothTangentooTopos} each of the smooth parametrised spectra $(\mathrm{Tw}_{\widehat{R}}, \mathrm{Line}_{\widehat{R}})$ is presented by a projectively fibrant $T\check{C}$-local $\mathrm{Sp}^\mathbb{N}_\mathrm{sSet}$-valued presheaf on smooth Cartesian spaces---though practical explicit models are generally difficult to get one's hands on.
Furthermore, for smooth parametrised spectra of the form $(\mathrm{Tw}_{\widehat{R}},\mathrm{Line}_{\widehat{R}})$ the descent spectral sequence of Remark \ref{rem:DescSS}, \eqref{eqn:DescSS}, coincides with the twisted Atiyah--Hirzebruch spectral sequence studied in \cite{grady_twisted_2017}.
\end{example}

Replacing $\mathrm{Sp}^\mathbb{N}_\mathrm{sSet}$ by $\mathrm{sSet}^\Sigma_{\mathrm{sSet}}$ gives another model for the tangent $\infty$-bundle $p\colon T\mathbf{H}\to \mathbf{H}$
\[
p_\Sigma\colon L_{T\check{C}}\mathrm{Fun} (\mathrm{CartSp}^\mathrm{op}, \mathrm{Sp}^\Sigma_\mathrm{sSet})\longrightarrow L_{\check{C}} \mathrm{Fun}(\mathrm{CartSp}^\mathrm{op},\mathrm{sSet})
\]
(Remark \ref{rem:TCatSymorSeq}).
The next result shows that $L_{T\check{C}}\mathrm{Fun} (\mathrm{CartSp}^\mathrm{op}, \mathrm{Sp}^\Sigma_\mathrm{sSet})$ is a simplicial symmetric monoidal model category (in the sense of \cite[Definition 4.1.7.8]{lurie_higher_2017}) with respect to the objectwise external smash product and, moreover, that $p_\Sigma$ is a symmetric monoidal simplicial Quillen functor.

\begin{theorem}
\label{thm:SmoothESmash}
The smooth tangent $\infty$-bundle $p^\otimes\colon T\mathbf{H}^\otimes \to \mathbf{H}^\times$ is presented by the symmetric monoidal simplicial Quillen functor
\[
p_\Sigma\colon L_{T\check{C}}\mathrm{Fun} (\mathrm{CartSp}^\mathrm{op}, \mathrm{Sp}^\Sigma_\mathrm{sSet})\longrightarrow L_{\check{C}} \mathrm{Fun}(\mathrm{CartSp}^\mathrm{op},\mathrm{sSet})\,.
\] 
\end{theorem}
\begin{proof}
It suffices to show that $L_{T\check{C}}\mathrm{Fun} (\mathrm{CartSp}^\mathrm{op}, \mathrm{Sp}^\Sigma_\mathrm{sSet})$ and $L_{\check{C}} \mathrm{Fun}(\mathrm{CartSp}^\mathrm{op},\mathrm{sSet})$ are simplicial symmetric monoidal model categories with respect to the objectwise external smash and Cartesian products respectively.
The main goal is therefore to check that the objectwise symmetric monoidal structures on $\mathrm{sSet}$- and $\mathrm{Sp}^\Sigma_\mathrm{sSet}$-valued presheaves are Quillen bifunctors for the projective model structure and remain Quillen bifunctors after localisation (at the set of \v{C}ech nerves and its tangent prolongation, respectively).

For a Cartesian space $U\in \mathrm{CartSp}^\mathrm{op}$ and $(X,A)\in \mathrm{Sp}^\Sigma_\mathrm{sSet}$, write $\underline{U}\otimes (X,A)$ for the presheaf
\[
\underline{U}\otimes (X,A)\colon  V\longmapsto C^\infty(V,U) \otimes (X,A) \,,
\]
with $\otimes$ on the right-hand side denoting the basewise simplicial tensoring of $\mathrm{Sp}^\Sigma_\mathrm{sSet}$ (Remark \ref{rem:EnrichSymSpec}); since each $C^\infty(V,U)$ is  discrete this is the same as forming $C^\infty(V,U)$-indexed coproducts.
Choose sets $\mathcal{I}_{\Sigma\mathrm{-Sp}}$ and $\mathcal{J}_{\Sigma\mathrm{-Sp}}$ of generating cofibrations and acyclic cofibrations for $\mathrm{Sp}^\Sigma_\mathrm{sSet}$, then
\[
\mathcal{I}^{\mathrm{proj}}:=\big\{\underline{U}\otimes i \,\big|\, U\in \mathrm{CartSp}, \, i\in \mathcal{I}_{\Sigma\mathrm{-Sp}}\big\}
\;\;
\mbox{ and }
\;\;
\mathcal{J}^{\mathrm{proj}}:=\big\{\underline{U}\otimes j \,\big|\, U\in \mathrm{CartSp}, \, j\in \mathcal{J}_{\Sigma\mathrm{-Sp}}\big\}
\]
are generating sets of cofibrations and acyclic cofibrations for the projective model structure on $\mathrm{Fun}(\mathrm{CartSp}^\mathrm{op},\mathrm{Sp}^\Sigma_\mathrm{sSet})$ \cite[Theorem 11.6.1]{hirschhorn_model_2003}.
For any Cartesian space $U$, the simplicial functor $\underline{U}\otimes (-)$ is left Quillen.
Fixing a pair of morphisms
\begin{align*}
\underline{U}\otimes i\colon \underline{U}\otimes (X_{i,0}, A_{i,0})&\longrightarrow \underline{U}\otimes (X_{i,1}, A_{i,1})
\\
\underline{V}\otimes j\colon \underline{V}\otimes (Y_{j,0}, B_{j,0})&\longrightarrow \underline{V}\otimes(Y_{j,1}, B_{j,1})\,,
\end{align*}
then since $\mathrm{CartSp}$ is closed under finite products we have
$
(\underline{U}\otimes i)\,\square\,(\underline{V}\otimes j)
\cong 
\underline{U\times V}\otimes (i\,\square\, j)
$.
It now follows that the objectwise external smash product is a Quillen bifunctor with respect to the projective model structure.
The Cartesian product is a Quillen bifunctor for the projective model structure on $\mathrm{Fun}(\mathrm{CartSp}^\mathrm{op},\mathrm{sSet})$ by a similar argument.
In both cases the fact that $\mathrm{CartSp}$ has a terminal object implies that the monoidal unit is cofibrant so that we have symmetric monoidal model categories, and the monoidal and simplicial structures are clearly compatible in the sense of \cite[Definition 4.1.7.7]{lurie_higher_2017}.

We now show that $L_{\check{C}}\mathrm{Fun}(\mathrm{CartSp}^\mathrm{op},\mathrm{sSet})$ is a simplicial symmetric monoidal model category.
Fix a cofibration of simplicial sets $i\colon K\hookrightarrow L$ together with Cartesian spaces $U, V$ and a good open cover $\mathcal{V}=\{V_i\hookrightarrow V\}$ of $V$.
Then there are natural isomorphisms of simplicial presheaves
\[
\big(\underline{U}\otimes K\big)\times \text{\v{C}}_\mathcal{V}\cong \big(\underline{U}\times \text{\v{C}}_\mathcal{V}\big)\otimes K
\cong  
\text{\v{C}}_{U\times \mathcal{V}} \otimes K\,,
\]
where $U\times \mathcal{V}=\{U\times V_i\hookrightarrow U\times V\}$ is a good open cover of $U\times V$.
The pushout-product of $\text{\v{c}}_\mathcal{V}\colon \text{\v{C}}_\mathcal{V}\to \underline{V}$ and $\underline{U}\otimes i$ is thus naturally isomorphic to the morphism $\rho$ in the diagram
\[
 \begin{tikzcd}
    \text{\v{C}}_{U\times \mathcal{V}} \otimes K
    \ar[r]\ar[d, "{\text{\v{c}}_{U\times \mathcal{V}}\otimes K}"'] & 
    \text{\v{C}}_{U\times \mathcal{V}} \otimes L
    \ar[dr, bend left=15, "{\text{\v{c}}_{U\times \mathcal{V}}\otimes L}"]
    \ar[d] &
    \\
    \underline{U\times V} \otimes K
    \ar[r]
    &
    P
    \ar[r, "\rho"]
    &
    \underline{U\times V} \otimes L\,,
  \end{tikzcd}
 \] 
where $P$ is the pushout.
The morphisms $\text{\v{c}}_{U\times \mathcal{V}}\otimes K$ and $\text{\v{c}}_{U\times \mathcal{V}}\otimes K$ are clearly $\check{C}$-local equivalences. 
As $\text{\v{C}}_{U\times\mathcal{V}}\otimes K\to \text{\v{C}}_{U\times\mathcal{V}}\otimes L$ is a cofibration, left properness of the $\check{C}$-local model structure implies that $\check{C}_{U\times \mathcal{V}}\otimes L\to P$, too, is a $\check{C}$-local equivalence.
It follows that $\rho$, and hence $\text{\v{c}}_{\mathcal{V}}\,\square\,(\underline{U}\otimes i)$ is a $\check{C}$-local equivalence.
This shows that the Cartesian product is a Quillen bifunctor for the $\check{C}$-local model structure on simplicial presheaves, so that $L_{\check{C}}\mathrm{Fun}(\mathrm{CartSp}^\mathrm{op}, \mathrm{sSet})$ is a symmetric monoidal model category.

Returning to presheaves of parametrised symmetric spectra, to show that the objectwise smash product is a Quillen bifunctor for the $T\check{C}$-local model structure it is sufficient to show that the set $\mathcal{I}^\mathrm{proj}\,\square\, T\check{C}$ consists of $T\check{C}$-local equivalences.
In fact, we need only consider pushout-products of morphisms of the form 
\begin{equation}
\label{eqn:THESmashMaps}
i:= \underline{U}\otimes \mathbf{\Sigma}^{\infty-k}_{\Delta^n} \Big(\partial\Delta^n_{+\Delta^n}
\longrightarrow
\Delta^n_{+\Delta^n} \Big)
\in\mathcal{I}^\mathrm{proj}
\quad
\mbox{ and }
\quad
\check{c}:= \mathbf{\Sigma}^{\infty-l}_+ \text{\v{c}}_{\mathcal{V}}\colon
\mathbf{\Sigma}^{\infty-l}_+ \text{\v{C}}_{\mathcal{V}}
\longrightarrow 
\mathbf{\Sigma}^{\infty-l}_+ \underline{V}\in T\check{C}
\end{equation}
as, since any external smash product in which at least one factor is a parametrised zero spectrum is itself a parametrised zero spectrum, all other morphisms in $\mathcal{I}^\mathrm{proj}\,\square\, T\check{C}$ lie the image of the left Quillen functor 
$
0_- \colon L_{\check{C}}\mathrm{Fun}(\mathrm{CartSp}^\mathrm{op},\mathrm{sSet})\to L_{T\check{C}}\mathrm{Fun}(\mathrm{CartSp}^\mathrm{op},\mathrm{Sp}^\Sigma_\mathrm{sSet})$.
For any retractive space $(X,Y)\in \mathrm{sSet}_{\dslash \mathrm{sSet}}$ (necessarily cofibrant), by  Lemma \ref{lem:SmashVSSuspend} there are natural isomorphisms of presheaves
\begin{align*}
\big(\underline{U}\otimes \mathbf{\Sigma}^{\infty-k}_X Y\big)
\esmash
\big(\mathbf{\Sigma}_+^{\infty-l}
\text{\v{C}}_\mathcal{V}\big)
&\cong
\mathbf{\Sigma}^{\infty-(k+l)}_{+} \text{\v{C}}_{U\times\mathcal{V}} \;\esmash \;\mathbf{\Sigma}^\infty_X Y
\\
\big(\underline{U}\otimes \mathbf{\Sigma}^{\infty-k}_X Y\big)
\esmash
\big(\mathbf{\Sigma}_+^{\infty-l}
\underline{V}\big)
&\cong
\mathbf{\Sigma}^{\infty-(k+l)}_{+} \underline{U\times V} \;\esmash \;\mathbf{\Sigma}^\infty_X Y\,,
\end{align*}
where $\mathbf{\Sigma}^\infty_X Y$ on the right-hand side is a constant presheaf.
Note that any smooth parametrised spectrum $(\mathcal{X},\mathcal{A})$, the $\mathrm{Sp}^\Sigma_\mathrm{sSet}$-valued presheaf given by taking objectwise internal homs
\[
F_\vartriangle\{\mathbf{\Sigma}^\infty_X Y, (\mathcal{X},\mathcal{A})\}\colon U\longmapsto F_\vartriangle \big\{\mathbf{\Sigma}^{\infty}_X Y, (\mathcal{X}(U), \mathcal{A}(U)) \big\}
\]
is also a smooth parametrised spectrum. 
From this, it follows by an adjointness argument that 
\[
\mathbf{\Sigma}^{\infty-(k+l)}_+ \text{\v{c}}_{U\times\mathcal{V}}\esmash \mathbf{\Sigma}^\infty_X Y
\colon
\mathbf{\Sigma}^{\infty-(k+l)}_{+} \text{\v{C}}_{U\times\mathcal{V}} \;\esmash \;\mathbf{\Sigma}^\infty_X Y
\longrightarrow
\mathbf{\Sigma}^{\infty-(k+l)}_{+} \underline{U\times V} \;\esmash \;\mathbf{\Sigma}^\infty_X Y
\]
 is a $T\check{C}$-local equivalence.
The pushout-product of the maps $i$ and $\check{c}$ of \eqref{eqn:THESmashMaps} is naturally isomorphic to the map denoted $\varrho$ in the diagram
\[
\begin{tikzcd}[column sep= small]
    \mathbf{\Sigma}^{\infty-(k+l)}_{+} \text{\v{C}}_{U\times\mathcal{V}} \;\esmash \;\mathbf{\Sigma}^\infty_{\Delta^n}\partial\Delta^n_{+\Delta^n}
    \ar[r]\ar[d, "\sim"'] 
    & 
    \mathbf{\Sigma}^{\infty-(k+l)}_{+} \text{\v{C}}_{U\times\mathcal{V}} \;\esmash \;\mathbf{\Sigma}^\infty_{\Delta^n}\Delta^n_{+\Delta^n}
    \ar[dr, bend left=15, "\sim"]
    \ar[d,"\sim"] &
    \\
    \mathbf{\Sigma}^{\infty-(k+l)}_{+} \underline{U\times V} \;\esmash \;\mathbf{\Sigma}^\infty_{\Delta^n}\partial\Delta^n_{+\Delta^n}
    \ar[r]
    &
    P
    \ar[r, "\varrho"]
    &
    \mathbf{\Sigma}^{\infty-(k+l)}_{+} \underline{U\times V} \;\esmash \;\mathbf{\Sigma}^\infty_{\Delta^n}\Delta^n_{+\Delta^n}\,,
\end{tikzcd}
\]
where $P$ is the pushout.
The outer vertical arrows are $T\check{C}$-local equivalences by the above argument, whereas the middle vertical arrow is a $T\check{C}$-local equivalence by left properness, using that the top horizontal arrow is a cofibration.
By the $2$-out-of-$3$ property $\varrho$, and hence $i\,\square\,\check{c}$, is a $T\check{C}$-local equivalence.

We conclude that $p_\Sigma$ is a symmetric monoidal Quillen functor, so that applying the operadic nerve $N^\otimes$ yields a symmetric monoidal functor $p^\otimes \colon T\mathbf{H}^\otimes \to \mathbf{H}^\times$ \cite[Proposition 4.1.3.10]{lurie_higher_2017}. 
Note that the underlying functor of $p^\otimes$ is indeed equivalent to the smooth tangent $\infty$-bundle $p\colon T\mathbf{H}\to \mathbf{H}$ by (the symmetric spectrum version of) Theorem \ref{thm:Tangent}.
\end{proof}

\begin{corollary}
There is a commuting diagram of symmetric monoidal functors 
\[
\begin{tikzcd}
  T\mathcal{S}
  \ar[d, "p^\otimes"]
  \ar[r,"T\mathrm{Disc}^\otimes"] &
  T\mathbf{H}^\otimes
  \ar[d, "p^\otimes"]
  \\
  \mathcal{S}^\times
  \ar[r, "\mathrm{Disc}^\times"]
  &
  \mathbf{H}^\times\,.
\end{tikzcd}
\]
\end{corollary}
\begin{proof}
There is a commutative diagram of symmetric monoidal simplicial (left and right) Quillen functors
\[
\begin{tikzcd}
  \mathrm{Sp}^\Sigma_\mathrm{sSet}
  \ar[r, "\mathrm{const}"]
  \ar[d, "p_\Sigma"]
  &
  L_{T\check{C}} \mathrm{Fun}(\mathrm{CartSp}^\mathrm{op},\mathrm{Sp}^\Sigma_\mathrm{sSet})
  \ar[d, "p_\Sigma"]
  \\
  \mathrm{sSet}
  \ar[r, "\mathrm{const}"]
  &
  L_{\check{C}}\mathrm{Fun}(\mathrm{CartSp}^\mathrm{op}, \mathrm{sSet})\,,
\end{tikzcd}
\]
in which both horizontal functors compute constant presheaves.
The result is obtained by applying the operadic nerve $N^\otimes$.
\end{proof}

\begin{remark}
\label{rem:TwistDiffPairings}
The external smash product of smooth parametrised spectra enables us to define pairings on twisted differential cohomology groups.
This is easiest to see in the $\infty$-categorical setting; for an $\mathbb{E}_n$-algebra object $(\mathcal{X},\mathcal{A})\in \mathrm{Alg}_{\mathbb{E}_n}(T\mathbf{H}^\otimes)$ with multiplication map
\[
(\mu, \nu) \colon (\mathcal{X},\mathcal{A})\esmash (\mathcal{X},\mathcal{A})\longrightarrow (\mathcal{X},\mathcal{A})\,,
\]
the $\infty$-stack of twists $\mathcal{X}$ is an $\mathbb{E}_n$-algebra object of $\mathbf{H}$ with multiplication $\mu$.
For a manifold $M$, the product of twists $\tau, \xi\in \mathcal{X}(M) \cong \mathbf{H}(\underline{M},\mathcal{X})$ is then the twist 
\[
\mu(\tau,\xi)\colon
\ast
\xrightarrow{\;\;(\tau, \xi)\;\;}
\mathbf{H}(\underline{M},\mathcal{X})\times
\mathbf{H}(\underline{M},\mathcal{X}) 
\cong \mathbf{H}(\underline{M},\mathcal{X}\times \mathcal{X})
\xrightarrow{\;\;\mathbf{H}(\underline{M},\mu)\;\;}
\mathbf{H}(\underline{M}, \mathcal{X})\,.
\]
Taking fibres at $\tau$, $\xi$ and $\mu(\tau, \xi)$, in view of Lemma \ref{lem:ESmash=FSmash}
the multiplication map $\nu$ induces a map of spectra
$
\tau^\ast \mathcal{A}(M) \wedge \xi^\ast \mathcal{A}(M) \to \mu(\tau,\xi)^\ast \mathcal{A}(M)
$
.
Passing to stable homotopy groups then gives a pairing on twisted differential cohomology groups:
\[
\begin{tikzcd}[row sep =small]
\mathcal{A}^{\tau +k}(M)\otimes \mathcal{A}^{\xi+l}(M)
\ar[rr]
  \ar[d, equal]
  &&
  \mathcal{A}^{\mu(\tau,\xi)+k+l}(M)
  \ar[d, equal]
  \\
  \pi^\mathrm{st}_{-k}\tau^\ast \mathcal{A}(M) \otimes \pi^\mathrm{st}_{-l}\xi^\ast \mathcal{A}(M)
  \ar[r]
  &
  \pi^\mathrm{st}_{-(k+l)}\big(\tau^\ast \mathcal{A}(M) \wedge \xi^\ast \mathcal{A}(M)\big)
  \ar[r] 
  &  
  \pi^\mathrm{st}_{-(k+l)}\mu(\tau,\xi)^\ast \mathcal{A}(M)
\end{tikzcd}
\]
The smooth spectral $\widehat{R}$-line bundles of Example \ref{ex:SmoothSpecLineBun} are $\mathbb{E}_\infty$-algebra objects of $T\mathbf{H}^\otimes$. The induced pairings on twisted differential $\widehat{R}$-cohomology have been studied in \cite{bunke_twisted_2014}.
\end{remark}

\begin{remark}
The model categories of \cite{hebestreit_multiplicative_2019} can be used to give an alternative equivalent construction of a symmetric monoidal model category modelling the smooth tangent $\infty$-topos $T\mathbf{H}^\otimes$.
Working with this latter model ought to allow for convenient constructions of smooth parametrised $\mathbb{E}_\infty$-ring spectra encoding commutative products in twisted differential cohomology.
We intend to address this question in future work.
\end{remark}



\end{document}